\documentclass{amsart}
\usepackage{amssymb}
\usepackage{amsmath}
\usepackage{amsfonts}

\newtheorem{theorem}{Theorem}
\theoremstyle{plain}

\newtheorem{corollary}{Corollary}

\newtheorem{definition}{Definition}

\newtheorem{lemma}{Lemma}

\newtheorem{proposition}{Proposition}
\newtheorem{remark}{Remark}

\numberwithin{equation}{section}
\input{tcilatex}

\begin{document}
\title[Homogeneous G-algebras]{ Homogeneous G-algebras II}
\author{Roberto Mart\'{\i}nez-Villa}
\address[R. Mart\'{\i}nez-Villa]{Centro de Ciencias Matem\'{a}ticas UNAM\\
Km 8 Antigua Carretera a P\'{a}tzcuaro, Morelia Mich. M\'{e}xico}
\email[R. Mart\'{\i}nez-Villa]{mvilla@matmor.unam.mx}
\urladdr{http://www.matmor.unam.mx}
\thanks{}
\author{Jeronimo Mondrag\'{o}n}
\curraddr[J. Mondrag\'{o}n]{Centro de Ciencias Matem\'{a}ticas UNAM\\
Km 8 Antigua Carretera a Patzcuaro, Morelia Mich, M\'{e}xico}
\email[J. Mondrag\'{o}n]{jeronimo@matmor.unam.mx}
\urladdr{http://www.matmor.unam.mx}
\thanks{}
\date{January 16, 2014}
\dedicatory{}
\thanks{}

\begin{abstract}
G-algebras, or Groebner bases algebras, were considered by Levandovsky in,
[Le], these algebras include very important families of algebras, like the
Weyl algebras [Co], [Mi$\ell $] and the universal enveloping algebra of a
finite dimensional Lie algebra. These algebras are not graded in a natural
way, but they are filtered. Going to the associated graded algebra is the
standard way to study their properties [Mi$\ell $]. In a previous paper
[MMo1] we studied an homogeneous version of the G-algebras, which results
naturally graded, and through a des homogenization property, we obtain an
ordinary G-algebra. The main results in our first paper were the following:

Homogeneous G-algebras are Koszul of finite global dimension, Artin Schelter
regular, noetherian and have a Poincare Birkoff basis. We give the structure
of their Yoneda algebras by generators and relations.

In the second part of the paper we consider the quantum polynomial ring and
prove, for the finitely generated graded modules, cohomology formulas
analogous to those we have in the usual polynomial ring.

In the third part we use the results of the previous part to study the
cohomology of the homogeneous G-algebras.

Part four was dedicated to the study of the relations between an homogeneous
G-algebra $\ B$ and its des homogenization $B$/(Z-1)$B=A$, Applications to
the study of finite dimensional Lie algebras and Weyl algebras were given.

In the last sections of the paper we study the relations, at the level of
modules, among the homogenized algebra $B$ , its graded localization $B_{z}$
, the des homogenization $A$=$B$/(Z-1)$B$ and the Yoneda algebra $B^{!}$.

In this paper we investigate the relations among all these algebras at the
level of derived categories.
\end{abstract}

\maketitle

\section{Some results on the homogeneous G- algebras}

Let $\Bbbk $ be a field of zero characteristic. A G- algebra $A_{n}$ has the
following description by generators and relations: $A_{n}$=$\Bbbk <$X$_{1}$,X%
$_{2}$,...X$_{n}>$/$<$X$_{i}$X$_{j}$- b$_{ij}$X$_{j}$X$_{i}$+$\overset{n}{%
\underset{k=1}{\sum }}$c$_{ij}^{k}$X$_{k}$\linebreak +d$_{ij}\mid $i$>$j, b$%
_{ij}\in \Bbbk $-\{0\}$>$, where $\Bbbk <$X$_{1}$,X$_{2}$,...X$_{n}>$ is the
free algebra in n generators.

Examples of these algebras are the following:

$\bigskip $i) The Weyl algebras $D_{n}$=$\Bbbk <$X$_{1}$,X$_{2}$,...X$_{n}$,$%
\partial _{1}$,$\partial _{2}$,...$\partial _{n}>$/I , with I= $<$[X$_{i}$,$%
\partial _{j}$]-$\delta _{ij}$,\linebreak \lbrack X$_{i}$,X$_{j}$],[$%
\partial _{i}$,$\partial _{j}$] $>$, and $\delta _{ij}$ Kronecker%
\'{}%
s delta.

ii) Let $\mathfrak{g}$ a finite dimensional Lie algebra with basis $\mu _{1}$%
,$\mu _{2}$,...$\mu _{n}$ and [$\mu _{i}$, $\mu _{j}$]=\linebreak $\sum
c_{i,j}^{k}\mu _{k}$. The Universal enveloping algebra is the algebra $U$($%
\mathfrak{g}$)=$\Bbbk <$X$_{1}$,X$_{2}$,...X$_{n}$,$>$/I, with the ideal I=$%
< $X$_{i}$X$_{j}$-X$_{j}$X$_{i}$- $\overset{n}{\underset{k=1}{\sum }}$c$%
_{ij}^{k}$X$_{k}>$.

iii) Quantized versions of i), ii).

\bigskip We know that G-algebras are notherian, they have a Poincare-Birkoff
basis, which induces a natural filtration, the associated graded algebra is
the quantum polynomial ring.

In the particular example of the Weyl algebras, we know that in addition
they are simple algebras.

It is of interest to study both, the category of finitely generated $A_{n}$%
-modules, and the derived category D$^{b}$(mod$_{A_{n}}$). This is not a
simple task since for examples i),ii) we do not in general know even the
irreducible $A_{n}$-modules, so what we do instead is to associate to $A_{n}$
the graded algebra $B_{n}$ defined by generators and relations as $B_{n}$=$%
\Bbbk <$X$_{1}$,X$_{2}$,...X$_{n}$,Z$>$/$<$X$_{i}$X$_{j}$- b$_{ij}$X$_{j}$X$%
_{i}$+$\overset{n}{\underset{k=1}{\sum }}$c$_{ij}^{k}$ZX$_{k}$+d$_{ij}$Z$%
^{2}\mid $i$>$j, b$_{ij}\in \Bbbk $-\{0\}, X$_{i}$Z-ZX$_{i}>$, where $\Bbbk
< $X$_{1}$,X$_{2}$,...X$_{n}$,Z$>$ is the free algebra in n generators. The
algebra $B_{n}$, thus defined will be called the homogeniztion of $A_{n}$
and we recover $A_{n}$ by the deshomogenization $A_{n}$ $\cong B_{n}$/ (Z-1)$%
B_{n}$. We call algebras like $B_{n}$, homogeneous G-algebras.

\bigskip\ Homogeneous G-algebras are much easier to study than G- algebras,
since, as we proved in our first paper, they are Koszul of finite global
dimension and Artin Schelter regular [AS], by Koszul theory [GM1],[GM2],[M]
,[Sm] its Yoneda algebra $B_{n}^{!}$ is a finite dimensional selfinjective
algebra, and there is an equivalence of triangulated categories D$^{b}$(gr$%
_{B_{n}^{!}}$)$\cong $D$^{b}$(gr$_{B_{n}}$) and a duality between the
categories of left $B_{n}$ and $B_{n}^{!}$ Koszul modules, respectively, $%
F:K_{B_{n}}\rightarrow K_{B_{n}^{!}}$. In order to understand the $A_{n}$
-modules via the finite dimensional algebra $B_{n}^{!}$, we must study the
relations between the categories mod$_{A_{n}}$ and gr$_{B_{n}^{!}}$. In
[MMo1] we concentrated in the structure of the algebras $B_{n}$ and $%
B_{n}^{!}$ and studied the relations among the categories mod$_{A}$, gr$%
_{B_{n}}$ and gr$_{^{B_{n}^{!}}}$. Here we study the relations among the
algebras $A_{n}$, $B_{n}$ and $B_{n}^{!}$ at the level of derived categories.

For the benefit of the reader we start recalling some of the results in our
first paper [MMo1]. The homogeneous G-algebra $B_{n}$ has a Poincare Birkoff
basis, $B_{n}$ is two sided noetherian of both, global and Gelfand Kirillov
dimension n+1. The polynomial algebra $\Bbbk $[Z] is contained in the center
of $B_{n}$, and $B_{n}$ is a finitely generated $\Bbbk $[Z]-algebra. Denote
by $C_{n}$=$\Bbbk _{q}$[X$_{1}$,X$_{2}$,...X$_{n}$] the quantum polynomial
algebra $C_{n}$ is a Koszul algebra with Yoneda algebra the quantized
exterior algebra $C_{n}^{!}$=$\Bbbk _{q}$[X$_{1}$,X$_{2}$,...X$_{n}$]/$<$X$%
_{i}^{2}>$, these algebras are related to $B_{n}$ and $B_{n}^{!}$ as
follows: $C_{n}$ is isomorphic to the quotient $B_{n}$/Z$B_{n}$ and $%
C_{n}^{!}$ is a subalgebra of $B_{n}^{!}$. Moreover, there is a left (right) 
$C_{n}^{!}$-module decomposition $B_{n}^{!}$=$C_{n}^{!}\oplus $ $C_{n}^{!}$Z
($B_{n}^{!}$=$C_{n}^{!}\oplus $Z$C_{n}^{!}$).

As we mentioned above, the relation of $B_{n}$ with the G-algebra $A_{n}$ is
given by the isomorphism $B_{n}$/(Z-1)$B_{n}\cong A_{n}$.

Since $Z$ is an element in the center of $B_{n}$, we take the graded
localization ($B_{n}$)$_{Z}$. The algebra ($B_{n}$)$_{Z}$ is a strongly
graded $%
\mathbb{Z}
$-graded algebra which has in degree zero an algebra (($B_{n}$)$_{Z}$)$_{0}$
isomorphic to ($B_{n}$)$_{Z}$/(Z-1)($B_{n}$)$_{Z}$, which is in turn
isomorphic to $B_{n}$/(Z-1)$B_{n}\cong A_{n}$. By Dade's theorem, [Da] there
is an exact equivalence of categories Gr$_{\text{(}B_{n}\text{)z}}\cong $Mod$%
_{A_{n}}$, which induces by restriction an exact equivalence of categories
between the corresponding categories of finitely generated modules gr$_{%
\text{(}B_{n}\text{)z}}\cong $mod$_{A_{n}}$. In view of the previous
statements, to understand the module structure over a G- algebra $A_{n}$ it
is enough to study the relations between gr$_{B_{n}}$ and the finitely
generated ($B_{n}$)z-modules, gr$_{\text{(}B_{n}\text{)z}}$, but this is
just the localization process, which is rather known, then to use Koszul
duality to relate gr$_{B_{n}}$ and gr$_{B_{n}^{!}}$. \linebreak

\bigskip\ We recall some basic facts about selfinjective finite dimensional
algebras that will be needed, in the particular situation we are
considering. For more details we refer to the paper by Yamagata [Y] .

Let $\Lambda $ be a finite dimensional selfinjective algebra over a field $%
\Bbbk $. Denote by D($\Lambda _{\Lambda })$=Hom$_{\Bbbk }$($\Lambda $,$\Bbbk 
$) the standard bimodule. There is an isomorphism of left $\Lambda $-modules 
$\varphi $: $\Lambda \rightarrow $D($\Lambda $), which induces by adjunction
a map $\beta ^{\prime }$:$\Lambda \otimes _{\Lambda }\Lambda \rightarrow
\Bbbk .$ By definition, $\beta ^{\prime }$(a$\otimes $b)=$\varphi $(b)(a)$.$%
The composition: $\Lambda \times \Lambda \overset{p}{\rightarrow }\Lambda
\otimes _{\Lambda }\overset{\beta ^{\prime }}{\Lambda \rightarrow }\Bbbk $, $%
\beta =\beta ^{\prime }p$ is a non degenerated $\Lambda $-bilinear form, and 
$\Lambda \otimes _{\Lambda }\Lambda $ is the cokernel of the map $\Lambda
\otimes _{\Bbbk }\Lambda \otimes _{\Bbbk }\Lambda \rightarrow \Lambda
\otimes _{\Bbbk }$ $\Lambda $ given by a$\otimes $b$\otimes $c$\rightarrow $%
ab$\otimes $c-a$\otimes $bc. Let $\pi $:$\Lambda \otimes _{\Bbbk }\Lambda
\rightarrow \Lambda \otimes _{\Lambda }\Lambda $ be the cokernel map.

The map $\beta ^{\prime }$ induces a map $\overline{\beta }$:$\Lambda
\otimes _{\Bbbk }\Lambda \rightarrow \Bbbk $ by $\overline{\beta }$(x$%
\otimes $y)=$\beta ^{\prime }$(y$\otimes $x). Hence $\overline{\beta }$ is
also non degenerated. In consequence, there is a $\Bbbk $-linear isomorphism 
$\psi $:$\Lambda \rightarrow $D($\Lambda $) given by: $\psi $(a$_{1}$)(a$%
_{2} $)=$\overline{\beta }$(a$_{2}\otimes $a$_{1}$)=$\beta $(a$_{1}\otimes $a%
$_{2} $)=$\varphi $(a$_{2}$)(a$_{1}$).

Set $\sigma $=$\psi ^{-1}\varphi $. There is a chain of equalities:

$\beta $($\sigma $(y),x)=$\beta $($\psi ^{-1}\varphi $(y),x)=$\psi \psi
^{-1}\varphi $(y)(x)=$\varphi $(y)(x)=$\overline{\beta }$(y$\otimes $x).

The map $\sigma $:$\Lambda \rightarrow \Lambda $ is an isomorphism of $\Bbbk 
$-algebras.

$\beta $($\sigma $(y$_{1}$y$_{2}$)$\otimes $z)=$\beta $(z$\otimes $y$_{1}$y$%
_{2}$)=$\beta $(zy$_{1}\otimes $y$_{2}$)=$\overline{\beta }$(y$_{2}$,zy$_{1}$%
)=$\beta $($\sigma $(y$_{2}$),zy$_{1}$)=$\beta $($\sigma $(y$_{2}$)z,y$_{1}$%
)\linebreak =$\overline{\beta }$(y$_{1}$,$\sigma $(y$_{2}$)z)=$\beta $($%
\sigma $(y$_{1}$),$\sigma $(y$_{2}$)z)=$\beta $($\sigma $(y$_{1}$)$\sigma $(y%
$_{2}$),z).

Since z is arbitrary and $\beta $ non degenerated $\sigma $(y$_{1}$y$_{2}$)=$%
\sigma $(y$_{1}$)$\sigma $(y$_{2}$).

Let D($\Lambda $)$_{\sigma }$ ($_{\sigma ^{-1}}$D($\Lambda $)) be the $%
\Lambda $-$\Lambda $ bimodule with right (left) multiplication shifted by $%
\sigma $ ($\sigma ^{-1}$). Then $\varphi $:$\Lambda \rightarrow $D($\Lambda $%
)$_{\sigma }$ and $\psi $:$\Lambda \rightarrow _{\sigma ^{-1}}$D($\Lambda $)
are isomorphisms of $\Lambda $- $\Lambda $ bimodules.

$\varphi $(xa)(y)=$\beta $(y,xa)=$\beta $($\sigma $(x)$\sigma $(a),y)=$\beta 
$($\sigma $(x),$\sigma $(a)y)=$\beta $($\sigma $(a)y,x)=

$\varphi $(x)($\sigma $(a)y)=$\varphi $(x)$\sigma $(a)(y), for all y.
Therefore $\varphi $(xa)=$\varphi $(x)$\sigma $(a). As claimed.

In a similar way, $\psi $(xb)(y)=$\varphi $(y)(xb)=$\beta $(xb,y)=$\beta $%
(x,by)=$\varphi $(by)(x)=$\psi $(x)b(y).

Since y is arbitrary, $\psi $(xb)=$\psi $(x)b.

In the other hand, $\varphi $(y)(x)=($\psi $(x)(y)=$\beta $(x,y)=$\beta $($%
\sigma \sigma ^{-1}$(x),y)=$\beta $(y,$\sigma ^{-1}$(x))\newline
=$\varphi $($\sigma ^{-1}$(x))(y). Hence, $\psi $(x)=$\varphi $($\sigma
^{-1} $(x)).

It follows: $\psi $(bx)=$\varphi $($\sigma ^{-1}$(b)$\sigma ^{-1}$(x))=$%
\sigma ^{-1}$(b)$\varphi $($\sigma ^{-1}$(x))=$\sigma ^{-1}$(b)$\psi $(x).

Let $M$ be a finitely generated $\Lambda $-module. Since $_{\sigma }\Lambda
\cong $D($\Lambda $) as bimodule, there are natural isomorphisms:

D($M^{\ast }$)=Hom$_{\Lambda }$(Hom$_{\Lambda }$($M$,$\Lambda $), D($\Lambda 
$))=Hom$_{\Lambda }$(Hom$_{\Lambda }$($M$,$\Lambda $), $_{\sigma }\Lambda $)$%
\cong \sigma M^{\ast \ast }\cong \sigma M$, where $\sigma M$=$M$ as abelian
group and multiplication by $\Lambda $ shifted by $\sigma $.

We look now to the case $\Lambda $ a positively graded selfinjective $\Bbbk $%
-algebra and \linebreak $\varphi $: $\Lambda \rightarrow $D($\Lambda $)[n]
an isomorphism of graded $\Lambda $-modules. Let a, x be elements of $%
\Lambda $ of degrees k and j, respectively. Then $\varphi $(xa)=$\varphi $(x)%
$\sigma $(a) is an homogeneous element of degree k+j and $\sigma $(a)=$%
\Sigma \sigma $(a)$_{i}$, with $\sigma $(a)$_{i}$ homogeneous elements of
degree i. Hence $\varphi $(x)$\sigma $(a)$_{i}$=0 for all i$\neq $j. But $%
\varphi $(x)$\sigma $(a)$_{i}$(1)=$\varphi $(x)($\sigma $(a)$_{i}$)=$\beta $(%
$\sigma $(a)$_{i}$,x)=0 for all homogeneous elements x. Hence $\beta $($%
\sigma $(a)$_{i}$,$\Lambda $)=0 and $\sigma $(a)$_{i}$=0 for i$\neq $j.

We have proved $\sigma $ is an isomorphism of graded $\Bbbk $-algebras,
which induces isomorphisms of graded $\Lambda $-$\Lambda $ bimodules: $%
\varphi $:$\Lambda \rightarrow $D($\Lambda $)$_{\sigma }$[n], and $\psi $:$%
\Lambda \rightarrow _{\sigma ^{-1}}$D($\Lambda $)[n]$.$

We only need to check $\psi $ is an isomorphism of graded $\Bbbk $-vector
spaces.

Being $\varphi $ a graded map, $\varphi $=\{$\varphi _{\ell }$\} and $%
\varphi _{\ell }$:$\Lambda _{\ell }\rightarrow $Hom$_{\Bbbk }$($\Lambda
_{n-\ell }$,$\Bbbk $)[n] isomorphisms, inducing maps $\beta _{\ell }^{\prime
}$:$\Lambda _{n-\ell }\otimes \Lambda _{\ell }\rightarrow \Bbbk $ and $%
\overline{\beta }_{\ell }$:$\Lambda _{\ell }\otimes _{\Bbbk }\Lambda
_{n-\ell }\rightarrow \Bbbk $, with $\overline{\beta }_{\ell }$(x$\otimes $%
y)=$\beta _{\ell }^{\prime }$(y$\otimes $x)$.$\linebreak Each $\overline{%
\beta }_{\ell }$induces maps $\psi _{\ell }\Lambda $:$_{n-\ell }\rightarrow $%
D($\Lambda _{\ell }$)[n] such that $\psi $=\{$\psi _{\ell }$\} becomes a
graded map.

In case $M$ is a finitely generated graded left $\Lambda $-module there is a
chain of isomorphisms of graded $\Lambda $-modules:

D($M^{\ast }$)=Hom$_{\Lambda }$(Hom$_{A}$($M$,$\Lambda $),D($\Lambda $))=Hom$%
_{\Lambda }$(Hom$_{\Lambda }$($M$,$\Lambda $),$_{\sigma }\Lambda $[-n])$%
\cong \sigma M^{\ast \ast }$[-n]$\cong $\linebreak $\sigma M$[-n].

Assume now $\Lambda $ is Koszul sefinjective with Nakayama automorphism $%
\sigma $ and Yoneda algebra $\Gamma $. It was remarked in [M] that under
these conditions there is natural action of $\sigma $ as a graded
automorphism of $\Gamma $, we will recall now this construction.

Let x be an element of Ext$_{\Lambda }^{n}$($\Lambda _{0}\Lambda $,$_{0}$)=$%
\underset{i,j}{\oplus }$Ext$_{\Lambda }^{n}$($S_{i}$,$S_{j}$), x=(x$_{i,j}$)
with x$_{i,j}$ the extension:

\begin{center}
x$_{i,j}$:$0\rightarrow S_{j}$[-n]$\rightarrow E_{1}\rightarrow
E_{2}\rightarrow ...\rightarrow E_{n}\rightarrow S_{i}\rightarrow 0.$
\end{center}

Then $\sigma $x$_{i,j}$ is the extension:

\begin{center}
$\sigma $x$_{i,j}$:$0\rightarrow \sigma S_{j}$[-n]$\rightarrow \sigma
E_{1}\rightarrow \sigma E_{2}\rightarrow ...\rightarrow \sigma
E_{n}\rightarrow \sigma S_{i}\rightarrow 0.$
\end{center}

Since $\sigma $ is a permutation of the graded simple, $\sigma $x=($\sigma $x%
$_{i,j}$) is an element of Ext$_{\Lambda }^{n}$($\Lambda _{0}$,$\Lambda
_{0}) $ and $\sigma :$Ext$_{\Lambda }^{n}$($\Lambda _{0}$,$\Lambda _{0}$)$%
\rightarrow $Ext$_{\Lambda }^{n}$($\Lambda _{0}$,$\Lambda _{0}$) is an
isomorphism of $\Bbbk $-vector spaces which extends to a graded automorphism
of $\Gamma $=$\underset{n\geq 0}{\oplus }$Ext$_{\Lambda }^{n}$($\Lambda _{0}$%
,$\Lambda _{0}$).

Let $M$ be a finitely generated (graded) $A$-module and x=(x$_{j}$)$\in $Ext$%
_{{}}^{n}$($M$,$\Lambda _{0}$)\linebreak =$\underset{j\geq 0}{\oplus }$Ext$%
_{\Lambda }^{n}$($M$,$S_{j}$).

\begin{center}
\ x$_{j}$: $0\rightarrow S_{j}$[-n]$\rightarrow E_{1}\rightarrow
E_{2}\rightarrow ...\rightarrow E_{n}\rightarrow M\rightarrow 0.$
\end{center}

Then $\sigma $x=($\sigma $x$_{j}$) with

\begin{center}
$\sigma $x$_{j}:$ $0\rightarrow \sigma S_{j}$[-n]$\rightarrow \sigma
E_{1}\rightarrow \sigma E_{2}\rightarrow ...\rightarrow \sigma
E_{n}\rightarrow \sigma M\rightarrow 0.$
\end{center}

In this case there is an isomorphism of $\Bbbk $-vector spaces: \newline
$\sigma $: Ext$_{\Lambda }^{n}$($M$,$\Lambda _{0}$)$\rightarrow $Ext$%
_{\Lambda }^{n}$($\sigma M$,$\Lambda _{0}$), which induces a graded
isomorphism: \newline
$\sigma $:$\underset{n\geq 0}{\oplus }$Ext$_{\Lambda }^{n}$($M$,$\Lambda
_{0} $)$\rightarrow \underset{n\geq 0}{\oplus }$Ext$_{\Lambda }^{n}$($\sigma
M$,$\Lambda _{0}$).

We also call the Nakayama automorphism to the automorphism $\sigma $ of $%
\Gamma $.

We now look in more detail to the Nakayama automorphism $\sigma $ of the
shriek algebra $B_{n}^{!}$ of the homogeneous algebra $B_{n}$.

The graded ring $B_{n}^{!}$ has a decomposition: ($B_{n}^{!}$)$_{0}$=($%
C_{n}^{!}$)$_{0}$, ($B_{n}^{!}$)$_{1}$=($C_{n}^{!}$)$_{1}\oplus $Z($%
C_{n}^{!} $)$_{0}$, ($B_{n}^{!}$)$_{2}$=($C_{n}^{!}$)$_{2}\oplus $Z($%
C_{n}^{!}$)$_{1}$, .... ($B_{n}^{!}$)$_{i}$=($C_{n}^{!}$)$_{i}\oplus $Z($%
C_{n}^{!}$)$_{i-1}$,... ($B_{n}^{!}$)$_{n}$=($C_{n}^{!}$)$_{n}\oplus $Z($%
C_{n}^{!}$)$_{n-1}$, ($B_{n}^{!}$)$_{n+1}$=Z($C_{n}^{!}$)$_{n}$

The algebra $C_{n}^{!}$ is the quantized exterior algebra in n variables,
hence

dim$_{\Bbbk }$($C_{n}^{!}$)$_{j}$=$\left( 
\begin{array}{c}
\text{n} \\ 
\text{j}%
\end{array}%
\right) $=$\left( 
\begin{array}{c}
\text{n} \\ 
\text{n-j}%
\end{array}%
\right) $=dim$_{\Bbbk }$($C_{n}^{!}$)$_{n-j}.$

Since ($B_{n}^{!}$)$_{j}$=($C_{n}^{!}$)$_{j}\oplus $Z($C_{n}^{!}$)$_{j-1}$,
it follows:

dim$_{\Bbbk }$($B_{n}^{!}$)$_{j}$=$\left( 
\begin{array}{c}
\text{n} \\ 
\text{j}%
\end{array}%
\right) $+$\left( 
\begin{array}{c}
\text{n} \\ 
\text{j-1}%
\end{array}%
\right) $=$\left( 
\begin{array}{c}
\text{n} \\ 
\text{n+1-j}%
\end{array}%
\right) $+$\left( 
\begin{array}{c}
\text{n} \\ 
\text{n-j}%
\end{array}%
\right) $=dim$_{\Bbbk }$($B_{n}^{!}$)$_{n+1-j}$.

The graded left module D( ($B_{n}^{!}$) decomposes in homogeneous components:

D( ($B_{n}^{!}$)=D(($B_{n}^{!}$)$_{n+1}$)+D(($B_{n}^{!}$)$_{n}$)+...D( ($%
B_{n}^{!}$)$_{0}$).

Each component ($B_{n}^{!}$)$_{i}$ has as basis paths of length i, either of
the form

X$_{j_{1}}$X$_{j_{2}}$...X$_{j_{i-1}}$Z, or of the form X$_{j_{1}}$X$%
_{j_{2}} $...X$_{j_{i}}$, and D((B$_{n}^{!}$)$_{\text{n+1-i}}$) has as basis
the dual basis of paths of length n+1-i.

The isomorphism of graded left modules: $\varphi $: $B_{n}^{!}\rightarrow $D(%
$B_{n}^{!}$)[-n-1], sends a path of, either of the form $\gamma $=X$_{j_{1}}$%
X$_{j_{2}}$...X$_{j_{i-1}}$Z, or of the form

$\gamma $=X$_{j_{1}}$X$_{j_{2}}$...X$_{j_{i}}$, to the dual basis $f_{\rho 
\text{-}\gamma }$ of the path $\rho $-$\gamma $ of length n+1-i, with $\rho $
the path of maximal length $\rho $=X$_{1}$X$_{2}$...X$_{n}$Z.

Since ($B_{n}^{!}$)$_{i}$=($C_{n}^{!}$)$_{i}\oplus $Z($C_{n}^{!}$)$_{i-1}$,
the isomorphisn $\phi $ restricts to isomorphisms of $\Bbbk $-vector spaces $%
\varphi $:($C_{n}^{!}$)$_{i}\rightarrow $D(Z($C_{n}^{!}$)$_{n-i-1}$) and $%
\varphi $:(Z$C_{n}^{!}$)$_{i-1}\rightarrow $D(($C_{n}^{!}$)$_{n-i}$), hence, 
$\varphi $ induces isomorphisms of $C_{n}^{!}$-modules $\varphi $:($%
C_{n}^{!} $)$\rightarrow $D(Z($C_{n}^{!}$)) and $\varphi $ : (Z$C_{n}^{!}$)$%
\rightarrow $D(($C_{n}^{!}$).

Now the isomorphism $\psi :B_{n}^{!}\rightarrow $D($B_{n}^{!}$)[-n-1] given $%
\psi $(b$_{1}$)(b$_{2}$)=$\varphi $(b$_{2}$)(b$_{1}$) is such that for c$%
_{1}\in $($C_{n}^{!}$)$_{i}$ and b$\in $B$_{n}$,b=$\overset{n+1}{\underset{%
i=0}{\sum }}$b$_{i}$, $\psi $(c$_{1}$)(b)=$\overset{n+1}{\underset{k=0}{\sum 
}}\varphi $(b$_{k}$)(c$_{1}$), since for all k$\neq $i the length of b$_{k}$
is different from n+1-i, then $\varphi $(b$_{k}$)(c$_{1}$)=0, $\psi $(c$_{1}$%
)$\in $D($Z$(C$_{n}^{!}$)$_{\text{n-i-1}}$), hence, $\psi $ induces an
isomorphism of graded $C_{n}^{!}$-modules $\psi $ : ($C_{n}^{!}$)$%
\rightarrow $D($Z$($C_{n}^{!}$)) and in a similar way an isomorphism $\psi $%
:(Z$C_{n}^{!}$)$\rightarrow $D(($C_{n}^{!}$). It follows the Nakayama
automorphism $\sigma $ restricts to an automorphisms of graded rings: $%
\sigma $ : $C_{n}^{!}\rightarrow C_{n}^{!}$ and of $C_{n}^{!}$-bimodules $%
\sigma $ : $ZC_{n}^{!}\rightarrow ZC_{n}^{!}$.

Any automorphism $\sigma $ of a ring $\Lambda $ takes the center to the
center, since $z\in Z(\Lambda )$ implies that for any b$\in \Lambda $, $%
\sigma $(Zb)=$\sigma $(Z)$\sigma $(b)=$\sigma $(bZ)=$\sigma $(b)$\sigma $(Z).

In case $\Lambda $ is the homogeneous Weyl algebra it is not difficult to
see that the center of $\Lambda $ is $\Bbbk \lbrack Z]$ and since $\sigma $%
(Z) is an homogeneous element of degree one in Z($B$)=$\Bbbk $[Z], it is
clear $\sigma $(Z)=kZ with k a non zero element of the field $\Bbbk $%
.\linebreak

\section{The derived Categories D$^{b}$(Qgr$_{B_{n}}$) and D$^{b}$(gr$%
_{(B_{n})_{Z}}$).}

We describe now the main results of the paper: let's consider a homogeneous
G-algebra $B_{n}$, and let Qgr$_{B_{n}}$ be the quotient category of gr$%
_{B_{n}}$, module the objects of finite length, this is: the category of
"tails" consider in [AZ],[G], [P]. Qgr$_{B_{n}}$ is an abelian category and
there is an exact functor $\pi $: gr$_{B_{n}}\rightarrow $Qgr$_{B_{n}}$. It
was proved in [MM],[MS] that there is an exact duality \underline{gr}$%
_{B_{n}^{!}}\cong $D$^{b}$(Qgr$_{B_{n}}$), with D$^{b}$(Qgr$_{B_{n}}$) the
derived category of bounded complex of objects in Qgr$_{B_{n}}$, we must
consider the relations between the categories D$^{b}$(Qgr$_{B_{n}}$) and D$%
^{b}$(gr$_{\text{(}B_{n}\text{)z}}$), which in view of the results of [MMo],
the last one is equivalent as triangulated categories to D$^{b}$(mod$%
_{A_{n}} $ ). We prove that there exists an exact dense functor: D$^{b}$($%
\psi $): D$^{b}$(Qgr$_{B_{n}}$) $\rightarrow $D$^{b}$(gr$_{\text{(}B_{n}%
\text{)z}}$) with kernel $\mathcal{T}$=KerD$^{b}$($\psi $)=\{$\pi $($%
M^{\circ }$)$\mid M^{\circ }\in $D$^{b}$(gr$_{B_{n}}$) for all i, H$^{i}$( $%
M^{\circ }$) is of Z-torsion\}. The main theorem of the section is that
there exists an equivalence of triangulated categories D$^{b}$(Qgr$_{B_{n}}$%
) /$\mathcal{T}\cong $D$^{b}$(gr$_{\text{(}B_{n}\text{)z}}$), where the
category $\mathcal{T}$ is an "epasse" subcategory of D$^{b}$(Qgr$_{B_{n}}$)
and D$^{b}$(Qgr$_{B_{n}}$) /$\mathcal{T}$ is the Verdier quotient.[Mi]

In section three, together with the quotient of categories considered in
section two, we consider a full embedding of a subcategory $\mathcal{F}$ of D%
$^{b}$(Qgr$_{B_{n}}$) in D$^{b}$(gr$_{\text{(}B_{n}\text{)z}}$). Here $%
\mathcal{F}$ is the full subcategory of D$^{b}$(Qgr$_{B_{n}}$) consisting of
the $\mathcal{T}$-local objects [Mi], this is: $\mathcal{F}$=\{ $X^{\circ
}\mid $Hom$_{D^{b}(Qgr_{B_{n}})}$ ($\mathcal{T}$, $X^{\circ }$)=0\}.

Using the duality $\overline{\phi }$: \underline{gr}$_{B_{n}^{!}}\rightarrow 
$D$^{b}$(Qgr$_{B_{n}}$) we obtain a pair of triangulated subcategories $(%
\mathcal{F}^{\prime },\mathcal{T}^{\prime })$ of \underline{gr}$_{B_{n}^{!}}$
such that $\mathcal{F}^{\prime }\rightarrow \mathcal{F}$ and $\mathcal{T}%
^{\prime }\rightarrow \mathcal{T}$ under the duality $\phi $. We obtain the
following characterization of the subcategories $\mathcal{T}^{\prime }$ and $%
\mathcal{F}^{\prime }$ of \underline{gr}$_{B^{!}}$: $\mathcal{T}^{\prime }$
is the smallest triangulated subcategories of \underline{gr}$_{B_{n}^{!}}$
containing the induced modules $M\otimes _{C_{n}^{!}}$ $B_{n}^{!}$ and
closed under the Nakayama automorphism, $\mathcal{T}^{\prime }$ has
Auslander-Reiten triangles and they are of the type $%
\mathbb{Z}
A_{\infty }$. For the category $\mathcal{F}^{\prime }$ we obtain the
following characterization: $\mathcal{F}^{\prime }$ consists of the graded
finitely generated $B_{n}^{!}$-modules whose restriction to $C_{n}^{!}$ is
injective. Furthermore, the category $\mathcal{F}^{\prime }$ is closed under
the Nakayama automorphism, it has Auslander-Reiten triangles and they are of
type $%
\mathbb{Z}
A_{\infty }$.

In order to obtain an equivalence instead of a duality we apply the usual
duality D:\underline{gr}$_{B_{n}^{!}}\rightarrow $\underline{gr}$_{B_{n}^{%
\text{!op}}}$ to obtain subcategories $T$ and $F$ of \underline{gr}$_{B_{n}^{%
\text{!op}}}$ with $T$=D($\mathcal{T}^{\prime }$) and $F$=D($\mathcal{F}%
^{\prime }$) where $T$ is the smallest triangulated subcategory of 
\underline{gr}$_{B_{n}^{\text{!op}}}$ containing the induced modules and $F$
is the full subcategory of \underline{gr}$_{B_{n}^{\text{!op}}}$ consisting
of those modules whose restriction to $C_{n}^{!}$ is projective.

Finally we obtain the main results of the paper: there is an equivalence of
triangulated categories \underline{gr}$_{B_{n}^{!op}}$/ $T\cong $D$^{b}$(mod$%
_{A_{n}}$) and there is a full embedding of triangulated categories $%
F\rightarrow $D$^{b}$(mod$_{A_{n}}$).

We start recalling the construction of the categories of "tails" QGr$%
_{B_{n}} $ and Qgr$_{B_{n}}.$

Let $M$ be a graded $B_{n}$-module, t($M$)=$\underset{L\in J}{\sum }L$ and $%
J $=\{$L\mid $sub module of $M$ dim$_{\Bbbk }L<\infty $\}.

Claim: t($M$/t($M$))=0.

Let $N$ be a finitely generated sub module of $M$ such that $N$+t($M$)/t($M$%
)=\linebreak $N$/$N\cap $t($M$) is finite dimensional over $\Bbbk $. Since $%
B_{n}$ is noetherian $N\cap $t($M$) is a finitely generated submodule of $%
t(M)$, hence of finite dimension over $\Bbbk $. It follows $N$ is finite
dimensional, so $N\subset $t($M$).

Let $N$ be an arbitrary sub module of $M$ with $N$+t($M$)/t($M$) finite
dimensional over $\Bbbk $, then $N$=$\sum N_{i}$, with $N_{i}$ finitely
generated, each $N_{i}$+t($M$)/t($M$) is finite dimensional, therefore $%
N_{i}\subset $t($M$). It follows $N\subset $t($M$) and $t$ is an idempotent
radical.

If we denote by gr$_{B_{n}}$ and gr$_{\text{(}B_{n}\text{)}_{Z}}$ the
categories of finitely generated graded $B_{n}$ and ($B_{n}$)$_{Z}$-modules,
respectively, then the localization functor $Q$ restricts to a functor $Q$:
gr$_{B_{n}}$ $\rightarrow $gr$_{\text{(}B_{n}\text{)}_{Z}}$.

\begin{definition}
Given a $B_{n}$-module $M,$ define the $Z$-torsion of $M$ as: t$_{Z}$($M$%
)=\linebreak \{m$\in $M$\mid $there exists n$>$0 with $Z^{n}m=0$\}. It is
clear t$_{Z}$ is an idempotent radical. We say $M$ is Z-torsion when t$_{Z}$(%
$M$)=$M$ and Z-torsion free if t$_{Z}$($M$)=0.
\end{definition}

The kernel of the natural map $M\longrightarrow M_{Z}$ is t$_{Z}$($M$).

\begin{definition}
We say that a (graded) $B_{n}$-module is torsion if t($M$)=$M$ and torsion
free if t($M$)=0.
\end{definition}

It is clear t($M$) is Z-torsion and t($M$)$\subset $t$_{Z}$($M$). Therefore
if $M$ is torsion then it is Z -torsion and if $M$ is Z-torsion free then it
is torsion free.

The torsion free modules form a Serre (or thick) subcategory of Gr$_{B_{n}}$%
we localize with respect to this subcategory as explained in [Ga],[P].
Denote by QGr$_{B_{n}}$ the quotient category and let $\pi $:Gr$%
_{B_{n}}\rightarrow $QGr$_{B_{n}}$ be the quotient functor, QGr$_{B_{n}}$%
=\linebreak Gr$_{B_{n}}$/Torsion, is an abelian category with enough
injectives and $\pi $ is an exact functor. When taking this quotient we are
inverting the maps of $B_{n}$-graded modules, f: $M\rightarrow N$ such that
Kerf and Cokerf are torsion.

The category QGr$_{B_{n}}$ has the same objects as Gr$_{B_{n}}$ and maps:

Hom$_{\text{QG}r_{B_{n}}}$($\pi $($M$),$\pi $($N$))=$\underrightarrow{\text{%
lim}}$Hom$_{\text{Gr}_{B_{n}}}$($M^{\prime }$,$N$/t($N$)), the limit running
through all the sub modules $M^{\prime }$of $M$ such that $M$/$M^{\prime }$
is torsion.

If $M$ is a finitely generated module, then the limit has a simpler form:

Hom$_{\text{QG}r_{B_{n}}}$($\pi $($M$),$\pi $($N$))=$\underset{\text{k}}{%
\underrightarrow{\text{lim}}}$Hom$_{\text{Gr}_{B_{n}}}$($M_{\geq k}$,$N$/t($%
N $))$.$

In case $N$ is torsion free: Hom$_{\text{QG}r_{B_{n}}}$($\pi $($M$),$\pi $($%
N $))=$\underset{\text{k}}{\underrightarrow{\text{lim}}}$Hom$_{\text{Gr}%
_{B_{n}}}$($M_{\geq k}$,$N$).

The functor $\pi $: Gr$_{B_{n}}\rightarrow $QGr$_{B_{n}}$ has a right
adjoint: $\varpi $: QGr$_{B_{n}}\rightarrow $Gr$_{B_{n}}$such that $\pi
\varpi \cong $1. [P].

If we denote by gr$_{B_{n}}$ the category of finitely generated graded \ $%
B_{n}$-modules, and by Qgr$_{B_{n}}$ the full subcategory of QGr$_{B_{n}}$
consisting of the objects $\pi $($N$) with $N$ finitely generated, then the
functor $\pi $ induces by restriction a functor: $\pi $:gr$%
_{B_{n}}\rightarrow $Qgr$_{B_{n}}$. The kernel of $\pi $ is: Ker$\pi $=\{$%
M\in $ gr$_{B_{n}}\mid \pi $($M$)=0\}=\{$M\in $ gr$_{B_{n}}\mid $t($M$)=$M$%
\}.

In the other hand, the functor

$Q$=($B_{n}$)$_{Z}\underset{B}{\otimes }$-: gr$_{B_{n}}\rightarrow $gr$_{%
\text{(}B_{n}\text{)}_{Z}}$ has kernel \{M$\in $gr$_{B_{n}}\mid $M$_{Z}$%
=0\}=\{M$\in $gr$_{B_{n}}\mid $t$_{Z}$($M$)=$M$\}.

It follows: Ker$\pi \subset $Ker(($B_{n}$)$_{Z}\underset{B}{\otimes }$-).

According to [P] (pag. 173 Cor. 3.11) there exists a unique functor $\psi $
such that the following diagram commutes:

\begin{center}
$%
\begin{array}{cccccc}
\text{gr}_{B_{n}} &  & \overset{\pi }{\rightarrow } &  & \text{Qgr}_{B_{n}}
&  \\ 
\text{(}B_{n}\text{)}_{Z}\underset{B}{\otimes }\text{-} & \searrow &  & 
\swarrow & \psi &  \\ 
&  & \text{gr}_{\text{(}B_{n}\text{)}_{Z}} &  &  & 
\end{array}%
$
\end{center}

This is: $\psi \pi $=($B_{n}$)$_{Z}\underset{B}{\otimes }$-.

\begin{proposition}
The functor $\psi $: Qgr$_{B_{n}}\rightarrow $gr$_{\text{(}B_{n}\text{)}%
_{Z}} $ is exact.
\end{proposition}

\begin{proof}
Let $0\rightarrow \pi $($M$)$\overset{\overset{\wedge }{f}}{\rightarrow }\pi 
$($N$)$\overset{\overset{\wedge }{g}}{\rightarrow }\pi $($L$)$\rightarrow 0$
be an exact sequence in Qgr$_{B_{n}}$. We may assume $M$,$N$,$L$ torsion
free. Then::

$\overset{\wedge }{\text{f}}\in \underset{k}{\underrightarrow{\text{lim}}}$%
Hom$_{\text{Gr}B_{n}}$($M_{\geq k}$,$N$) and $\overset{\wedge }{\text{g}}\in 
\underset{s}{\underrightarrow{\text{lim}}}$Hom$_{\text{Gr}B_{n}}$($N_{\geq
s} $,$L$).

There exist exact sequences: $0\rightarrow M_{\geq k+1}\rightarrow M_{\geq
k}\rightarrow M_{\geq k}/M_{\geq k+1}\rightarrow 0$ which induces an exact
sequence:

$0\rightarrow $Hom$_{\text{Gr}_{B_{n}}}$($M_{\geq k}$/$M_{\geq k+1}$,$N$)$%
\rightarrow Hom_{\text{Gr}_{B_{n}}}$($M_{\geq k}$,$N$)$\rightarrow $Hom$_{%
\text{Gr}_{B_{n}}}$($M_{\geq k+1}$,$N$).

Since we are assuming $N$ is torsion free Hom$_{\text{Gr}_{B_{n}}}$($M_{\geq
k}$/$M_{\geq k+1}$,$N$)=0.

Hence $\overset{\wedge }{\text{f}}\in $Hom$_{\text{Gr}_{B_{n}}}$($\pi $($M$),%
$\pi $($N$))=$\underset{k\geq 0}{\cup }$Hom$_{\text{Gr}_{B_{n}}}$($M_{\geq
k} $,$N$).

The map $\overset{\wedge }{\text{f }}$is represented by f: $M_{\geq
k}\rightarrow N.$ Similarly, $\overset{\wedge }{\text{g}}$ is represented by
a map g: $N_{\geq \ell }\rightarrow L$ and we have a sequence: $M_{\geq
k+\ell }\overset{f}{\rightarrow }N_{\geq \ell }\overset{g}{\rightarrow }L$
with $\overset{\wedge }{\text{(gf)}}$=$\overset{\wedge }{\text{g}}\overset{%
\wedge }{\text{f}}$=0, which implies gf factors through a torsion module,
but $L$ torsion free implies gf=0. Since $M_{\geq k+\ell }$ is torsion free,
f is a monomorphism. If Cokerg is torsion, there exists an s$\geq $0 such
that Cokerg$_{\geq s}$=0. Taking a large enough truncation we obtain a
sequence: $M_{\geq s}\overset{f}{\rightarrow }N_{\geq s}\overset{g}{%
\rightarrow }L_{\geq s}$ with f a monomorphism, g an epimorphism and gf=0.

Consider the exact sequences: $0\rightarrow $ $M_{\geq s}\overset{f^{\prime
\prime }}{\rightarrow }$Kerg$\rightarrow H\rightarrow 0$,

$0\rightarrow $Kerg$\rightarrow N_{\geq s}\rightarrow L_{\geq s}\rightarrow
0.$

Applying $\pi $ we obtain the following isomorphism of exact sequences:

$%
\begin{array}{ccccccc}
0\rightarrow & \pi \text{(}M_{\geq s}\text{)} & \overset{\pi f^{\prime
\prime }}{\rightarrow } & \pi \text{(Kerg)} & \rightarrow & \pi \text{(}H%
\text{)} & \rightarrow 0 \\ 
& \downarrow \cong &  & \downarrow \cong &  &  &  \\ 
0\rightarrow & \pi \text{(}M\text{)} & \overset{\overset{\wedge }{f}}{%
\rightarrow } & \text{Ker}\overset{\wedge }{\text{g}} & \rightarrow & 0 & 
\end{array}%
$

It follows $\pi $($H$)=0 and $H$ is torsion, so there exists an integer t$%
\geq $0 such that $H_{\geq t}$=0. Finally taking a large enough truncation
we get an exact sequence:

$0\rightarrow M_{\geq s}\overset{f}{\rightarrow }N_{\geq s}\overset{g}{%
\rightarrow }L_{\geq s}\rightarrow 0$ such that the following sequences are
isomorphic:

$%
\begin{array}{ccccccc}
0\rightarrow & \pi \text{(}M_{\geq s}\text{)} & \overset{\pi f}{\rightarrow }
& \pi \text{(}N_{\geq s}\text{)} & \overset{\pi g}{\rightarrow } & \pi \text{%
(}L_{\geq s}\text{)} & \rightarrow 0 \\ 
& \downarrow \cong &  & \downarrow \cong &  & \downarrow \cong &  \\ 
0\rightarrow & \pi \text{(}M\text{)} & \overset{\overset{\wedge }{f}}{%
\rightarrow } & \pi \text{(}N\text{)} & \overset{\overset{\wedge }{g}}{%
\rightarrow } & \pi \text{(}L\text{)} & \rightarrow 0%
\end{array}%
$

Applying $\psi $ we have an exact sequence: $0\rightarrow \psi \pi $($M$)$%
\overset{\psi \overset{\wedge }{f}}{\rightarrow }\psi \pi $($N$)$\overset{%
\psi \overset{\wedge }{g}}{\rightarrow }\psi \pi $($L$)$\rightarrow 0$,
which is isomorphic to $0\rightarrow $($M_{\geq s}$)$_{Z}\overset{f_{Z}}{%
\rightarrow }$($N_{\geq s}$)$_{Z}\overset{g_{Z}}{\rightarrow }$($L_{\geq s}$)%
$_{Z}\rightarrow 0.$

We have proved $\psi $ is exact.
\end{proof}

The functor $\psi $ has a derived functor: D($\psi $):D$^{b}$(Qgr$_{B_{n}}$)$%
\rightarrow $D$^{b}$(gr$_{(B_{n})_{Z}}$), we will study next its properties.

Observe Qgr$_{B_{n}}$ does not have neither enough projective nor enough
injective objects.

\begin{lemma}
Let $0\rightarrow N_{1}\overset{d_{1}}{\rightarrow }N_{2}\overset{d_{2}}{%
\rightarrow }...N_{\ell -1}\overset{d_{\ell -1}}{\rightarrow }N_{\ell
}\rightarrow 0$ be a sequence of $B_{n}$- modules and assume the
compositions d$_{i}$d$_{i-1}$ factors through a module of Z-torsion. Then
there exists a complex:

$0\rightarrow N_{1}\overset{\overset{\wedge }{d_{1}}}{\rightarrow }%
N_{2}\oplus $t$_{Z}$($N_{1}$)$\overset{\overset{\wedge }{d_{2}}}{\rightarrow 
}N_{3}\oplus $t$_{Z}$($N_{2}$)...$N_{\ell -1}$ $\oplus $t$_{Z}$($N_{\ell })%
\overset{\overset{\wedge }{d_{\ell -1}}}{\rightarrow }N_{\ell }\rightarrow 0$%
, where $\overset{\wedge }{\text{d}_{1}}$=$\left[ 
\begin{array}{c}
\text{-d}_{1} \\ 
\text{s}_{1}%
\end{array}%
\right] $, $\overset{\wedge }{\text{d}_{i}}$= $\left[ 
\begin{array}{cc}
\text{(-1)}^{i}\text{d}_{i} & \text{j}_{i+1} \\ 
\text{s}_{i} & \text{(-1)}^{i}\text{d}_{i+1}^{\prime }%
\end{array}%
\right] $ and $\overset{\wedge }{\text{d}_{\ell -1}}$=$\left[ 
\begin{array}{cc}
\text{-d}_{\ell -1} & \text{j}_{\ell }%
\end{array}%
\right] $, and the maps j$_{i}$:t$_{Z}$($N_{i}$)$\rightarrow N_{i}$ are the
natural inclusions.
\end{lemma}

\begin{proof}
Each morphism d$_{i}$:$N_{i}\rightarrow N_{i+1}$ induces by restriction a
map \linebreak d$_{i}^{\prime }$:t$_{Z}$($N_{i}$)$\rightarrow $t$_{Z}$($%
N_{i+1}$) such that the following diagram commutes:

$%
\begin{array}{ccccccccc}
\text{t}_{Z}\text{(}N_{1}\text{)} & \overset{d_{1}^{\prime }}{\rightarrow }
& \text{t}_{Z}\text{(}N_{2}\text{)} & \overset{d_{2}^{\prime }}{\rightarrow }
& \text{t}_{Z}\text{(}N_{3}\text{)} & \rightarrow \text{...} & \text{t}_{Z}%
\text{(}N_{\ell -1}\text{)} & \overset{d_{\ell -1}^{\prime }}{\rightarrow }
& \text{t}_{Z}\text{(}N_{\ell }\text{)} \\ 
\downarrow \text{j}_{1} &  & \downarrow \text{j}_{2} &  & \downarrow \text{j}%
_{3} &  & \downarrow \text{j}_{\ell -1} &  & \downarrow \text{j}_{\ell } \\ 
N_{1} & \overset{d_{1}}{\rightarrow } & N_{2} & \overset{d_{2}}{\rightarrow }
& N_{3} &  & N_{\ell -1} & \overset{d_{\ell -1}}{\rightarrow } & N_{\ell }%
\end{array}%
$

Since the compositions d$_{i}$d$_{i-1}$ factors through a module of $Z$%
-torsion, there exist maps s$_{i}$:$N_{i}\rightarrow $t$_{Z}$($N_{i+2}$)
such that j$_{i+2}$s$_{i}$=d$_{i+1}$d$_{i}$.

We have the following equalities: j$_{i+2}$s$_{i}$j$_{i}$=d$_{i+1}$d$_{i}$j$%
_{i}$=d$_{i+1}$j$_{i}$d$_{i}^{\prime }$=j$_{i+2}$d$_{i+1}^{\prime }$d$%
_{i}^{\prime }$ and j$_{i+2}$ a monomorphism implies s$_{i}$j$_{i}$=d$%
_{i+1}^{\prime }$d$_{i}^{\prime }$ .

We can easily check $\overset{\wedge }{\text{d}_{i+1}}\overset{\wedge }{%
\text{d}_{i}}$=0.
\end{proof}

\begin{proposition}
Denote by $Q$ the localization functor $Q$=($B_{n}$)$_{Z}\underset{B}{%
\otimes }$- and by $C^{b}$(-), the category of bounded complexes. The
induced functor $C^{b}$($Q$):$C^{b}$(gr$_{_{B_{n}}}$)$\rightarrow $ $C^{b}$%
(gr$_{(B_{n})_{Z}})$ is dense.
\end{proposition}

\begin{proof}
Let $0\rightarrow \overset{\wedge }{M}_{1}\overset{\delta _{1}}{\rightarrow }
$ $\overset{\wedge }{M}_{2}\overset{\delta _{2}}{\rightarrow }...\overset{%
\wedge }{M}$ $_{\ell -1}\overset{\delta _{\ell -1}}{\rightarrow }\overset{%
\wedge }{M}$ $_{\ell }\rightarrow 0$ be a complex in $C^{b}$(gr$%
_{(B_{n})_{Z}}$).

For each $\overset{\wedge }{M}_{i}$there exists a finitely generated graded $%
B_{n}$-submodule $M_{i}$ such that ($M_{i}$)$_{Z}\cong \overset{\wedge }{M}%
_{i}$ and a graded morphism d$_{i}$:Z$^{k_{i}}M_{i}\rightarrow M_{i+1}$ of $%
B_{n}$-modules such that (d$_{i}$)$_{Z}$:(Z$^{k_{i}}M_{i}$)$_{Z}\rightarrow $%
($M_{i+1}$)$_{Z}$ is isomorphic $\delta _{i}$: $\overset{\wedge }{M}%
_{i}\rightarrow \overset{\wedge }{M}_{i+1}$. Let k be $\overset{\ell }{%
\underset{i=0}{\sum }}$k$^{i}$. We then have a chain of $B_{n}$-morphisms:

Z$^{k}M_{1}\overset{d_{1}}{\rightarrow }$Z$^{k_{2}+..k_{\ell }}M_{2}\overset{%
d_{2}}{\rightarrow }$Z$^{k_{3}+..k_{\ell }}M_{3}\overset{d_{3}}{\rightarrow }
$...Z$^{k_{\ell -1}+..k_{\ell }}M_{\ell -1}\overset{d_{\ell -1}}{\rightarrow 
}$ Z$^{k_{\ell }}M_{\ell }$. Changing notation write $M_{i}$ instead of Z$%
^{k_{i}+..k_{\ell }}M_{i}$.

We then have a chain of morphisms:\newline
$M_{1}$ $\overset{d_{1}}{\rightarrow }M_{2}\overset{d_{2}}{\rightarrow }%
...M_{\ell -1}\overset{d_{\ell -1}}{\rightarrow }M_{\ell }$ such that ($%
M_{1} $)$_{Z}$ $\overset{d_{1_{Z}}}{\rightarrow }$($M_{2}$)$_{Z}\overset{%
d_{2_{Z}}}{\rightarrow }$...($M_{\ell -1}$)$_{Z}\overset{d_{\ell -1_{Z}}}{%
\rightarrow } $($M_{\ell }$)$_{Z}$ is isomorphic to the complex: $%
0\rightarrow \overset{\wedge }{M}_{1}\overset{\delta _{1}}{\rightarrow }$ $%
\overset{\wedge }{M}_{2}\overset{\delta _{2}}{\rightarrow }...\overset{%
\wedge }{M}$ $_{\ell -1}\overset{\delta _{\ell -1}}{\rightarrow }\overset{%
\wedge }{M}$ $_{\ell }\rightarrow 0.$

This implies (d$_{i}$d$_{i-1}$)$_{Z}$=0, which means d$_{i}$d$_{i-1}$
factors through a Z-torsion module. By lemma? there exists a complex:\newline
0$\rightarrow M_{1}\overset{\overset{\wedge }{d_{1}}}{\rightarrow }%
M_{2}\oplus $t$_{Z}$($M_{1}$)$\overset{\overset{\wedge }{d_{2}}}{\rightarrow 
}M_{3}\oplus $t$_{Z}$($M_{2}$)...$M_{\ell -1}$ $\oplus $t$_{Z}$($M_{\ell }$)$%
\overset{\overset{\wedge }{d_{\ell -1}}}{\rightarrow }M_{\ell }\rightarrow $%
0 such that \linebreak 0$\rightarrow M_{1Z}\overset{\overset{\wedge }{%
d_{1_{Z}}}}{\rightarrow }$($M_{2}\oplus $t$_{Z}$($M_{1}$))$_{Z}\overset{%
\overset{\wedge }{d_{2_{Z}}}}{\rightarrow }$($M_{3}\oplus $t$_{Z}$($M_{2}$))$%
_{Z}$...($M_{\ell -1}$ $\oplus $t$_{Z}$($M_{\ell }$))$_{Z}\overset{\overset{%
\wedge }{d_{\ell -1_{Z}}}}{\rightarrow }M_{\ell Z}\rightarrow $0 is
isomorphic to: $0\rightarrow \overset{\wedge }{M}_{1}\overset{\delta _{1}}{%
\rightarrow }$ $\overset{\wedge }{M}_{2}\overset{\delta _{2}}{\rightarrow }%
...\overset{\wedge }{M}$ $_{\ell -1}\overset{\delta _{\ell -1}}{\rightarrow }%
\overset{\wedge }{M}$ $_{\ell }\rightarrow 0$.
\end{proof}

\begin{corollary}
The functor $C^{b}$($\psi $):$C^{b}$(Qgr$_{_{B_{n}}}$)$\rightarrow $ $C^{b}$%
(gr$_{(B_{n})_{Z}}$) is dense.
\end{corollary}

\begin{proof}
There are functors $C^{b}$($\pi $):$C^{b}$(gr$_{B_{n}}$)$\rightarrow C^{b}$%
(Qgr$_{B_{n}}$) and $C^{b}$($\psi $):$C^{b}$(Qgr$_{B_{n}}$)$\rightarrow
C^{b} $(gr$_{(B_{n})_{Z}}$) such that $C^{b}$($\psi $) $C^{b}$($\pi $)=$%
C^{b} $($Q$) and $C^{b}$($Q$) dense implies $C^{b}$($\psi $) is dense.
\end{proof}

\begin{corollary}
The induced functors $K^{b}(Q):$ $K^{b}(gr_{B_{n}})\rightarrow $ $%
K^{b}(gr_{(B_{n})_{Z}})$ and $K^{b}(\psi ):$ $K^{b}(Qgr_{B_{n}})\rightarrow $
$K^{b}(gr_{(B_{n})_{Z}})$ are dense.
\end{corollary}

\begin{proof}
Interpreting $K^{b}$($\mathcal{A}$) as the stable category of $C^{b}$($%
\mathcal{A}$), denote by $\tau $:$C^{b}$($\mathcal{A}$)$\rightarrow K^{b}$($%
\mathcal{A}$) the corresponding functor. There is a commutative diagram:

$%
\begin{array}{ccc}
C^{b}\text{(gr}_{B_{n}}\text{)} & \overset{C^{b}(Q)}{\rightarrow } & C^{b}%
\text{(gr}_{(B_{n})_{Z}}\text{)} \\ 
\downarrow \tau &  & \downarrow \tau \\ 
K^{b}\text{(gr}_{B_{n}}\text{)} & \overset{K^{b}(Q)}{\rightarrow } & K^{b}%
\text{(gr}_{(B_{n})_{Z}}\text{)}%
\end{array}%
$

Since the functors $\tau $ and $C^{b}$($Q$) are dense, the functor $K^{b}$($%
Q $) is dense.

As above we have isomorphisms: $K^{b}$($\psi $) $K^{b}$($\pi $)$\cong $ $%
K^{b}$($Q$). It follows $K^{b}$($\psi $) is dense.
\end{proof}

\begin{corollary}
The induced functors D$^{b}$($Q$):D$^{b}$(gr$_{B_{n}}$)$\rightarrow $D$^{b}$%
(gr$_{(B_{n})_{Z}}$) and \linebreak D$^{b}$($\psi $):D$^{b}$(Qgr$_{B_{n}}$)$%
\rightarrow $D$^{b}$(gr$_{(B_{n})_{Z}}$) are dense.
\end{corollary}

\begin{proof}
Since the functors $\pi $:gr$_{B_{n}}\rightarrow $Qgr$_{B_{n}}$ and $\psi $%
:Qgr$_{B_{n}}\rightarrow $gr(B$_{n}$)$_{Z}$ are exact they induce derived
functors D$^{b}$($\pi $):D$^{b}$(gr$_{B_{n}}$)$\rightarrow $D$^{b}$(Qgr$%
_{B_{n}}$), D$^{b}$($\psi $):D$^{b}$(Qgr$_{B_{n}}$)$\rightarrow $D$^{b}$(gr$%
_{B_{n}{}_{Z}}$) such that D$^{b}$($\psi $)D$^{b}$($\pi $)=D$^{b}$($Q$).

There is a commutative exact diagram:

$%
\begin{array}{ccc}
K^{b}\text{(gr}_{B_{n}}\text{)} & \overset{K^{b}(Q)}{\rightarrow } & K^{b}%
\text{(gr}_{(B_{n})_{Z}}\text{)} \\ 
\downarrow &  & \downarrow \\ 
\text{D}^{b}\text{(gr}_{B_{n}}\text{)} & \overset{D^{b}(Q)}{\rightarrow } & 
\text{D}^{b}\text{(gr}_{(B_{n})_{Z}}\text{)}%
\end{array}%
$

where the functors corresponding to the columns are dense, hence D$^{b}$($Q$%
) is dense, which in turn implies D$^{b}$($\psi $) is dense.
\end{proof}

We will describe next he kernel of the functor D$^{b}$($\psi $). By
definition, KerD$^{b}$($\psi $)=\linebreak \{$\overset{\wedge }{M^{\circ }}%
\mid $D$^{b}$($\psi $)($\overset{\wedge }{M^{\circ }}$) is acyclic\}.

\begin{proposition}
There is the following description of $\mathcal{T=}$KerD$^{b}$($\psi $).KerD$%
^{b}$($\psi $)=\linebreak \{$\pi M^{\circ }\mid M^{\circ }\in $D$^{b}$(gr$%
_{B_{n}}$) such that for all i, H$^{i}$($M^{\circ }$) is of Z-torsion\}.
\end{proposition}

\begin{proof}
The kernel of the functor D$^{b}$($\psi $) is the category $\mathcal{T}$ of
complexes:

$\overset{\sim }{N}^{\circ }:0\rightarrow \pi N_{1}\overset{\overset{\wedge }%
{d_{1}}}{\rightarrow }\pi N_{2}\overset{\overset{\wedge }{d_{2}}}{%
\rightarrow }\pi N_{3}...\pi N_{\ell -1}$ $\overset{\overset{\wedge }{%
d_{\ell -1}}}{\rightarrow }\pi N_{\ell }\rightarrow 0$, such that:

$0\rightarrow \psi \pi N_{1}\overset{\psi \overset{\wedge }{d_{1}}}{%
\rightarrow }\psi \pi N_{2}\overset{\psi \overset{\wedge }{d_{2}}}{%
\rightarrow }\psi \pi N_{3}...\psi \pi N_{\ell -1}$ $\overset{\overset{%
\wedge }{\psi d_{\ell -1}}}{\rightarrow }\psi \pi N_{\ell }\rightarrow 0$ is
acyclic.

Proceeding as above, we may assume each $N_{i}$ and each map $\overset{%
\wedge }{\text{d}_{i}}$ lifts to a map d$_{i}$:($N_{i}$)$_{\geq
k}\rightarrow N_{i+1}$ such that the map $\pi $(d$_{i}$):$\pi $(($N_{i}$)$%
_{\geq k}$)$\rightarrow \pi $($N_{i+1}$) is isomorphic to $\overset{\wedge }{%
\text{d}_{i}}$: $\pi N_{i}\rightarrow \pi N_{i+1}.$

Taking a large enough truncation we get a complex of $B_{n}$-modules:

$N_{\geq k}^{\circ }$: 0$\rightarrow $($N_{1}$)$_{\geq k}\overset{d_{1}}{%
\rightarrow }$($N_{2}$)$_{\geq k}\overset{d_{2}}{\rightarrow }$...($N_{\ell
-1}$)$_{\geq k}\overset{d_{\ell -1}}{\rightarrow }$($N_{\ell }$)$_{\geq
k}\rightarrow $0, \linebreak such that $\pi $($N_{\geq k}^{\circ }$)$\cong 
\overset{\sim }{N}^{\circ }.$

The complex:

($N_{\geq k}^{\circ }$)$_{Z}$: 0$\rightarrow $($N_{1}$)$_{\geq k}$)$_{Z}%
\overset{d_{1_{Z}}}{\rightarrow }$(($N_{2}$)$_{\geq k}$)$_{Z}\overset{%
d_{2_{Z}}}{\rightarrow }$...(($N_{\ell -1}$)$_{\geq k}$)$_{Z}\overset{%
d_{\ell -1_{Z}}}{\rightarrow }$(($N_{\ell }$)$_{\geq k}$)$_{Z}\rightarrow $0
is isomorphic to 0$\rightarrow \psi \pi N_{1}\overset{\psi \overset{\wedge }{%
d_{1}}}{\rightarrow }\psi \pi N_{2}\overset{\psi \overset{\wedge }{d_{2}}}{%
\rightarrow }\psi \pi N_{3}$...$\psi \pi N_{\ell -1}$ $\overset{\overset{%
\wedge }{\psi d_{\ell -1}}}{\rightarrow }\psi \pi N_{\ell }\rightarrow $0,
hence it is acyclic.

Changing notation we have a complex of $B_{n}$-modules:

$N^{\circ }$: 0$\rightarrow N_{1}\overset{d_{1}}{\rightarrow }N_{2}\overset{%
d_{2}}{\rightarrow }...N_{\ell -1}\overset{d_{\ell -1}}{\rightarrow }N_{\ell
}\rightarrow $0 such that \linebreak ($N^{\circ }$)$_{Z}$: 0$\rightarrow $($%
N_{1}$)$_{Z}\overset{d_{1_{Z}}}{\rightarrow }$($N_{2}$)$_{Z}\overset{%
d_{2_{Z}}}{\rightarrow }$...($N_{\ell -1}$)$_{Z}\overset{d_{\ell -1_{Z}}}{%
\rightarrow }$($N_{\ell }$)$_{Z}\rightarrow $0 is acyclic.

We have exact sequences:

$%
\begin{array}{cccccccccc}
&  &  &  &  & 0 &  &  &  &  \\ 
&  &  &  &  & \downarrow &  &  &  &  \\ 
0\rightarrow & \text{Kerd}_{1} & \rightarrow & N_{1} & \overset{d_{1}}{%
\rightarrow } & \text{Imd}_{1} & \rightarrow 0 &  &  &  \\ 
&  &  &  &  & \downarrow \text{j} &  &  &  &  \\ 
&  &  &  & 0\rightarrow & \text{Kerd}_{2} & \rightarrow & N_{2} & \overset{%
d_{2}}{\rightarrow } & N_{3} \\ 
&  &  &  &  & \downarrow &  &  &  &  \\ 
&  &  &  &  & \text{H}^{1}\text{(}N^{\circ }\text{)} &  &  &  &  \\ 
&  &  &  &  & \downarrow &  &  &  &  \\ 
&  &  &  &  & 0 &  &  &  & 
\end{array}%
$

Localizing we get exact sequences:

$%
\begin{array}{cccccccccc}
&  &  &  &  & 0 &  &  &  &  \\ 
&  &  &  &  & \downarrow &  &  &  &  \\ 
0\rightarrow & \text{(Kerd}_{1}\text{)}_{Z} & \rightarrow & \text{(}N_{1}%
\text{)}_{Z} & \overset{d_{1_{Z}}}{\rightarrow } & \text{(Imd}_{1}\text{)}%
_{Z} & \rightarrow 0 &  &  &  \\ 
&  &  &  &  & \downarrow \text{j}_{Z} &  &  &  &  \\ 
&  &  &  & 0\rightarrow & \text{(Kerd}_{2}\text{)}_{Z} & \rightarrow & \text{%
(}N_{2}\text{)}_{Z} & \overset{d_{2_{Z}}}{\rightarrow } & \text{(}N_{3})_{Z}
\\ 
&  &  &  &  & \downarrow &  &  &  &  \\ 
&  &  &  &  & \text{H}^{1}\text{(}N^{\circ }\text{)}_{Z} &  &  &  &  \\ 
&  &  &  &  & \downarrow &  &  &  &  \\ 
&  &  &  &  & 0 &  &  &  & 
\end{array}%
$

where j$_{Z}$, d$_{1_{Z}}$ are isomorphisms, hence H$^{0}$($N^{\circ }$)$%
_{Z} $=0, H$^{1}$($N^{\circ }$)$_{Z}$=0. Therefore H$^{0}$($N^{\circ }$) and
H$^{1}$($N^{\circ }$) are Z-torsion. More generally for each i the modules H$%
^{i}$($N^{\circ }$) are Z-torsion.
\end{proof}

\begin{remark}
Let $M$ be a Koszul left $B^{!}$-module such that $F$($M$)=$\underset{n\geq 0%
}{\oplus }$Ext$_{B^{!}}^{n}$($M$,$B_{0}^{!}$) is a $B$-module of Z-torsion
and $\sigma :B_{n}\rightarrow B_{n}$ the Nakayama automorphism defined in
Section 1. Then $F$($\sigma M$) is of Z-torsion, in particular $F$(D($%
M^{\ast }$)[n]) is of Z-torsion.
\end{remark}

Since $F$($\sigma M$)=$\sigma FM$ for x$\in F(M)$ there is an integer k$\geq 
$0 such that Z$^{k}$x=0 and in $\sigma FM$, Z$^{k}\ast $x=$\sigma $(Z$^{k}$%
)x=c$^{k}$Z$^{k}$x=0.

\begin{corollary}
\bigskip The Nakayama automorphism $\sigma $: $B_{n}\rightarrow B_{n}$
induces an autoequivalence D$^{b}$($\sigma $): D$^{b}$(gr$_{B_{n}}$)$%
\rightarrow $D$^{b}$(gr$_{B_{n}}$) and $\mathcal{T}$ is invariant under D$%
^{b}$($\sigma $).
\end{corollary}

\begin{proof}
We saw in Section 1 that given an automorphism of graded algebras $\sigma $:$%
B_{n}\rightarrow B_{n}$, there is an autoequivalence gr$_{B_{n}}\rightarrow $%
gr$_{B_{n}}$, that we also denote by $\sigma $, such that $\sigma $($M$) is
the module $M$ with twisted multiplication b$\in B_{n}$ and m$\in M$, b$\ast 
$m=$\sigma $(b)m, clearly $\sigma $ is an exact functor that sends modules
of finite length into modules of finite length. Then $\sigma $ induces an
exact functor: $\sigma $: Qgr$_{B_{n}}\rightarrow $Qgr$_{B_{n}}$. Therefore
an autoequivalence: D$^{b}$($\sigma $):D$^{b}$(gr$_{B_{n}}$)$\rightarrow $D$%
^{b}$(gr$_{B_{n}}$)$.$ If $M$ is a module of Z-torsion, then $\sigma M$ is
of Z-torsion. From this it is clear that D$^{b}$($\sigma $) sends an element
of $\mathcal{T}$ to an element of $\mathcal{T}$.
\end{proof}

The category $\mathcal{T}$=KerD($\psi $) is "epasse" (thick) and we can take
the Verdier quotient D$^{b}$(QgrB$_{n}$) /$\mathcal{T}.$[Mi]

Our aim is to prove the main result of the section:

\begin{theorem}
There exists an equivalence of triangulated categories:

D$^{b}$(Qgr$_{B_{n}}$)/$\mathcal{T}\cong $D$^{b}$(gr$_{(B_{n})_{Z}}$).
\end{theorem}

Let $\overset{\wedge }{\text{s}}^{-1}\overset{\wedge }{\text{f}}:$ $\overset{%
\wedge }{X}^{\circ }\rightarrow \overset{\wedge }{Y}^{\circ }$be a map in $%
\Psi $($\mathcal{T}$) (in Miyachi's notation). This is a roof

\begin{center}
$%
\begin{array}{ccccc}
&  & \overset{\wedge }{K}^{\circ } &  &  \\ 
& \overset{\wedge }{\text{f}}\nearrow &  & \nwarrow \overset{\wedge }{\text{s%
}} &  \\ 
\overset{\wedge }{X}^{\circ } &  &  &  & \overset{\wedge }{Y}^{\circ }%
\end{array}%
$,
\end{center}

where $\overset{\wedge }{X}^{\circ }$is a complex of the form:

$\bigskip \overset{\wedge }{X}^{\circ :}$:$0\rightarrow \pi X^{n_{0}}\overset%
{\overset{\wedge }{d}}{\rightarrow }\pi X^{n_{0+1}}\overset{\overset{\wedge }%
{d}}{\rightarrow }$...$\rightarrow \pi X^{n_{0}+\ell -1}\overset{\overset{%
\wedge }{d}}{\rightarrow }$ $\pi X^{n_{0}+\ell }\rightarrow 0.$

After a proper truncation there exists complexes of graded $B_{n}$-modules $%
X^{\circ }$, $K^{\circ }$, $Y^{\circ }$ such that $\pi X^{\circ }\cong $ $%
\overset{\wedge }{X}^{\circ }$, $\pi K^{\circ }\cong \overset{\wedge }{K}%
^{\circ :}$,$\pi $ $Y^{\circ }\cong \overset{\wedge }{Y}^{\circ :}$and
graded maps f:$X^{\circ }\rightarrow K^{\circ }$, s: $Y^{\circ }\rightarrow
K^{\circ }$ such that $\pi $f=$\overset{\wedge }{\text{f}}$, $\pi s=\overset{%
\wedge }{\text{s}}$, the roof $\overset{\wedge }{\text{s}}^{-1}\overset{%
\wedge }{\text{f}}$ becomes:

\begin{center}
$%
\begin{array}{ccccc}
&  & \pi K^{\circ } &  &  \\ 
& \pi \text{f}\nearrow &  & \nwarrow \pi \text{s} &  \\ 
\pi X^{\circ } &  &  &  & \pi Y^{\circ }%
\end{array}%
$,
\end{center}

where $\pi $s is a quasi isomorphism.

There is a triangle in $K^{b}$(gr$_{B_{n}}$): $X^{\circ }\overset{f}{%
\rightarrow }K^{\circ }\overset{g}{\rightarrow }Z^{\circ }\overset{h}{%
\rightarrow }X^{\circ }$[-1] which induces a morphism of triangles:

\begin{center}
$%
\begin{array}{ccccccc}
X^{\circ } & \overset{f}{\rightarrow } & K^{\circ } & \overset{g}{%
\rightarrow } & Z^{\circ } & \overset{h}{\rightarrow } & X^{\circ }\text{[-1]%
} \\ 
\uparrow \text{u} &  & \uparrow \text{s} &  & \uparrow \text{1} &  & 
\uparrow \text{u[-1]} \\ 
X^{\prime \circ } & \rightarrow & Y^{\circ } & \overset{gs}{\rightarrow } & 
Z^{\circ } & \rightarrow & X^{\prime \circ }\text{[-1]}%
\end{array}%
$
\end{center}

Applying $\pi $ we obtain a morphism of triangles:

\begin{center}
$%
\begin{array}{ccccccc}
\pi X^{\circ } & \overset{\pi f}{\rightarrow } & \pi K^{\circ } & \overset{%
\pi g}{\rightarrow } & \pi Z^{\circ } & \pi \overset{h}{\rightarrow } & \pi
X^{\circ }\text{[-1]} \\ 
\uparrow \pi \text{u} &  & \uparrow \pi \text{s} &  & \uparrow \text{1} &  & 
\uparrow \pi \text{u[-1]} \\ 
\pi X^{\prime \circ } & \rightarrow & \pi Y^{\circ } & \overset{\pi gs}{%
\rightarrow } & \pi Z^{\circ } & \rightarrow & \pi X^{\prime \circ }\text{%
[-1]}%
\end{array}%
$
\end{center}

By definition of $\Psi $($\mathcal{T}$) the object $\pi Z^{\circ }\in 
\mathcal{T},$which means $Z^{\circ }$ has homology of Z-torsion. The maps $%
\pi $s, $\pi $u are quasi isomorphisms. Applying the functor $\psi $ we
obtain a triangle: $\psi \pi X^{\circ }\overset{\psi \pi f}{\rightarrow }%
\psi \pi K^{\circ }\overset{\psi \pi g}{\rightarrow }\psi \pi Z^{\circ }%
\overset{\psi \pi h}{\rightarrow }\psi \pi X^{\circ }$[-1] where $\psi \pi
Z^{\circ }$is acyclic. It follows $\psi \pi $ f is invertible in D$^{b}$(gr$%
_{(B_{n})_{Z}}$).

We have proved the functor: D$^{b}$($\psi $):D$^{b}$(Qgr$_{B_{n}}$)$%
\rightarrow $D$^{b}$(gr$_{(B_{n})_{Z}}$) sends elements of $\Psi $($\mathcal{%
T}$) to invertible elements in D$^{b}$(gr(B$_{n}$)$_{Z}$). By [Mi] Prop.
712, there exists a functor $\theta $:D$^{b}$(Qgr$_{B_{n}}$)/$\mathcal{T}%
\rightarrow $D$^{b}$(gr$_{(B_{n})_{Z}}$) such that the triangle:

\begin{center}
$%
\begin{array}{ccccc}
D^{b}\text{(Qgr}_{B_{n}}\text{)} &  & \overset{D^{b}(\psi )}{\rightarrow } & 
& \text{D}^{b}\text{(gr}_{\text{(B}_{n}\text{)}_{Z}}\text{)} \\ 
& Q\searrow &  & \nearrow \theta &  \\ 
&  & \text{D}^{b}\text{(Qgr}_{B_{n}}\text{)/}\mathcal{T} &  & 
\end{array}%
$,
\end{center}

commutes.

Since $D^{b}$($\psi $) is dense, so is $\theta .$

Before proving $\theta $ is an equivalence, we will need two lemmas:

\begin{lemma}
Let $K^{\circ }$, $L^{\circ }$ be complexes in $C^{b}$(gr$_{B_{n}}$) and let 
$\overset{\wedge }{\text{f}}$:$K_{Z}^{\circ }\rightarrow $ $L_{Z}^{\circ }$
be a morphism of complexes of graded ($B_{n}$)$_{Z}$-modules. Then there
exists a bounded complex of graded $B_{n}$-modules $N^{\circ }$ and a map of
complexes f:$N^{\circ }\rightarrow L^{\circ }$ such that $N_{Z}^{\circ
}\cong K_{Z}^{\circ }$ and f$_{Z}$=$\overset{\wedge }{\text{f}}.$
\end{lemma}

\begin{proof}
Let $K^{\circ }$, $L^{\circ }$ be the complexes: $K^{\circ }$: $0\rightarrow
K^{0}\overset{d}{\rightarrow }$ $K^{1}\overset{d}{\rightarrow }...K^{n-1}%
\overset{d}{\rightarrow }K^{n}\rightarrow 0$ and $L^{\circ }$: $0\rightarrow
L^{0}\overset{d}{\rightarrow }$ $L^{1}\overset{d}{\rightarrow }...L^{n-1}%
\overset{d}{\rightarrow }L^{n}\rightarrow 0$.

Each map $\overset{\wedge }{\text{f}}_{i}$:$K_{Z}^{i}\rightarrow $ $%
L_{Z}^{i} $ lifts to a map f$_{i}$:Z$^{k_{i}}K^{i}\rightarrow L^{i}$ such
that (f$_{i}$)$_{Z}$=$\overset{\wedge }{\text{f}}_{i}$. Let k be max\{k$_{j}$%
\}. Then we have the following diagram:

$%
\begin{array}{ccccccccc}
0\rightarrow & \text{Z}^{k}K^{0} & \overset{d}{\rightarrow } & \text{Z}%
^{k}K^{1} & \overset{d}{\rightarrow } & \text{Z}^{k}K^{2} & \overset{d}{%
\rightarrow }\text{...} & \text{Z}^{k}K^{n} & \rightarrow 0 \\ 
& \downarrow \text{f}_{0} &  & \downarrow \text{f}_{1} &  & \downarrow \text{%
f}_{2} &  & \downarrow \text{f}_{n} &  \\ 
0\rightarrow & L^{0} & \overset{d}{\rightarrow } & L^{1} & \overset{d}{%
\rightarrow } & L^{2} & \overset{d}{\rightarrow }\text{...} & L^{n} & 
\rightarrow 0%
\end{array}%
$

where (df$_{i-1}$-f$_{i}$d)$_{Z}$=d$\overset{\wedge }{f}_{i-1}$-$\overset{%
\wedge }{f}_{i}$d=0.The map df$_{i-1}$-f$_{i}$d factors though t$_{Z}$(L$%
^{i} $).

There exist maps s$_{i-1}$:Z$^{k}K^{i-1}\rightarrow $t$_{Z}$($L^{i}$), j$%
_{i} $:t$_{Z}$($L^{i}$)$\rightarrow L^{i}$ such that f$_{i}$d-df$_{i-1}$=j$%
_{i}$s$_{i-1}$ and the diagrams:

$%
\begin{array}{ccc}
\text{t}_{Z}\text{(}L^{i-1}\text{)} & \overset{d^{\prime }}{\rightarrow } & 
\text{t}_{Z}\text{(}L^{i}\text{)} \\ 
\downarrow \text{j}_{i-1} &  & \downarrow \text{j}_{i} \\ 
L^{i-1} & \overset{d}{\rightarrow } & L^{i}%
\end{array}%
$

commute.

We have the following equalities: (f$_{i}$d-df$_{i-1}$)d=j$_{i}$s$_{i-1}$d,
-df$_{i-1}$d=j$_{i}$s$_{i-1}$d and \linebreak d(f$_{i-1}$d-df$_{i-2}$)=df$%
_{i-1}$f=dj$_{i-1}$s$_{i-2}$=j$_{i}$ds$_{i-2}$.

But j$_{i}$ mono implies s$_{i-1}$d+ds$_{i-2}$=0.

We have proved that the sequence\newline
$N^{\circ }$: 0$\rightarrow $Z$^{k}K^{0}\overset{\overset{\wedge }{d_{0}}}{%
\rightarrow }$Z$^{k}K^{1}\oplus $t$_{Z}$($L^{1}$)$\overset{\overset{\wedge }{%
d_{1}}}{\rightarrow }$Z$^{k}K^{2}\oplus $t$_{Z}$($L^{2}$)...Z$%
^{k}K^{n-1}\oplus $t$_{Z}$($L^{n-1}$)$\overset{\overset{\wedge }{d_{\ell -1}}%
}{\rightarrow }$Z$^{k}K^{n}\rightarrow $0, with maps: $\overset{\wedge }{%
\text{d}_{0}}$ = $\left[ 
\begin{array}{c}
\text{d} \\ 
\text{s}_{0}%
\end{array}%
\right] $, $\overset{\wedge }{\text{d}_{i}}$ =$\left[ 
\begin{array}{cc}
\text{d} & \text{0} \\ 
\text{s}_{i} & \text{d}^{\prime }%
\end{array}%
\right] $, $\overset{\wedge }{\text{d}_{\ell -1}}$=$\left[ 
\begin{array}{cc}
\text{d} & \text{0}%
\end{array}%
\right] $ is a complex of $B_{n}$-modules and (f$_{i}$,-j$_{i}$):Z$%
^{k}K^{i}\oplus $t$_{Z}$($L^{i}$)$\rightarrow L^{i}$, (f,-j):$N^{\circ
}\rightarrow L^{\circ }$ is a map of complexes such that $N_{Z}^{\circ
}\cong K_{Z}^{\circ }$ and (f,j)$_{Z}$=$\overset{\wedge }{\text{f}}$.
\end{proof}

\begin{lemma}
Let $K^{\circ }$, $L^{\circ }$ be complexes in $C^{b}$(gr$_{B_{n}}$) and $%
\overset{\wedge }{\text{f}}$:$L_{Z}^{\circ }\rightarrow $ $K_{Z}^{\circ }$
be a morphism of complexes of graded ($B_{n}$)$_{Z}$-modules which is
homotopic to zero. Then there exist bounded complexes of $B_{n}$-modules, $%
M^{\circ }$, $N^{\circ }$and a map of complexes f:$N^{\circ }\rightarrow
M^{\circ }$, such that f is homotopic to zero, $N_{Z}^{\circ }\cong
L_{Z}^{\circ }$, $M_{Z}^{\circ }\cong K_{Z}^{\circ }$ and f$_{Z}\cong $ $%
\overset{\wedge }{\text{f}}$.
\end{lemma}

\begin{proof}
Consider the following diagram:

$%
\begin{array}{ccccccccc}
0\rightarrow & L_{Z}^{0} & \overset{d_{Z}}{\rightarrow } & L_{Z}^{1} & 
\overset{d_{Z}}{\rightarrow } & L_{Z}^{2} & \overset{d_{Z}}{\rightarrow }%
\text{...} & L_{Z}^{m} & \rightarrow 0 \\ 
& \downarrow \overset{\wedge }{\text{f}}_{0} & \text{s}_{1}\swarrow & 
\downarrow \overset{\wedge }{\text{f}_{1}} & \text{s}_{2}\swarrow & 
\downarrow \overset{\wedge }{\text{f}}_{2} & \text{s}_{m}\swarrow & 
\downarrow \overset{\wedge }{\text{f}}_{m} &  \\ 
0\rightarrow & K_{Z}^{0} & \overset{d_{Z}}{\rightarrow } & K_{Z}^{1} & 
\overset{d_{Z}}{\rightarrow } & K_{Z}^{2} & \overset{d_{Z}}{\rightarrow }%
\text{...} & L_{Z}^{n} & \rightarrow 0%
\end{array}%
$,

where s:$L_{Z}^{\circ }\rightarrow K_{Z}^{\circ }$[-1] is the homotopy,
hence $\overset{\wedge }{\text{f}}_{i}$=d$_{Z}$s$_{i}$+s$_{i+1}$d$_{Z}$.

For each $i$ there exist integers k$_{i}$, k$_{i}^{\prime }$ and maps t$_{i}$%
: Z$^{k_{i}}L^{i}\rightarrow K^{i-1}$ and f$_{i}^{\prime }$: Z$%
^{k_{i}^{\prime }}L^{i}\rightarrow K^{i}$ such that (t$_{i}$)$_{Z}$=s$_{i}$
and (f$_{i}^{\prime }$)$_{Z}$=$\overset{\wedge }{\text{f}}_{i}$. Taking
k=max\{k$_{j}$\}we have maps:

$%
\begin{array}{ccccccccc}
0\rightarrow & \text{Z}^{k}L^{0} & \overset{d}{\rightarrow } & \text{Z}%
^{k}L^{1} & \overset{d}{\rightarrow } & \text{Z}^{k}L^{2} & \overset{d}{%
\rightarrow }\text{...} & \text{Z}^{k}L^{m} & \rightarrow 0 \\ 
& \downarrow \text{f}_{0}^{\prime } & \text{t}_{1}\swarrow & \downarrow 
\text{f}_{1}^{\prime } & \text{t}_{2}\swarrow & \downarrow \text{f}%
_{2}^{\prime } & \text{t}_{m}\swarrow & \downarrow \text{f}_{m}^{\prime } & 
\\ 
0\rightarrow & K^{0} & \overset{d}{\rightarrow } & K^{1} & \overset{d}{%
\rightarrow } & K^{2} & \overset{d}{\rightarrow }\text{...} & K^{m} & 
\rightarrow 0%
\end{array}%
$

Consider the map: (f$_{i}^{\prime }$-(t$_{i+1}$d+dt$_{i}$))$_{Z}$=$\overset{%
\wedge }{\text{f}}_{i}$-(s$_{i+1}$d$_{Z}$+d$_{Z}$s$_{i}$=0. As
above,\linebreak\ f$_{i}^{\prime }$-(t$_{i+1}$d+dt$_{i}$) factors through a
Z-torsion module and there exist maps: \linebreak v$_{i}$:Z$%
^{k}L^{i}\rightarrow $t$_{Z}$($K^{i}$) and inclusions j$_{i}$:t$_{Z}$($K^{i}$%
)$\rightarrow K^{i}$ such that f$_{i}^{\prime }$-(t$_{i+1}$d+dt$_{i}$)=-j$%
_{i}$v$_{i}$ or f$_{i}^{\prime }$+j$_{i}$v$_{i}$=t$_{i+1}$d+dt$_{i}$.

Set f$_{i}$=f$_{i}^{\prime }$+j$_{i}$v$_{i}$.Then (f$_{i}$)$_{Z}$=(f$%
_{i}^{\prime }$)$_{Z}$=$\overset{\wedge }{\text{f}}_{i}$.

But we have now f$_{i}$=t$_{i+1}$d+dt$_{i}$,f$_{i-1}$=t$_{i}$d+dt$_{i-1}$
imply f$_{i}$d=dt$_{i}$d=f$_{i-1}$d.
\end{proof}

We can prove now the theorem.

i) $\theta $ is faithful.

Let $\mathcal{T}_{1}$ be $\mathcal{T}_{1}$=\{$X^{\circ }\in C^{b}$(gr$%
_{B_{n}}$)$\mid $H$^{i}$($X^{\circ }$) is torsion for all i\} and $\mathcal{T%
}_{2}$=\{$X^{\circ }\in C^{b}$(gr$_{B_{n}}$)$\mid $H$^{i}$($X^{\circ }$) is
Z-torsion for all i\}.

A map in D$^{b}$(Qgr$_{B_{n}}$)/$\mathcal{T}$ can be written as follows:

\begin{center}
$%
\begin{array}{ccccccccc}
&  & \pi K^{\circ } &  &  &  & \pi L^{\circ } &  &  \\ 
& \pi \text{f}\nearrow &  & \nwarrow \pi \text{s} &  & \pi \text{t}\nearrow
&  & \nwarrow \pi \text{g} &  \\ 
\pi X^{\circ } &  &  &  & \pi Y^{\circ } &  &  &  & \pi Z^{\circ }%
\end{array}%
$
\end{center}

where t, s$\in \Psi $($\mathcal{T}_{1}$) and g$\in \Psi $($\mathcal{T}_{2}$).

In $K^{b}$(gr$_{B_{n}}$) we have maps:

\begin{center}
$%
\begin{array}{ccccccccc}
&  & K^{\circ } &  &  &  & L^{\circ } &  &  \\ 
& \text{f}\nearrow &  & \nwarrow \text{s} &  & \text{t}\nearrow &  & 
\nwarrow \text{g} &  \\ 
X^{\circ } &  &  &  & Y^{\circ } &  &  &  & Z^{\circ }%
\end{array}%
$
\end{center}

We have an exact sequence of complexes:

\begin{center}
$%
\begin{array}{ccccccc}
0\rightarrow & Y^{\circ } & \overset{\mu }{\rightarrow } & K^{\circ }\oplus
L^{\circ }\oplus I^{\circ } & \overset{\upsilon }{\rightarrow } & W^{\circ }
& \rightarrow 0%
\end{array}%
$
\end{center}

Where $I^{\circ }$ is a complex which is a sum of complexes of the form: $%
0\rightarrow X\overset{1}{\rightarrow }X\rightarrow 0$, hence acyclic. The
maps $\mu ,\upsilon $ are of the form: $\mu =\left[ 
\begin{array}{c}
\text{s} \\ 
\text{t} \\ 
\text{u}%
\end{array}%
\right] $ and $\upsilon =\left[ 
\begin{array}{ccc}
\text{t}^{\prime } & \text{s}^{\prime } & \text{v}%
\end{array}%
\right] .$

By the long homology sequence, there is an exact sequence: *)

\begin{center}
...$\rightarrow $H$^{i+1}$($W^{\circ }$)$\rightarrow $H$^{i}$($Y^{\circ }$)$%
\overset{H^{i}(\mu )}{\rightarrow }$H$^{i}$($K^{\circ }$)$\oplus $H$^{i}$($%
L^{\circ }$)$\overset{H^{i}(\upsilon )}{\rightarrow }$H$^{i}$($W^{\circ }$)$%
\rightarrow $H$^{i-1}$($Y^{\circ }$)$\rightarrow $...
\end{center}

Since $\pi $ is an exact functor, for any complex $\pi $H$^{i}$($X^{\circ }$)%
$\cong $H$^{i}$($\pi X^{\circ }$) and the exact sequence *) induces an exact
sequence: **)

\begin{center}
...$\rightarrow \pi $H$^{i+1}$($W^{\circ }$)$\rightarrow \pi $H$^{i}$($%
Y^{\circ }$)$\overset{\pi H^{i}(\mu )}{\rightarrow }\pi $H$^{i}$($K^{\circ }$%
)$\oplus \pi $H$^{i}$($L^{\circ }$)$\overset{\pi H^{i}(\upsilon )}{%
\rightarrow }\pi $H$^{i}$($W^{\circ }$)$\rightarrow \pi $H$^{i-1}$($Y^{\circ
}$)$\rightarrow $...
\end{center}

Which is isomorphic to the complex:

\begin{center}
...$\rightarrow $H$^{i+1}$($\pi W^{\circ }$)$\rightarrow $H$^{i}$($\pi
Y^{\circ }$)$\overset{H^{i}(\pi \mu )}{\rightarrow }$H$^{i}$($\pi K^{\circ }$%
)$\oplus $H$^{i}$($\pi L^{\circ }$)$\overset{H^{i}(\pi \upsilon )}{%
\rightarrow }$H$^{i}$($\pi W^{\circ }$)$\rightarrow $H$^{i-1}$($\pi Y^{\circ
}$)$\rightarrow $...
\end{center}

The maps H$^{i}$($\pi $s), H$^{i}$($\pi $t) are isomorphisms. Hence it
follows H$^{i}$($\pi \mu $) is for each i a splittable monomorphism and for
each $i$ there is an exact sequence:

\begin{center}
0$\rightarrow $H$^{i}$($\pi Y^{\circ }$)$\overset{H^{i}(\pi \mu )}{%
\rightarrow }$H$^{i}$($\pi K^{\circ }$)$\oplus $H$^{i}$($\pi L^{\circ }$)$%
\overset{H^{i}(\pi \upsilon )}{\rightarrow }$H$^{i}$($\pi W^{\circ }$)$%
\rightarrow $0
\end{center}

which can be embedded in a commutative exact diagram:

\begin{center}
$%
\begin{array}{ccccccc}
&  &  & 0 &  & 0 &  \\ 
&  &  & \downarrow &  & \downarrow &  \\ 
& 0 &  & \text{H}^{i}\text{(}\pi L^{\circ }\text{)} & \overset{1}{%
\rightarrow } & \text{H}^{i}\text{(}\pi L^{\circ }\text{)} &  \\ 
& \downarrow &  & \downarrow \left[ 
\begin{array}{c}
\text{0} \\ 
\text{1}%
\end{array}%
\right] &  & \downarrow \text{H}^{i}\text{(}\pi \text{s}^{\prime }\text{)} & 
\\ 
0\rightarrow & \text{H}^{i}\text{(}\pi Y^{\circ }\text{)} & \rightarrow & 
\text{H}^{i}\text{(}\pi K^{\circ }\text{)}\oplus \text{H}^{i}\text{(}\pi
L^{\circ }\text{)} & \rightarrow & \text{H}^{i}\text{(}\pi W^{\circ }\text{)}
& \rightarrow 0 \\ 
& \downarrow \text{H}^{i}\text{(}\pi \text{s)} &  & \downarrow \left[ 
\begin{array}{cc}
\text{1} & \text{0}%
\end{array}%
\right] &  & \downarrow &  \\ 
0\rightarrow & \text{H}^{i}\text{(}\pi K^{\circ }\text{)} & \overset{1}{%
\rightarrow } & \text{H}^{i}\text{(}\pi K^{\circ }\text{)} & \rightarrow & 0
&  \\ 
& \downarrow &  & \downarrow &  &  &  \\ 
& 0 &  & 0 &  &  & 
\end{array}%
$
\end{center}

By this and a similar diagram it follows H$^{i}$($\pi $s$^{\prime }$),H$^{i}$%
($\pi $t$^{\prime }$) are isomorphisms.

We have a commutative diagram in $K^{b}$(Qgr$_{B_{n}}$):

\begin{center}
\bigskip $%
\begin{array}{ccccccccc}
&  &  &  &  &  &  &  &  \\ 
&  &  &  & \pi W^{\circ } &  &  &  &  \\ 
&  &  & \pi \text{t}^{\prime }\nearrow &  & \nwarrow \pi \text{s}^{\prime }
&  &  &  \\ 
&  & \pi K^{\circ } &  &  &  & \pi L^{\circ } &  &  \\ 
& \pi \text{f}\nearrow &  & \nwarrow \pi s &  & \pi t\nearrow &  & \nwarrow
\pi \text{g} &  \\ 
\pi X^{\circ } &  &  &  & \pi Y^{\circ } &  &  &  & \pi Z^{\circ }%
\end{array}%
$
\end{center}

Then $\theta $(($\pi $g)$^{-1}\pi $t($\pi $s)$^{-1}\pi $f)=D$^{b}$($\psi $)((%
$\pi $s$^{\prime }\pi $g)$^{-1}\pi $t$^{\prime }\pi $f))=(s$_{Z}^{\prime }$g$%
_{Z}$)$^{-1}$t$_{Z}^{\prime }$f$_{Z}$=0.

But s$_{Z}^{\prime }$, g$_{Z}$, t$_{Z}^{\prime }$ are isomorphisms in D$^{b}$%
(gr$_{B_{z}}$). It follows f$_{Z}$=0 in D$^{b}$(gr$_{B_{z}}$).

Then there is a quasi isomorphism of complexes $\upsilon :\overset{\wedge }{N%
}^{\circ }\rightarrow X_{Z}^{\circ }$ such that f$_{Z}\upsilon $ is
homotopic to zero. By Lemma 3, there is a bounded complex $N^{\circ }$ of $%
B_{n}$-modules and a map $\nu :N^{\circ }\rightarrow X^{\circ }$such that $%
N_{Z}^{\circ }\cong \overset{\wedge }{N}^{\circ }$and $\nu _{Z}$ can be
identified with $\upsilon .$

According to Lemma 3. there is an integer k$\geq $0 such that the
composition of maps Z$^{k}N^{\circ }\overset{res\nu }{\rightarrow }X^{\circ }%
\overset{f}{\rightarrow }K^{\circ }$ is homotopic to zero and (res$\nu $)$%
_{Z}$=$\nu _{Z}$ is a quasi isomorphism. This implies res$\nu \in \Psi $($%
\mathcal{T}_{2}$) and $\pi $f=0 in D$^{b}$(Qgr$_{B_{n}}$)/$\mathcal{T}$.

Therefore $\pi $g)$^{-1}\pi $t($\pi $s)$^{-1}\pi $f=0 in D$^{b}$(Qgr$%
_{B_{n}} $)/$\mathcal{T}.$

ii) $\theta $ is full.

Let

\begin{center}
$%
\begin{array}{ccccc}
&  & K_{Z}^{\circ } &  &  \\ 
& \overset{\wedge }{\text{s}}\swarrow &  & \searrow \overset{\wedge }{\text{f%
}} &  \\ 
X_{Z}^{\circ } &  &  &  & Y_{Z}^{\circ }%
\end{array}%
$
\end{center}

be a map in D$^{b}$(gr$_{(B_{n})_{Z}}$). By lemma 1, there exists a complex:%
\newline
$N^{\circ }$:0$\rightarrow $Z$^{k}K^{0}\overset{\overset{\wedge }{d_{0}}}{%
\rightarrow }$ Z$^{k}K^{1}\oplus $t$_{Z}$($Y^{1}$)$\overset{\overset{\wedge }%
{d_{1}}}{\rightarrow }$ Z$^{k}K^{2}\oplus $t$_{Z}$($Y^{2}$)...Z$%
^{k}K^{n-1}\oplus $t$_{Z}$($Y^{n-1}$)$\overset{\overset{\wedge }{d_{\ell -1}}%
}{\rightarrow }$Z$^{k}K^{n}\rightarrow $0, where the differentials are of
the form: $\overset{\wedge }{\text{d}_{0}}$ = $\left[ 
\begin{array}{c}
\text{d} \\ 
\text{s}_{0}%
\end{array}%
\right] $, $\overset{\wedge }{\text{d}_{i}}$ =$\left[ 
\begin{array}{cc}
\text{d} & \text{0} \\ 
\text{s}_{i} & \text{d}^{\prime }%
\end{array}%
\right] $, $\overset{\wedge }{\text{d}_{\ell -1}}$=$\left[ 
\begin{array}{cc}
\text{d} & \text{0}%
\end{array}%
\right] $ and a map f:$N^{\circ }\rightarrow Y^{\circ }$such that $%
N_{Z}^{\circ }\cong K_{Z}^{\circ }$ and f$_{Z}$= $\overset{\wedge }{\text{f}}
$. Changing $N^{\circ }$ for $K^{\circ }$ we may assume $\overset{\wedge }{%
\text{f}}$ is a localized map f$_{Z}$ and get a roof:

\begin{center}
$%
\begin{array}{ccccc}
&  & N_{Z}^{\circ } &  &  \\ 
& \overset{\wedge }{\text{s}}\swarrow &  & \searrow \text{f}_{Z} &  \\ 
X_{Z}^{\circ } &  &  &  & Y_{Z}^{\circ }%
\end{array}%
.$
\end{center}

We now lift $\overset{\wedge }{\text{s}}$ to a map of complexes s:$\overset{%
\wedge }{N}^{\circ }\rightarrow X^{\circ }$:

\begin{center}
$%
\begin{array}{cccccccccc}
\text{0}\rightarrow & \text{Z}^{k}\text{N}^{0} & \overset{\overset{\wedge }{d%
}_{0}}{\rightarrow } & \text{Z}^{k}\text{N}^{1}\oplus \text{t}_{Z}\text{(X}%
^{1}\text{)} & \overset{\overset{\wedge }{d}_{1}}{\rightarrow } & \text{Z}%
^{k}\text{N}^{2}\oplus \text{t}_{Z}\text{(X}^{2}\text{)} & \text{...} & 
\overset{\overset{\wedge }{d_{m-1}}}{\rightarrow } & \text{Z}^{k}\text{N}^{m}
& \rightarrow \text{0} \\ 
\text{s:} & \downarrow \text{s}_{0} &  & \downarrow \text{s}_{1} &  & 
\downarrow \text{s}_{2} &  &  & \downarrow \text{s}_{m} &  \\ 
\text{0}\rightarrow & \text{X}^{0} & \overset{d}{\rightarrow } & \text{X}^{1}
& \overset{d}{\rightarrow } & \text{X}^{2} & \text{...} & \overset{d}{%
\rightarrow } & \text{X}^{m} & \rightarrow \text{0}%
\end{array}%
$
\end{center}

with s$_{z}$=$\overset{\wedge }{\text{s}}$.

We have a commutative diagram:\newline
$%
\begin{array}{cccccccccc}
\text{0}\rightarrow & \text{Z}^{k}\text{N}^{0} & \overset{\overset{\wedge }{d%
}_{0}}{\rightarrow } & \text{Z}^{k}\text{N}^{1}\oplus \text{t}_{Z}\text{(X}%
^{1}\text{)} & \overset{\overset{\wedge }{d}_{1}}{\rightarrow } & \text{Z}%
^{k}\text{N}^{2}\oplus \text{t}_{Z}\text{(X}^{2}\text{)} & \text{...} & 
\overset{\overset{\wedge }{d_{m-1}}}{\rightarrow } & \text{Z}^{k}\text{N}^{m}
& \rightarrow \text{0} \\ 
\text{(10}) & \downarrow \text{1} &  & \downarrow \text{(10)} &  & 
\downarrow \text{(10)} &  &  & \downarrow \text{1} &  \\ 
\text{0}\rightarrow & \text{Z}^{k}\text{N}^{0} & \overset{d}{\rightarrow } & 
\text{Z}^{k}\text{N}^{1} & \overset{d}{\rightarrow } & \text{Z}^{k}\text{N}%
^{2} & \text{...} & \overset{d}{\rightarrow } & \text{Z}^{k}\text{N}^{m} & 
\rightarrow \text{0}%
\end{array}%
$

We obtain the following roof:

\begin{center}
$%
\begin{array}{ccccccc}
&  & \overset{\wedge }{N}^{\circ } &  &  &  &  \\ 
& \text{{\small s}}\swarrow &  & \searrow \text{{\small (10)}} &  &  &  \\ 
X^{\circ } &  &  &  & \text{Z}^{k}N^{\circ } &  &  \\ 
&  &  &  &  & \searrow \text{{\small f}} &  \\ 
&  &  &  &  &  & Y^{\circ }%
\end{array}%
$
\end{center}

Localizing we obtain:

\begin{center}
$%
\begin{array}{ccccccc}
&  & \overset{\wedge }{N_{Z}}^{\circ } &  &  &  &  \\ 
& \text{{\small s}}_{Z}\swarrow &  & \searrow \text{{\small (10)}}_{Z} &  & 
&  \\ 
X_{Z}^{\circ } &  &  &  & \text{Z}^{k}N_{Z}^{\circ } &  &  \\ 
&  &  &  &  & \searrow \text{{\small f}}_{Z} &  \\ 
&  &  &  &  &  & Y_{Z}^{\circ }%
\end{array}%
$
\end{center}

with $\overset{\wedge }{N_{Z}}^{\circ }\overset{{\small (10)}_{Z}}{%
\rightarrow }$Z$^{k}N_{Z}^{\circ }\cong N_{Z}^{\circ }$ isomorphisms, s$_{Z}$%
=$\overset{\wedge }{\text{s}}$, f$_{z}$=$\overset{\wedge }{\text{f}}$.

We have proved $\theta $ is full.

\section{The category of $\mathcal{T}$-local objects.}

Let $\mathcal{F}$ be the full subcategory of D$^{b}$(Qgr$_{B_{n}}$)
consisting of $\mathcal{T}$-local objects, this is: $\mathcal{F}$=\{$%
X^{\circ }\in $D$^{b}$(Qgr$_{B_{n}}$)$\mid $Hom$_{D^{b}(Qgr_{B_{n}})}$($%
\mathcal{T}$,$X^{\circ }$)=0\}.

According to [Mi], Prop. 9.8, for each $Y^{\circ }\in $D$^{b}$(Qgr$_{B_{n}}$%
) and $X^{\circ }\in \mathcal{F}$, \linebreak\ Hom$_{D^{b}(Qgr_{B_{n}})}$($%
Y^{\circ }$,$X^{\circ }$)=Hom$_{D^{b}(Qgr_{B_{n}})/\mathcal{T}}$($QY^{\circ
} $,$QX^{\circ }$)$\cong $Hom$_{D^{b}(gr_{(B_{n})_{Z}})}$($\psi Y^{\circ }$,$%
\psi X^{\circ }$).

In particular there is a full embedding of $\mathcal{F}$ in D$^{b}$(gr$%
_{(B_{n})_{Z}}$).

According to [MM] and [MS] there is a duality of triangulated
categories\linebreak\ $\overline{\phi }$ : \underline{\text{gr}}$%
_{_{B_{n}^{!op}}}\rightarrow $D$^{b}$(Qgr$_{B_{n}}$) induced by the duality $%
\phi :$gr$_{_{_{B_{n}^{!op}}}}\rightarrow \mathcal{LC}$P$_{B_{n}}$, with $%
\mathcal{LC}$P$_{B_{n}}$ the category of linear complexes of graded
projective $B_{n}$-modules. If $M=\underset{i\geq k_{0}}{\oplus }M_{i}$ is a
graded $B_{n}^{!op}$-module, then $\phi $($M$) is a complex of the form:%
\newline
D(M)$\otimes _{B_{0}}B_{n}:\rightarrow $...D($M_{k_{0}+n}$)$\otimes
_{B_{0}}B_{n}$[-k$_{0}$-n]$\rightarrow D$(M$_{k_{0}+n-1}$)$\otimes
_{B_{0}}B_{n}$[-k$_{0}$-n+1]$\rightarrow $...

D($M_{k_{0}+1}$)$\otimes _{B_{0}}B_{n}$[-k$_{0}$-1]$\rightarrow $D($%
M_{k_{0}} $)$\otimes _{B_{0}}B_{n}$[-k$_{0}$]$\rightarrow $0$.$

$\overline{\phi }$($M$) is the complex:\newline
$\pi ($D(M)$\otimes _{B_{0}}B_{n}$)$:\rightarrow $...$\pi ($D($M_{k_{0}+n}$)$%
\otimes _{B_{0}}B_{n}$[-k$_{0}$-n])$\rightarrow \pi (D$(M$_{k_{0}+n-1}$)$%
\otimes _{B_{0}}B_{n}$[-k$_{0}$-n+1])\linebreak $\rightarrow $...$\pi ($D($%
M_{k_{0}+1}$)$\otimes _{B_{0}}B_{n}$[-k$_{0}$-1])$\rightarrow \pi ($D($%
M_{k_{0}}$)$\otimes _{B_{0}}B_{n}$[-k$_{0}$])$\rightarrow $0$.$

If we compose with the usual duality we obtain an equivalence of
triangulated categories: $\overline{\phi }$D: \underline{gr}$%
_{_{B_{n}^{!op}}}\rightarrow $D$^{b}$(Qgr$_{B_{n}}$).

Under the duality $\overline{\phi }$ there is a pair ($\mathcal{F}^{\prime }$%
,$\mathcal{T}^{\prime }$) such that $\mathcal{F}^{\prime }\rightarrow 
\mathcal{F}$ and $\mathcal{T}^{\prime }\rightarrow \mathcal{T}$ corresponds
to the pair ($\mathcal{T}$,$\mathcal{F}$).

We want to characterize the subcategories $\mathcal{F}^{\prime },\mathcal{T}%
^{\prime }$of \underline{gr}$_{_{B_{n}^{!}}}$.

We shall start by recalling some properties of the finitely generated graded 
$B_{n}$-modules.

The algebra $B_{n}$ is a Koszul algebra of finite global dimension, under
such conditions, for any finitely generated graded $B_{n}$-module $M$ there
is a truncation $M_{\geq k}$ such that $M_{\geq k}$[k] is Koszul [M2. But in
Qgr$_{B_{n}}$ the objects $\pi M$ and $\pi M_{\geq k}$ are isomorphic, hence
we can consider only Koszul $B_{n}$-modules and their shifts. Assume $M$ is
finitely generated but of infinite dimension over $\Bbbk $. The torsion part
t($M$) is finite dimensional over $\Bbbk $, hence there is a torsion free
truncation $M_{\geq k}$ of $M$, so we may assume $M$ torsion free and Koszul.

Let%
\'{}%
s suppose $M$ is of Z-torsion.There exists an integer n such that Z$%
^{n-1}M\neq $0 and Z$^{n}M$=0. There is a filtration $M\supset $Z$M\supset $Z%
$^{2}M$...$\supset $Z $^{n-1}M\supset 0$. Since $Z$ is an element of degree
one (Z$M$)$_{i}$=Z$M_{i-1}$, which implies (Z$^{j}M$)$_{\geq k}$=Z$^{j}$($M$ 
$_{\geq k-j}$).

Truncation of Koszul is Koszul and we can take large enough truncation in
order to have (Z$^{j}M$)$_{\geq k}$ Koszul for all $j$. Changing $M$ for $%
M_{\geq k}$ we may assume all Z$^{j}M$ are Koszul. [GM1],[GM2].

There is a commutative exact diagram:

\begin{center}
$%
\begin{array}{ccccccccc}
&  &  & 0 &  & 0 &  &  &  \\ 
&  &  & \downarrow &  & \downarrow &  &  &  \\ 
&  & 0\rightarrow & \Omega \text{(}M\text{)} & \rightarrow & \Omega \text{(}M%
\text{/Z}M\text{)} & \rightarrow & \text{Z}M & \rightarrow 0 \\ 
&  &  & \downarrow &  & \downarrow &  &  &  \\ 
&  &  & P & \overset{1}{\rightarrow } & P &  &  &  \\ 
&  &  & \downarrow &  & \downarrow &  &  &  \\ 
0\rightarrow & \text{Z}M & \rightarrow & M & \rightarrow & M\text{/Z}M & 
\rightarrow 0 &  &  \\ 
&  &  & \downarrow &  & \downarrow &  &  &  \\ 
&  &  & 0 &  & 0 &  &  & 
\end{array}%
$
\end{center}

the modules $\Omega $($M$), Z$M$ are Koszul generated in the same degree, it
follows $M$/Z$M$ is Koszul and for any integer k$\geq $1 there is an exact
sequence:

$%
\begin{array}{ccccccc}
0\rightarrow & \Omega ^{k}\text{(}M\text{)} & \rightarrow & \Omega ^{k}\text{%
(}M\text{/Z}M\text{)} & \rightarrow & \Omega ^{k-1}\text{(Z}M\text{)} & 
\rightarrow 0%
\end{array}%
$. By $[$GM1$]$ there is an exact sequence:

$0\rightarrow $Hom$_{B_{n}}$($\Omega ^{k-1}$(Z$M$),$B_{_{n}0}$)$\rightarrow $%
Hom$_{B_{n}}$($\Omega ^{k}$($M$/Z$M$),$B_{_{n}0}$)$\rightarrow $

Hom$_{B_{n}}$($\Omega ^{k}$($M$),$B_{_{n}0}$)$\rightarrow 0$ or an exact
sequence:

*) $0\rightarrow $Ext$_{B_{n}}^{k-1}$(Z$M$,$B_{_{n}0}$)$\rightarrow $Ext$%
_{B_{n}}^{k}$($M$/Z$M$,$B_{_{n}0}$)$\rightarrow $Ext$_{B_{n}}^{k}$($M$,$%
B_{_{n}0}$)$\rightarrow 0$.

We will denote by $F_{B_{n}}$($N$)=$\underset{k\geq 0}{\oplus }$Ext$%
_{B_{n}}^{k}$($N$,$B_{_{n}0}$) the Koszul duality functor \linebreak $%
F_{B_{n}}$:$K_{B_{n}}\rightarrow K_{B_{n}^{!}}$ .

Adding all sequences *) we obtain an exact sequence:

$0\rightarrow F_{B_{n}}$(Z$M$)[-1]$\rightarrow F_{B_{n}}$($M$/Z$M$)$%
\rightarrow F_{B_{n}}$($M$)$\rightarrow 0$

We can apply the same argument to any module $Z^{j}M$ to get an exact
sequence:

$0\rightarrow F_{B_{n}}$(Z$^{j+1}M$)[-j-1]$\rightarrow F_{B_{n}}$(Z$^{j}M$/Z$%
^{j+1}M$)[-j]$\rightarrow F_{B_{n}}$(Z$^{j}M$)[-j]$\rightarrow $0.

Gluing all short exact sequences we obtain a long exact sequence of Koszul
up to shifting $B_{n}^{!}$-modules:

**) $0\rightarrow F_{B_{n}}$(Z$^{n-1}M$)[-n+1]$\rightarrow F_{B_{n}}$(Z$%
^{n-2}M$/Z$^{n-1}M$)[-n+2]...$\rightarrow $

$F_{B_{n}}$($M$/Z$M$)$\rightarrow F_{B_{n}}$($M$)$\rightarrow 0$

It will be enough to study non semisimple Koszul $B_{n}$-modules $N$ such
that Z$N=0$. They can be considered as $C_{n}$-modules.

We have the following commutative exact diagram:

\begin{center}
$%
\begin{array}{ccccccc}
& 0 &  & 0 &  &  &  \\ 
& \downarrow &  & \downarrow &  &  &  \\ 
& \text{Z}B_{n}^{n_{0}} & \overset{\text{1}}{\rightarrow } & \text{Z}%
B_{n}^{n_{0}} &  &  &  \\ 
& \downarrow &  & \downarrow &  &  &  \\ 
0\rightarrow & \Omega _{B}\text{(}N\text{)} & \rightarrow & B_{n}^{n_{0}} & 
\rightarrow & N & \rightarrow 0 \\ 
& \downarrow &  & \downarrow &  & \downarrow \text{1} &  \\ 
0\rightarrow & \Omega _{C}\text{(}N\text{)} & \rightarrow & C_{n}^{n_{0}} & 
\rightarrow & N & \rightarrow 0 \\ 
& \downarrow &  & \downarrow &  &  &  \\ 
& 0 &  & 0 &  &  & 
\end{array}%
$
\end{center}

The algebra $B_{n}$ is an integral domain and in consequence the free $B_{n}$%
-modules are torsion free and Z$B_{n}^{n_{0}}$is isomorphic to $%
B_{n}^{n_{0}} $[-1].

The exact sequence: $0\rightarrow $Z$B_{n}^{n_{0}}\rightarrow \Omega _{B}$($%
N $)$\rightarrow \Omega _{C}$($N$)$\rightarrow 0$ consists of graded modules
generated in degree one and the first two term are Koszul, by [GM] this
implies $\Omega _{C}$($N)$ is Koszul as $B_{n}$-module.

There is a commutative exact diagram:

\begin{center}
$%
\begin{array}{ccccccc}
&  &  & 0 &  & 0 &  \\ 
&  &  & \downarrow &  & \downarrow &  \\ 
& 0 & \rightarrow & B_{n}^{n_{0}}\text{[-1]} & \rightarrow & \text{Z}%
B_{n}^{n_{0}} & \rightarrow 0 \\ 
& \downarrow &  & \downarrow &  & \downarrow &  \\ 
0\rightarrow & \Omega _{B}^{2}\text{(}N\text{)} & \rightarrow & 
B_{n}^{n_{0}+n_{1}}\text{[-1]} & \rightarrow & \Omega _{B}\text{(}N\text{)}
& \rightarrow 0 \\ 
& \downarrow &  & \downarrow &  & \downarrow &  \\ 
0\rightarrow & \Omega _{B}\Omega _{C}\text{(}N\text{)} & \rightarrow & 
B_{n}^{n_{1}}\text{[-1]} & \rightarrow & \Omega _{C}\text{(}N\text{)} & 
\rightarrow 0 \\ 
& \downarrow &  & \downarrow &  & \downarrow &  \\ 
& 0 &  & 0 &  & 0 & 
\end{array}%
$
\end{center}

In particular $\Omega _{B}^{2}$($N$)$\cong \Omega _{B}\Omega _{C}$($N$).

Since $\Omega _{C}$($N$)$\subset C_{n}^{n_{0}}$it is a $C_{n}$-module and we
have the following commutative exact diagram:

\begin{center}
$%
\begin{array}{ccccccc}
& 0 &  & 0 &  &  &  \\ 
& \downarrow &  & \downarrow &  &  &  \\ 
& \text{Z}B_{n}^{n_{1}}\text{[-1]} & \overset{1}{\rightarrow } & \text{Z}%
B_{n}^{n_{1}}\text{[-1]} &  &  &  \\ 
& \downarrow &  & \downarrow &  &  &  \\ 
0\rightarrow & \Omega _{B}\Omega _{C}\text{(}N\text{)} & \rightarrow & 
B_{n}^{n_{1}}\text{[-1]} & \rightarrow & \Omega _{C}\text{(}N\text{)} & 
\rightarrow 0 \\ 
& \downarrow &  & \downarrow &  & \downarrow 1 &  \\ 
0\rightarrow & \Omega _{C}^{2}\text{(}N\text{)} & \rightarrow & C_{n}^{n_{1}}
& \rightarrow & \Omega _{C}\text{(}N\text{)} & \rightarrow 0 \\ 
& \downarrow &  & \downarrow &  &  &  \\ 
& 0 &  & 0 &  &  & 
\end{array}%
$
\end{center}

and an exact sequence: $0\rightarrow B_{n}^{n_{1}}$[-2]$\rightarrow \Omega
_{B}^{2}$($N$)$\rightarrow \Omega _{C}^{2}$($N$)$\rightarrow 0$.

\bigskip In general there exist exact sequences:

$0\rightarrow B_{n}^{n_{k-1}}$[-k]$\rightarrow \Omega _{B}^{k}$($N$)$%
\rightarrow \Omega _{C}^{k}$($N$)$\rightarrow 0$.

which induce exact sequences:\newline
$0\rightarrow $Hom$_{B_{n}}$($\Omega _{C}^{k}$(N),B$_{_{n}0}$)$\rightarrow $%
Hom$_{B_{n}}$($\Omega _{B}^{k}$(N),B$_{_{n}0}$)$\rightarrow $Hom$_{B_{n}}$(B$%
_{n}^{n_{k-1}}$[-k],B$_{_{n}0}$)$\rightarrow 0.$

The module $\Omega _{C}^{k}$($N$) is annihilated by $Z$ which implies $%
J_{B}\Omega _{C}^{k}$($N$)=$J_{C}\Omega _{C}^{k}$($N$). The module $%
B_{_{n}0}\cong C_{_{n}0}\cong \Bbbk $

Hom$_{B_{n}}$($\Omega _{C}^{k}$($N$),$B_{_{n}0}$)$\cong $Hom$_{B_{n}0}$($%
\Omega _{C}^{k}$($N$)/$J_{B}\Omega _{C}^{k}$($N$),$B_{_{n}0}$)$\cong $

Hom$_{C_{n}0}$($\Omega _{C}^{k}$($N$)/$J_{c}\Omega _{C}^{k}$($N$),$C_{_{n}0}$%
)$\cong $Hom$_{C_{n}}$($\Omega _{C}^{k}$($N$),$C_{_{n}0}$)$\cong $Ext$%
_{C_{n}}^{k}$($N$,$C_{_{n}0}$).

We then have an exact sequence: *) $0\rightarrow F_{C_{n}}$($N$)$\overset{%
\alpha }{\rightarrow }$ $F_{B_{n}}$($N$)$\rightarrow \overset{m}{\underset{%
k=1}{\oplus }}S^{n_{k-1}}$[k]$\rightarrow 0$

\begin{lemma}
The map $\alpha $ is a morphism of $C_{n}^{!}$-modules.
\end{lemma}

\begin{proof}
Let $x$ be an element of Ext$_{C_{n}}^{k}$($\Bbbk $,$\Bbbk $) and y$\in $Ext$%
_{C_{n}}^{k}$($N$,$\Bbbk $) we want to prove $\alpha $(xy)=x$\alpha $(y).

The element x is an extension: $0\rightarrow \Bbbk \rightarrow E\rightarrow
\Bbbk \rightarrow 0$ and $y:$ $0\rightarrow \Bbbk \rightarrow V\rightarrow
\Omega _{C}^{k-1}$($N$)$\rightarrow 0$, the induced map $f$ given below
corresponds to y:

$%
\begin{array}{ccccccc}
0\rightarrow & \Omega _{C}^{k}\text{(}N\text{)} & \rightarrow & 
C_{n}^{n_{k-1}} & \rightarrow & \Omega _{C}^{k-1}\text{(}N\text{)} & 
\rightarrow 0 \\ 
& \downarrow \text{f} &  & \downarrow &  & \downarrow \text{1} &  \\ 
0\rightarrow & \Bbbk & \rightarrow & V & \rightarrow & \Omega _{C}^{k-1}%
\text{(}N\text{)} & \rightarrow 0%
\end{array}%
$

Consider the following pull back:

$%
\begin{array}{ccccccc}
0\rightarrow & \Bbbk & \rightarrow & L & \rightarrow & \Omega _{C}^{k}\text{(%
}N\text{)} & \rightarrow 0 \\ 
& \downarrow \text{1} &  & \downarrow &  & \downarrow \text{f} &  \\ 
0\rightarrow & \Bbbk & \rightarrow & E & \rightarrow & \Bbbk & \rightarrow 0%
\end{array}%
$

The exact sequence: $0\rightarrow B_{n}^{n_{k-1}}$[-k]$\rightarrow \Omega
_{B}^{k}$($N$)$\overset{\pi _{k}}{\rightarrow }\Omega _{C}^{k}$($N$)$%
\rightarrow 0$ induces a pull back of $B_{n}$-modules:

$%
\begin{array}{ccccccc}
0\rightarrow & \Bbbk & \rightarrow & W & \rightarrow & \Omega _{B}^{k}\text{(%
}N\text{)} & \rightarrow 0 \\ 
& \downarrow \text{1} &  & \downarrow &  & \downarrow \pi _{k} &  \\ 
0\rightarrow & \Bbbk & \rightarrow & L & \rightarrow & \Omega _{C}^{k}\text{(%
}N\text{)} & \rightarrow 0%
\end{array}%
$

It was proved above the existence of commutative exact diagrams:

$%
\begin{array}{ccccccc}
&  &  & 0 &  & 0 &  \\ 
&  &  & \downarrow &  & \downarrow &  \\ 
& 0 & \rightarrow & B_{n}^{n_{k-2}} & \rightarrow & B_{n}^{n_{k-2}} & 
\rightarrow 0 \\ 
& \downarrow &  & \downarrow &  & \downarrow &  \\ 
0\rightarrow & \Omega _{B}^{k}\text{(}N\text{)} & \rightarrow & 
B_{n}^{n_{k-2}+n_{k-1}} & \rightarrow & \Omega _{B}^{k-1}\text{(}N\text{)} & 
\rightarrow 0 \\ 
& \downarrow &  & \downarrow &  & \downarrow &  \\ 
0\rightarrow & \Omega _{B}^{k}\text{(}N\text{)} & \rightarrow & 
B_{n}^{n_{k-1}} & \rightarrow & \Omega _{C}^{k-1}\text{(}N\text{)} & 
\rightarrow 0 \\ 
& \downarrow &  & \downarrow &  & \downarrow &  \\ 
& 0 &  & 0 &  & 0 & 
\end{array}%
$

and

$%
\begin{array}{ccccccc}
0\rightarrow & \Omega _{B}^{k}\text{(}N\text{)} & \rightarrow & 
B_{n}^{n_{k-1}} & \rightarrow & \Omega _{C}^{k-1}\text{(}N\text{)} & 
\rightarrow 0 \\ 
& \downarrow &  & \downarrow &  & \downarrow \text{1} &  \\ 
0\rightarrow & \Omega _{C}^{k}\text{(}N\text{)} & \rightarrow & 
C_{n}^{n_{k-1}} & \rightarrow & \Omega _{C}^{k-1}\text{(}N\text{)} & 
\rightarrow 0 \\ 
& \downarrow &  & \downarrow &  &  &  \\ 
& 0 &  & 0 &  &  & 
\end{array}%
$

Gluing diagrams we obtain a commutative diagram with exact rows:\newline
$%
\begin{array}{cccccccccc}
& \text{0}\rightarrow  & \Omega _{B}^{k+1}\text{(N)} & \rightarrow  & \text{B%
}_{n}^{n_{k}+n_{k-1}} & \rightarrow  & \text{B}_{n}^{n_{k-2}+n_{k-1}} & 
\rightarrow  & \Omega _{B}^{k-1}\text{(N)} & \rightarrow \text{0} \\ 
&  & \varphi \downarrow  &  & \downarrow  & \searrow  & \downarrow  &  & 
\downarrow  &  \\ 
\alpha \text{(xy):} & \text{0}\rightarrow  & \Bbbk  & \rightarrow  & \text{W}
& \rightarrow  & \text{B}_{n}^{n_{k-1}} & \rightarrow  & \Omega _{C}^{k-1}%
\text{(N)} & \rightarrow \text{0} \\ 
&  & \downarrow \text{1} &  & \downarrow  &  & \downarrow  &  & \downarrow 
\text{1} &  \\ 
\text{xy:} & \text{0}\rightarrow  & \Bbbk  & \rightarrow  & \text{L} & 
\rightarrow  & \text{C}_{n}^{n_{k-1}} & \rightarrow  & \Omega _{C}^{k-1}%
\text{(N)} & \rightarrow \text{0} \\ 
&  & \downarrow \text{1} &  & \downarrow  &  & \downarrow  &  & \downarrow 
\text{1} &  \\ 
\text{xy:} & \text{0}\rightarrow  & \Bbbk  & \rightarrow  & \text{E} & 
\rightarrow  & \text{V} & \rightarrow  & \Omega _{C}^{k-1}\text{(N)} & 
\rightarrow \text{0}%
\end{array}%
$

$\alpha (xy)=\phi \in Ext_{B_{n}}^{k+1}(N,K)$ .

In the other hand we have the following commutative diagram with exact rows:

$%
\begin{array}{ccccccc}
\text{0}\rightarrow & \Omega _{B}^{k+1}\text{(N)} & \rightarrow & \text{B}%
_{n}^{n_{k}+n_{k-1}} & \rightarrow & \Omega _{B}^{k}\text{(N)} & \rightarrow 
\text{0} \\ 
& \downarrow \text{1} &  & \downarrow &  & \downarrow \pi _{k} &  \\ 
\text{0}\rightarrow & \Omega _{B}^{k+1}\text{(N)} & \rightarrow & \text{B}%
_{n}^{n_{k}} & \rightarrow & \Omega _{C}^{k}\text{(N)} & \rightarrow \text{0}
\\ 
& \downarrow \pi _{k+1} &  & \downarrow &  & \downarrow \text{1} &  \\ 
\text{0}\rightarrow & \Omega _{C}^{k+1}\text{(N)} & \rightarrow & \text{C}%
_{n}^{n_{k}} & \rightarrow & \Omega _{C}^{k}\text{(N)} & \rightarrow \text{0}
\\ 
& \downarrow \varphi ^{\prime } &  & \downarrow &  & \downarrow \text{1} & 
\\ 
\text{0}\rightarrow & \Bbbk & \rightarrow & \text{L} & \rightarrow & \Omega
_{C}^{k}\text{(N)} & \rightarrow \text{0} \\ 
& \downarrow \text{1} &  & \downarrow &  & \downarrow \text{f} &  \\ 
\text{0}\rightarrow & \Bbbk & \rightarrow & \text{E} & \rightarrow & \Bbbk & 
\rightarrow \text{0}%
\end{array}%
$

Gluing diagrams we obtain the following commutative exact diagrams:

$%
\begin{array}{ccccccc}
\text{0}\rightarrow & \Omega _{B}^{k+1}\text{(N)} & \rightarrow & \text{B}%
_{n}^{n_{k}+n_{k-1}} & \rightarrow & \Omega _{B}^{k}\text{(N)} & \rightarrow 
\text{0} \\ 
& \downarrow \varphi ^{\prime }\pi _{k+1} &  & \downarrow &  & \downarrow 
\text{f}\pi _{k} &  \\ 
\text{0}\rightarrow & \Bbbk & \rightarrow & \text{E} & \rightarrow & \Bbbk & 
\rightarrow \text{0}%
\end{array}%
$

and

$%
\begin{array}{ccccccc}
\text{0}\rightarrow & \Omega _{B}^{k+1}\text{(N)} & \rightarrow & \text{B}%
_{n}^{n_{k}+n_{k-1}} & \rightarrow & \Omega _{B}^{k}\text{(N)} & \rightarrow 
\text{0} \\ 
& \downarrow \varphi &  & \downarrow &  & \downarrow \text{f}\pi _{k} &  \\ 
\text{0}\rightarrow & \Bbbk & \rightarrow & \text{E} & \rightarrow & \Bbbk & 
\rightarrow \text{0}%
\end{array}%
$

The map $\varphi ^{\prime }$ corresponds with xf$\in $Ext$_{C_{n}}^{k+1}$($N$%
,$\Bbbk )$, $\alpha $(y)=$\pi _{k}$f, x$\alpha $(y)=$\varphi $.

Then we have: $\alpha $(xy)=$\alpha $(xf)=$\varphi ^{\prime }\pi _{k+1}$=$%
\Omega $(f$\pi _{k}$)=$\varphi $=x$\alpha $(y).
\end{proof}

\begin{lemma}
There is an isomorphism: $B_{n}^{!}F_{C}$($M$)=$F_{B}$($M$).
\end{lemma}

\begin{proof}
Let x be an element of Ext$_{B_{n}}^{k}$($M$,$\Bbbk $) and $\varphi $ the
corresponding morphism: $\varphi $:$\Omega _{B}^{k}$($M$)$\rightarrow \Bbbk $%
. As above there exists the following commutative exact diagram:\newline
$%
\begin{array}{ccccccc}
&  &  & 0 &  & 0 &  \\ 
&  &  & \downarrow  &  & \downarrow  &  \\ 
& 0 & \rightarrow  & B_{n}^{n_{k-2}}\text{[-k+1]} & \rightarrow  & 
B_{n}^{n_{k-2}}\text{[-k+1]} & \rightarrow 0 \\ 
& \downarrow  &  & \downarrow  &  & \downarrow  &  \\ 
0\rightarrow  & \Omega _{B}^{k}\text{(}M\text{)} & \rightarrow  & 
B_{n}^{n_{k-2}+n_{k-1}}\text{[-k+1]} & \rightarrow  & \Omega _{B}^{k-1}\text{%
(}M\text{)} & \rightarrow 0 \\ 
& \downarrow  &  & \downarrow  &  & \downarrow  &  \\ 
0\rightarrow  & \Omega _{B}\Omega _{C}^{k-1}\text{(}M\text{)} & \rightarrow 
& B_{n}^{n_{k-1}}\text{[-k+1]} & \rightarrow  & \Omega _{C}^{k-1}\text{(}M%
\text{)} & \rightarrow 0 \\ 
& \downarrow  &  & \downarrow  &  & \downarrow  &  \\ 
& 0 &  & 0 &  & 0 & 
\end{array}%
$

Since the module $\Omega _{B}^{k}$($M$) is generated in degree k, there
exists an exact sequence of graded modules generated in degree k:

$0\rightarrow \Omega _{B}^{k}$($M$)$\rightarrow JB_{n}^{n_{k-1}}$[-k+1]$%
\rightarrow J\Omega _{C}^{k-1}$($M$)$\rightarrow 0$, which in turn induces
an exact sequence: $0\rightarrow J\Omega _{B}^{k}$($M$)$\rightarrow
J^{2}B_{n}^{n_{k-1}}$[-k+1]$\rightarrow J^{2}\Omega _{C}^{k-1}$($M$)$%
\rightarrow 0$ and there exists a commutative exact diagram:\newline
$%
\begin{array}{ccccccc}
& 0 &  & 0 &  & 0 &  \\ 
& \downarrow  &  & \downarrow  &  & \downarrow  &  \\ 
\text{0}\rightarrow  & \text{J}\Omega _{B}^{k}\text{(M)} & \rightarrow  & 
\text{J}^{2}\text{B}_{n}^{n_{k-1}} & \rightarrow  & \text{J}^{2}\Omega
_{C}^{k-1}\text{(M)} & \rightarrow \text{0} \\ 
& \downarrow  &  & \downarrow  &  & \downarrow  &  \\ 
\text{0}\rightarrow  & \Omega _{B}^{k}\text{(M)} & \overset{j}{\rightarrow }
& \text{JB}_{n}^{n_{k-1}} & \rightarrow  & \text{J}\Omega _{C}^{k-1}\text{(M)%
} & \rightarrow \text{0} \\ 
& \downarrow \pi  &  & \downarrow \overline{\pi } &  & \downarrow  &  \\ 
\text{0}\rightarrow  & \Omega _{B}^{k}\text{(M)/J}\Omega _{B}^{k}\text{(M)}
& \underset{\overline{q}}{\overset{\overline{j}}{\rightleftarrows }} & \text{%
JB}_{n}^{n_{k-1}}\text{/J}^{2}\text{B}_{n}^{n_{k-1}} & \rightarrow  & \text{J%
}\Omega _{C}^{k-1}\text{(M)/J}^{2}\Omega _{C}^{k-1}\text{(M)} & \rightarrow 
\text{0} \\ 
& \downarrow  &  & \downarrow  &  & \downarrow  &  \\ 
& 0 &  & 0 &  & 0 & 
\end{array}%
$

Since $\overline{\text{q}}\overline{\text{j}}$=1, it follows $\overline{%
\text{q}}\overline{\pi }$j= $\overline{\text{q}}\overline{\text{j}}\pi $=$%
\pi $.

Being $\Bbbk $ semisimple, the map $\varphi $ factors as follows:

$%
\begin{array}{ccccc}
\Omega _{B}^{k}\text{(}M\text{)} &  & \overset{\varphi }{\rightarrow } &  & 
\Bbbk \\ 
& \pi \searrow &  & \nearrow \text{t} &  \\ 
&  & \Omega _{B}^{k}\text{(}M\text{)/}J\Omega _{B}^{k}\text{(}M\text{)} &  & 
\end{array}%
.$

Set f=t$\overline{\text{q}}\overline{\pi }$, f: $JB_{n}^{n_{k-1}}$[-k+1]$%
\rightarrow \Bbbk $. Then fj=t$\overline{\text{q}}\overline{\pi }$j=t$\pi $=$%
\varphi .$

Consider the commutative diagram with exact rows:\newline
$%
\begin{array}{cccccccc}
& 0\rightarrow  & \Omega _{B}\Omega _{C}^{k-1}\text{(}M\text{)} & 
\rightarrow  & B_{n}^{n_{k-1}} & \rightarrow  & \Omega _{C}^{k-1}\text{(}M%
\text{)} & \rightarrow 0 \\ 
&  & \downarrow \text{j} &  & \downarrow \text{1} &  & \downarrow \rho  & 
\\ 
& 0\rightarrow  & JB^{n_{k-1}}\text{[-k+1]} & \rightarrow  & B_{n}^{n_{k-1}}
& \rightarrow  & \Omega _{C}^{k-1}\text{(}M\text{)/}J\Omega _{C}^{k-1}\text{(%
}M\text{)} & \rightarrow 0 \\ 
&  & \downarrow \text{f} &  & \downarrow  &  & \downarrow \cong  &  \\ 
x: & 0\rightarrow  & \Bbbk  & \rightarrow  & E & \rightarrow  & \underset{%
n_{k-1}}{\oplus \Bbbk } & \rightarrow 0%
\end{array}%
.$

The map corresponding to the last column is: p=$\left[ 
\begin{array}{c}
\text{p}_{1} \\ 
\overset{.}{\underset{.}{.}} \\ 
\text{p}_{n_{k-1}}%
\end{array}%
\right] $: $\Omega _{C}^{k-1}$($M$)$\rightarrow \underset{n_{k-1}}{\oplus
\Bbbk }$. Each p$_{i}$: $\Omega _{C}^{k-1}$($M$)$\rightarrow \Bbbk $
corresponds to an element of Ext$_{C_{n}}^{k-1}$($M$,$\Bbbk $).

x=(x$_{1}$,x$_{2}$...x$_{n_{k-1}}$) and each x$_{i}$ is an extension: $%
0\rightarrow \Bbbk \rightarrow E_{i}\rightarrow \Bbbk \rightarrow 0.$ Taking
pull backs:

$%
\begin{array}{cccccccc}
\text{x}_{i}\text{p}_{i}\text{:} & 0\rightarrow & \Bbbk & \rightarrow & L_{i}
& \rightarrow & \Omega _{C}^{k-1}\text{(}M\text{)} & \rightarrow 0 \\ 
&  & \downarrow \text{1} &  & \downarrow &  & \downarrow \text{p}_{i} &  \\ 
x_{i}: & 0\rightarrow & \Bbbk & \rightarrow & E_{i} & \rightarrow & \Bbbk & 
\rightarrow 0%
\end{array}%
$

where each x$_{i}$p$_{i}\in B_{n}^{!}$Ext$_{C_{n}}^{k-1}$($M$,$\Bbbk $) and
xp=$\sum $x$_{i}$p$_{i}\in B_{n}^{!}$Ext$_{C_{n}}^{k-1}$($M,\Bbbk $).

There is also the following induced diagram with exact rows:\newline
$%
\begin{array}{cccccccc}
& 0\rightarrow  & \Omega _{B}\Omega _{C}^{k-1}\text{(}M\text{)} & 
\rightarrow  & B_{n}^{n_{k-1}}\text{[k+1]} & \rightarrow  & \Omega _{C}^{k-1}%
\text{(}M\text{)} & \rightarrow 0 \\ 
&  & \downarrow \text{h} &  & \downarrow  &  & \downarrow 1 &  \\ 
& 0\rightarrow  & \Bbbk  & \rightarrow  & L & \rightarrow  & \Omega
_{C}^{k-1}\text{(}M\text{)} & \rightarrow 0 \\ 
&  & \downarrow \text{1} &  & \downarrow  &  & \downarrow \text{p} &  \\ 
x: & 0\rightarrow  & \Bbbk  & \rightarrow  & E & \rightarrow  & \underset{%
n_{k-1}}{\oplus \Bbbk } & \rightarrow 0%
\end{array}%
$

Gluing the diagrams we obtain the following commutative exact diagram:%
\newline
$%
\begin{array}{cccccccc}
& 0\rightarrow  & \Omega _{B}\Omega _{C}^{k-1}\text{(}M\text{)} & 
\rightarrow  & B_{n}^{n_{k-1}}\text{[k+1]} & \rightarrow  & \Omega _{C}^{k-1}%
\text{(}M\text{)} & \rightarrow 0 \\ 
&  & \downarrow \text{h} &  & \downarrow  &  & \downarrow \text{p} &  \\ 
x: & 0\rightarrow  & \Bbbk  & \rightarrow  & E & \rightarrow  & \underset{%
n_{k-1}}{\oplus \Bbbk } & \rightarrow 0%
\end{array}%
$

It follows h=fj=$\varphi $ up to homotopy.

We have proved $B_{n}^{!}$Ext$_{C_{n}}^{k-1}$($M$,$\Bbbk $)=Ext$_{B_{n}}^{k}$%
($M$,$\Bbbk $).

It follows by induction $B_{n}^{!}F_{C}$($M$)=$F_{B}$($M$).
\end{proof}

We can prove now the following:

\begin{proposition}
Let $M$ be a Koszul non semisimple $B_{n}$-module with Z$M$=0, $F_{B}$:$%
K_{B}\rightarrow K_{B^{!}}$, $F_{C}$:$K_{C}\rightarrow K_{C^{!}}$ Koszul
dualities. Then there is an isomorphism of $B_{n}^{!}$-modules: $B_{n}^{!}%
\underset{C^{!}}{\otimes }F_{C}$($M$)$\cong F_{B}$($M$).
\end{proposition}

\begin{proof}
We proved in the previous lemma that the map $\mu $:$B_{n}^{!}\underset{C^{!}%
}{\otimes }F_{C}$($M$)$\rightarrow F_{B}$($M$) given by multiplication is
surjective and we know that $B_{n}^{!}$=$C_{n}^{!}\oplus $Z$C_{n}^{!}$, so
there is a splittable sequence of $C_{n}^{!}$-modules: $0\rightarrow
C_{n}^{!}\rightarrow B_{n}^{!}\rightarrow $Z$C_{n}^{!}\rightarrow 0$ which
induces a commutative exact diagram:\newline
$%
\begin{array}{ccccccc}
0\rightarrow  & C_{n}^{!}\underset{C^{!}}{\otimes }F_{C}\text{(}M\text{)} & 
\rightarrow  & B_{n}^{!}\underset{C^{!}}{\otimes }F_{C}\text{(}M\text{)} & 
\rightarrow  & \text{Z}C_{n}^{!}\underset{C^{!}}{\otimes }F_{C}\text{(}M%
\text{)} & \rightarrow 0 \\ 
& \downarrow \cong  &  & \downarrow \mu  &  & \downarrow \mu ^{\prime \prime
} &  \\ 
0\rightarrow  & F_{C}\text{(}M\text{)} & \rightarrow  & F_{B}\text{(}M\text{)%
} & \rightarrow  & \overset{m}{\oplus }S^{n_{k-1}}\text{[k]} & \rightarrow 0
\\ 
&  &  & \downarrow  &  & \downarrow  &  \\ 
&  &  & 0 &  & 0 & 
\end{array}%
$

By dimensions $\mu ^{\prime \prime }$ is an isomorphism, therefore $\mu $ is
an isomorphism.
\end{proof}

It as proved in [MM], [MS] that the duality $\phi $:gr$_{_{_{B_{n}^{!}}}}%
\rightarrow \mathcal{LC}$P$_{B_{n}}$, with $\mathcal{LC}$P$_{B_{n}}$ the
category of linear complexes of graded projective $B_{n}$-modules, induces a
duality of triangulated categories $\overline{\phi }:$\underline{gr}$%
_{_{B_{n}^{!}}}\rightarrow $D$^{b}$(Qgr$_{B_{n}^{op}}$). In particular given
a complex $\pi X^{\circ }$ in D$^{b}$(Qgr$_{B_{n}^{op}}$), there is a
totally linear complex (see [MM] for definition) $Y^{\circ }$ such that $\pi
Y^{\circ }$ is isomorphic to $\pi X^{\circ }$, Moreover, $Y^{\circ }$ is
quasi isomorphic to a linear complex of projective $P^{\circ }$ and by [MS]$%
P^{\circ }=\phi $($M$). Therefore $\pi X^{\circ }$ is quasi isomorphic to $%
\pi \phi $($M$).

It was proved in [MS] that H$^{i}$($\phi $($M$))=0 for all i$\neq $0 if and
only if $M$ is Koszul and in this case if $G_{B_{n}^{!}}:K_{B_{n}^{!}}%
\rightarrow K_{B_{n}}$ is Koszul duality, then H$^{0}$($\phi $($M$))$\cong
G_{B_{n}^{!}}$($M$).

Since $B_{n}^{!}$ is a finite dimensional algebra, it follows by [MZ] that
for any finitely generated $B_{n}^{!}$-module $M$ there exists an integer k$%
\geq $0 such that $\Omega ^{k}M$ is weakly Koszul.

since $\Omega $ is the shift in the triangulated category \underline{gr}$%
_{_{B_{n}^{!}}}$and $\overline{\phi }$ is a duality it follows $\overline{%
\phi }$($\Omega ^{k}M$)$\cong \overline{\phi }$($M$)[k]. Being the category $%
\mathcal{T}$ triangulated, it is invariant under shift and $\pi \phi $($M$)$%
\in \mathcal{T}$ if and only if $\pi \phi $($\Omega ^{k}M$)$\in \mathcal{T}$.

We may assume $M$ is weakly Koszul and $M$=$\underset{i\geq 0}{\sum }M_{i}$, 
$M_{0}\neq $0. By [MZ] there exists an exact sequence: $0\rightarrow
K_{M}\rightarrow M\rightarrow L\rightarrow 0$ with $K_{M}<M_{0}>$ generated
by the degree zero part of $M$, $K_{M}$ is Koszul and $J^{j}K_{M}$=$%
J^{j}M\cap K_{M}$ for all j$>$0.

Being $\phi $ an exact functor there is an exact sequence: $0\rightarrow
\phi $($L$)$\rightarrow \phi $($M$)$\rightarrow \phi $($K_{M}$)$\rightarrow
0 $ of complexes of $B_{n}$-modules, which induces a long exact sequence:

...$\rightarrow $H$^{1}$($\phi $($L$))$\rightarrow $H$^{1}$($\phi $($M$))$%
\rightarrow $H$^{1}$($\phi $($K_{M}$))$\rightarrow $H$^{0}$($\phi $($L$))$%
\rightarrow $H$^{0}$($\phi $($M$))$\rightarrow $H$^{0}$($\phi $($K_{M}$))$%
\rightarrow 0$

where H$^{0}$($\phi $($M$))$\cong $H$^{0}$($\phi $($K_{M}$)) and H$^{i}$($%
\phi $($L$))$\cong $H$^{i}$($\phi $($M$) for all i$\neq $0. Being $K_{M}$
Koszul H$^{0}$($\phi $($K_{M}$))$\cong G_{B_{n}^{!}}$($K_{M})$ and $%
G_{B_{n}^{!}}$($K_{M}$) is of Z-torsion.

According to [MZ] there is a filtration: $M$=$U_{p}\supset U_{p-1}\supset
...U_{1}\supset U_{0}$=$K_{M}$ such that $U_{i}$/$U_{i-1}$is Koszul and $%
J^{k}U_{i}\cap U_{i-1}$=$J^{k}U_{i-1}.$

The module $L$ is weakly Koszul and it has a filtration: $L$= $%
U_{p}/U_{0}\supset U_{p-1}/U_{0}\supset $\linebreak ...$U_{1}/U_{0}$ with
factors Koszul, it follows by induction each $G_{B_{n}^{!}}$($U_{i}$/$%
U_{i-1} $)=$V_{i}$ is a Koszul $B_{n}$-module of Z-torsion.

Each $V_{i}$ has a filtration: $V_{i}\supset $Z$V_{i}\supset $Z$%
^{2}V_{i}\supset $...$\supset $Z$^{k_{i}}V_{i}\supset 0$, Z$%
^{k_{i}}V_{i}\neq $0, Z$^{k_{i}+1}V_{i}$=0. After a truncation $V_{i\geq
n_{i}}$we may assume all Z$^{j}V_{i}$ Koszul. But $V_{i\geq n_{i}}$=$%
J^{n_{i}}V_{i}\cong G_{B_{n}^{!}}$($\Omega ^{n_{i}}$($U_{i}$/$U_{i-1}$)).
Taking n=max\{n$_{i}$\} we change $M$ for $\Omega ^{n}$($M$), which is
weakly Koszul with filtration: $\Omega ^{n}M$=$\Omega ^{n}U_{p}\supset
\Omega ^{n}U_{p-1}\supset $...$\Omega ^{n}U_{1}\supset \Omega ^{n}U_{0}$.

We may assume all $Z^{j}V_{i}$ are Koszul. There exist exact sequences:

*) $0\rightarrow F_{B_{n}}$(Z$^{k_{i}}V_{i}$)[-k$_{i}$]$\rightarrow
F_{B_{n}} $(Z$^{k_{i}-1}V_{i}$/$Z^{k_{i}}V_{i}$)[-k$_{i}$+1]...$\rightarrow $

$F_{B_{n}}$($V_{i}$/Z$V_{i}$)$\rightarrow U_{i}$/$U_{i-1}\rightarrow 0$

where each $F_{B_{n}}$(Z$^{j}V_{i}$/Z$^{j+1}V_{i}$)$\cong B_{n}^{!}\underset{%
C_{n}^{!}}{\otimes }X_{ij}$ is an induced module of a Koszul $C_{n}^{!}$%
-module $X_{ij}.$

\begin{lemma}
Let $R$ be a $%
\mathbb{Z}
$graded $\Bbbk $-algebra, with $\Bbbk $ a field, $M$ a graded left $R$%
-module and $N$ a graded right $R$-module. Then $M\underset{R}{\otimes }N$
is a graded $\Bbbk $-module such that $M\underset{R}{\otimes }N$[j]$\cong $($%
M\underset{R}{\otimes }N$)[j]as graded $\Bbbk $-modules.
\end{lemma}

\begin{proof}
Recall the definition of the graded tensor product [Mac]:

Let $\psi $:$M\underset{\Bbbk }{\otimes }R\underset{\Bbbk }{\otimes }%
N\rightarrow M\underset{\Bbbk }{\otimes }N$ be the map: $\psi $(m$\otimes $r$%
\otimes $n)=mr$\otimes $n-m$\otimes $rn. Then Cok$\psi =M\underset{R}{%
\otimes }N.$

The $\Bbbk $- module $M\underset{\Bbbk }{\otimes }N$ has grading: ($M%
\underset{\Bbbk }{\otimes }N$)$_{k}$=$\underset{i+j=k}{\sum }$ $M_{i}%
\underset{\Bbbk }{\otimes }N_{j}$. It follows $M\underset{\Bbbk }{\otimes }N$
[j]$\cong $($M\underset{\Bbbk }{\otimes }N$ )[j].\newline

and there is an isomorphism of exact sequences:

$%
\begin{array}{cccccc}
M\underset{\Bbbk }{\otimes }R\underset{\Bbbk }{\otimes }\text{(}N\text{[j])}
& \rightarrow & M\underset{\Bbbk }{\otimes }\text{(}N\text{[j])} & 
\rightarrow & M\underset{R}{\otimes \text{(}}N\text{[j])} & \rightarrow 0 \\ 
\downarrow \cong &  & \downarrow \cong &  & \downarrow \cong &  \\ 
\text{(}M\underset{\Bbbk }{\otimes }R\underset{\Bbbk }{\otimes }N\text{)[j]}
& \rightarrow & \text{(}M\underset{\Bbbk }{\otimes }N\text{)[j]} & 
\rightarrow & \text{(}M\underset{R}{\otimes }N\text{)[j]} & \rightarrow 0%
\end{array}%
$
\end{proof}

\begin{lemma}
Let $B_{n}^{!}$ and $C_{n}^{!}$ be the algebras given above, for any
finitely generated graded $C_{n}^{!}$-module $M$ there is an isomorphism: $%
\Omega _{B_{n}^{!}}$($B_{n}^{!}\underset{C_{n}^{!}}{\otimes }M$)$\cong
B_{n}^{!}\underset{C_{n}^{!}}{\otimes }\Omega _{C_{n}^{!}}$($M$).

\begin{proof}
Let $0\rightarrow \Omega _{C_{n}^{!}}$($M$)$\rightarrow F\rightarrow
M\rightarrow 0$ be an exact sequence with $F$ free of rank r, the graded
projective cover of $M$. Then $\Omega _{C_{n}^{!}}$($M)\subset
J_{C_{n}^{!}}F.$

We proved $B_{n}^{!}=C_{n}^{!}\oplus $Z$C_{n}^{!}$, therefore $J_{B_{n}^{!}}$%
=$J_{C_{n}^{!}}$+Z$C_{n}^{!}.$ It follows: $B_{n}^{!}\underset{C_{n}^{!}}{%
\otimes }J_{C_{n}^{!}}$=$C_{n}^{!}\underset{C_{n}^{!}}{\otimes }%
J_{C_{n}^{!}} $+Z$C_{n}^{!}\underset{C_{n}^{!}}{\otimes }J_{C_{n}^{!}}$=$%
J_{C_{n}^{!}}$+Z$\underset{C_{n}^{!}}{\otimes }J_{C_{n}^{!}}\subset
J_{B_{n}^{!}}.$

Therefore: $B_{n}^{!}\underset{C_{n}^{!}}{\otimes }\Omega _{C_{n}^{!}}$($M$)$%
\subset B_{n}^{!}\underset{C_{n}^{!}}{\otimes \underset{r}{\oplus }}%
J_{C_{n}^{!}}\cong \underset{r}{\oplus }B_{n}^{!}\underset{C_{n}^{!}}{%
\otimes }J_{C_{n}^{!}}\subset \underset{r}{\oplus }J_{B_{n}^{!}}\cong
J_{B_{n}^{!}}$($B_{n}^{!}\underset{C_{n}^{!}}{\otimes }F$).

It follows: $0\rightarrow B_{n}^{!}\underset{C_{n}^{!}}{\otimes }\Omega
_{C_{n}^{!}}$($M$)$\rightarrow B_{n}^{!}\underset{C_{n}^{!}}{\otimes }%
F\rightarrow B_{n}^{!}\underset{C_{n}^{!}}{\otimes }M\rightarrow 0$ is exact
and $B_{n}^{!}\underset{C_{n}^{!}}{\otimes }F$ is the graded projective
cover of $B_{n}^{!}\underset{C_{n}^{!}}{\otimes }M$. Then $\Omega
_{B_{n}^{!}}$($B_{n}^{!}\underset{C_{n}^{!}}{\otimes }M$)$\cong B_{n}^{!}%
\underset{C_{n}^{!}}{\otimes }\Omega _{C_{n}^{!}}$($M$).
\end{proof}
\end{lemma}

\begin{lemma}
Let $B_{n}^{!}$ and $C_{n}^{!}$ be the algebras given above and let $M$ be a
Koszul $C_{n}^{!}$-module. Then $B_{n}^{!}\underset{C_{n}^{!}}{\otimes }M$
is Koszul and $G_{B_{n}^{!}}$($B_{n}^{!}\underset{C_{n}^{!}}{\otimes }M$)$%
\cong G_{C_{n}^{!}}(M$).
\end{lemma}

\begin{proof}
Let ...$\rightarrow F_{n}$[-n]$\rightarrow F_{n-1}$[-n+1]$\rightarrow $...$%
F_{1}$[-1]$\rightarrow F_{0}\rightarrow M\rightarrow 0$ be a graded
projective resolution of $M$,with each $F_{i}$ free of rank r$_{i}.$%
Tensoring with $B_{n}^{!}\underset{C_{n}^{!}}{\otimes }$ we obtain a graded
projective resolution of $B_{n}^{!}\underset{C_{n}^{!}}{\otimes }M$: $%
\rightarrow $($B_{n}^{!}\underset{C_{n}^{!}}{\otimes }F_{n}$)[-n]$%
\rightarrow $($B_{n}^{!}\underset{C_{n}^{!}}{\otimes }F_{n-1}$)[-n+1]$%
\rightarrow $...($B_{n}^{!}\underset{C_{n}^{!}}{\otimes }F_{1}$)[-1]$%
\rightarrow B_{n}^{!}\underset{C_{n}^{!}}{\otimes }F_{0}\rightarrow B_{n}^{!}%
\underset{C_{n}^{!}}{\otimes }M\rightarrow 0$ with each $B_{n}^{!}\underset{%
C_{n}^{!}}{\otimes }F_{i}$ free $B_{n}^{!}$-modules of rank r$_{i}$.

Moreover, Ext$_{B_{n}^{!}}^{n}$($B_{n}^{!}\underset{C_{n}^{!}}{\otimes }M$,$%
\Bbbk $)$\cong $Hom$_{B_{n}^{!}}$($\Omega ^{n}$($B_{n}^{!}\underset{C_{n}^{!}%
}{\otimes }M$),$\Bbbk $)$\cong $

Hom$_{B_{n}^{!}}$($B_{n}^{!}\underset{C_{n}^{!}}{\otimes }\Omega ^{n}M$,$%
\Bbbk $)$\cong $Hom$_{C_{n}^{!}}$($\Omega ^{n}M$,$\Bbbk $)$\cong $Ext$%
_{C_{n}^{!}}^{n}$($M,\Bbbk $).

Therefore: $G_{B_{n}^{!}}$($B_{n}^{!}\underset{C_{n}^{!}}{\otimes }M$)$\cong
G_{C_{n}^{!}}$($M$).
\end{proof}

\begin{remark}
\bigskip To be $G_{B_{n}^{!}}$($B_{n}^{!}\underset{C_{n}^{!}}{\otimes }M$) a 
$C_{n}$-module means: Z$G_{B_{n}^{!}}$($B_{n}^{!}\underset{C_{n}^{!}}{%
\otimes }M$)=0.
\end{remark}

We know $B_{n}\cong \underset{m\geq 0}{\oplus }$Ext$_{B_{n}^{!}}^{m}$($\Bbbk 
$,$\Bbbk $), $C_{n}\underset{m\geq 0}{\oplus }$Ext$_{C_{n}^{!}}^{m}$($\Bbbk $%
,$\Bbbk $), and $B_{n}$/Z$B_{n}\cong C_{n}$. Since $C_{n}^{!}$ is a sub
algebra of $B_{n}^{!}$, given an extension x$:0\rightarrow \Bbbk \rightarrow
E_{1}\rightarrow E_{2}\rightarrow ...E_{n}\rightarrow \Bbbk $ $\rightarrow 0$
of $B_{n}^{!}$, we obtain by restriction of scalars an extension \newline
resx: $0\rightarrow $res$\Bbbk \rightarrow $res$E_{1}\rightarrow $res$%
E_{2}\rightarrow $...res$E_{n}\rightarrow $res$\Bbbk $ $\rightarrow 0$ of $%
C_{n}^{!}$-modules, where res$M$ is the module $M$ with multiplication of
scalars restricted to $C_{n}^{!}.$It is clear res(xy)=res(x)res(y) and
restriction gives an homomorphism of graded $\Bbbk $-algebras: res:$\underset%
{m\geq 0}{\oplus }$Ext$_{B_{n}^{!}}^{m}$($\Bbbk $,$\Bbbk $)$\rightarrow 
\underset{m\geq 0}{\oplus }$Ext$_{C_{n}^{!}}^{m}$($\Bbbk $,$\Bbbk $).

\begin{lemma}
There is an homomorphism: $\rho $:Ext$_{B_{n}^{!}}^{1}$($\Bbbk $,$\Bbbk $)$%
\rightarrow $Ext$_{B_{n}^{!}}^{1}$((B$_{n}^{!}\underset{C_{n}^{!}}{\otimes }%
\Bbbk $,$\Bbbk $), given by the Yoneda product $\rho $(x)=x$\mu $ (pull
back) of the exact sequence $x$ with the multiplication map $\mu $:B$_{n}^{!}%
\underset{C_{n}^{!}}{\otimes }\Bbbk \rightarrow \Bbbk $, such that the
composition of the map, \linebreak $\psi _{1}$:Ext$_{B_{n}^{!}}^{1}$((B$%
_{n}^{!}\underset{C_{n}^{!}}{\otimes }\Bbbk $,$\Bbbk $)$\rightarrow $Ext$%
_{C_{n}^{!}}^{1}$($\Bbbk $,$\Bbbk $) in the previous lemma with $\rho $, is
the restriction: $\psi \rho $=res.
\end{lemma}

\begin{proof}
Let x be the extension: x: $0\rightarrow \Bbbk \rightarrow E\rightarrow
\Bbbk \rightarrow 0.$ Since $B_{n}^{!}$ is a free $C_{n}^{!}$-module, there
is a commutative exact diagram\newline
$%
\begin{array}{ccccccc}
0\rightarrow  & B_{n}^{!}\underset{C_{n}^{!}}{\otimes \Bbbk } & \rightarrow 
& B_{n}^{!}\underset{C_{n}^{!}}{\otimes }E & \rightarrow  & B_{n}^{!}%
\underset{C_{n}^{!}}{\otimes }\Bbbk  & \rightarrow 0 \\ 
& \downarrow \mu  &  & \downarrow \mu  &  & \downarrow \mu  &  \\ 
0\rightarrow  & \Bbbk  & \rightarrow  & E & \rightarrow  & \Bbbk  & 
\rightarrow 0%
\end{array}%
$

with $\mu $ multiplication.

This diagram splits in two diagrams:\newline
$%
\begin{array}{ccccccc}
0\rightarrow  & B_{n}^{!}\underset{C_{n}^{!}}{\otimes \Bbbk } & \rightarrow 
& B_{n}^{!}\underset{C_{n}^{!}}{\otimes }E & \rightarrow  & B_{n}^{!}%
\underset{C_{n}^{!}}{\otimes }\Bbbk  & \rightarrow 0 \\ 
& \downarrow \mu  &  & \downarrow  &  & \downarrow \text{1} &  \\ 
0\rightarrow  & \Bbbk  & \rightarrow  & W & \rightarrow  & B_{n}^{!}\underset%
{C_{n}^{!}}{\otimes }\Bbbk  & \rightarrow 0 \\ 
& \downarrow \text{1} &  & \downarrow  &  & \downarrow \mu  &  \\ 
0\rightarrow  & \Bbbk  & \rightarrow  & E & \rightarrow  & \Bbbk  & 
\rightarrow 0%
\end{array}%
$

Then $\rho $(x)=x$\mu $=$\mu $($B_{n}^{!}\otimes $x).

For any finitely generated $C_{n}^{!}$-module $M$ there is an isomorphism $%
\alpha $ obtained as the composition of the natural isomorphisms:\newline
Hom$_{B_{n}^{!}}$($B_{n}^{!}\underset{C_{n}^{!}}{\otimes }M$,$\Bbbk $)$\cong 
$Hom$_{C_{n}^{!}}$($M$,Hom$_{B_{n}^{!}}$($B_{n}^{!}$,$\Bbbk $))$\cong $Hom$%
_{C_{n}^{!}}$($M$,$\Bbbk $).

If j:$M\rightarrow B_{n}^{!}\underset{_{C_{n}^{!}}}{\otimes }M$ be the map
j(m)=1$\otimes $m and f:$B_{n}^{!}\underset{C_{n}^{!}}{\otimes }M\rightarrow
\Bbbk $ is any map, then $\alpha $(f)=fj.

Then $\psi \rho $(x)=$\psi $(x$\mu $) is the top sequence in the commutative
exact diagram:

$%
\begin{array}{ccccccc}
0\rightarrow & \Bbbk & \rightarrow & L & \rightarrow & \Bbbk & \rightarrow 0
\\ 
& \downarrow \text{1} &  & \downarrow &  & \downarrow \text{j} &  \\ 
0\rightarrow & \Bbbk & \rightarrow & W & \rightarrow & B_{n}^{!}\underset{%
C_{n}^{!}}{\otimes }\Bbbk & \rightarrow 0 \\ 
& \downarrow \text{1} &  & \downarrow &  & \downarrow \mu &  \\ 
0\rightarrow & \Bbbk & \rightarrow & E & \rightarrow & \Bbbk & \rightarrow 0%
\end{array}%
$

Since $\mu $j=1, gluing both diagrams we obtain $\psi \rho $(x)=resx.
\end{proof}

\begin{lemma}
Under the conditions of the previous lemma the map $\rho $ is surjective.
\end{lemma}

\begin{proof}
Since $B_{n}^{!}$=$C_{n}^{!}\oplus $Z$C_{n}^{!}$ , $B_{n}^{!}\underset{%
C_{n}^{!}}{\otimes }\Bbbk $ is a graded vector space of dimension two with
one copy of $\Bbbk $ in degree zero and one copy of $\Bbbk $ in degree one.
Hence the multiplication map $\mu $:$B_{n}^{!}\underset{C_{n}^{!}}{\otimes }%
\Bbbk \rightarrow \Bbbk $ is an epimorphsi with kernel $u$:$\Bbbk $[-1]$%
\rightarrow $ $B_{n}^{!}\underset{C_{n}^{!}}{\otimes }\Bbbk .$

Let y:0$\rightarrow \Bbbk $[-1]$\rightarrow E\rightarrow B_{n}^{!}\underset{%
C_{n}^{!}}{\otimes }\Bbbk \rightarrow 0$ be an element of Ext$%
_{B_{n}^{!}}^{1}$($\Bbbk $,B$_{n}^{!}\underset{C_{n}^{!}}{\otimes }\Bbbk $)
and take the pullback:

$%
\begin{array}{ccccccc}
0\rightarrow & \Bbbk \text{[-1]} & \rightarrow & N & \rightarrow & \Bbbk 
\text{[-1]} & \rightarrow 0 \\ 
& \downarrow \text{1} &  & \downarrow &  & \downarrow u &  \\ 
0\rightarrow & \Bbbk \text{[-1]} & \rightarrow & E & \rightarrow & B_{n}^{!}%
\underset{C_{n}^{!}}{\otimes }\Bbbk & \rightarrow 0%
\end{array}%
$

But the top exact sequence split because the ends are generated in the same
degree and the algebra is Koszul or equivalently there is a lifting $v$:$%
\Bbbk $[-1]$\rightarrow E$ of $u$ and we get a commutative exact diagram:

$%
\begin{array}{ccccccc}
& 0 & \rightarrow & \Bbbk \text{[-1]} & \rightarrow & \Bbbk \text{[-1]} & 
\rightarrow 0 \\ 
& \downarrow &  & \downarrow v &  & \downarrow u &  \\ 
0\rightarrow & \Bbbk \text{[-1]} & \rightarrow & E & \rightarrow & B_{n}^{!}%
\underset{C_{n}^{!}}{\otimes }\Bbbk & \rightarrow 0 \\ 
& \downarrow \text{1} &  & \downarrow &  & \downarrow \mu &  \\ 
0\rightarrow & \Bbbk \text{[-1]} & \rightarrow & L & \rightarrow & \Bbbk & 
\rightarrow 0%
\end{array}%
$

Proving $\rho $ is surjective.
\end{proof}

\begin{corollary}
The map res:$\underset{m\geq 0}{\oplus }$Ext$_{B_{n}^{!}}^{m}$($\Bbbk $,$%
\Bbbk $)$\rightarrow \underset{m\geq 0}{\oplus }$Ext$_{C_{n}^{!}}^{m}$($%
\Bbbk $,$\Bbbk $) is a surjective homomorphism of algebras and the kernel of
res is the ideal Z$B_{n}=B_{n}$Z.
\end{corollary}

\begin{proof}
Since both $B_{n}^{!}$ and $C_{n}^{!}$ are Koszul algebras they are graded
algebras generated in degree one, and it follows from the lemma that for any
m$>$0 the map res:Ext$_{B_{n}^{!}}^{m}$($\Bbbk $,$\Bbbk $)$\rightarrow $Ext$%
_{C_{n}^{!}}^{m}$($\Bbbk $,$\Bbbk $) is surjective.

Observe that for any homomorphism f:$B_{n}\rightarrow C_{n}$ Z is in the
kernel.

We have in $B_{n}$ the equality X$_{1}\delta _{1}$-$\delta _{1}$X$_{1}$=Z$%
^{2}$. Since $C_{n}$ is commutative, f(X$_{1}\delta _{1}$-$\delta _{1}$X$%
_{1} $)=f(X$_{1}$)f($\delta _{1}$)-f($\delta _{1}$)f(X$_{1}$)=f(Z)$^{2}$=0,

But since $C_{n}$ is an integral domain, it follows f(Z)=0.

In particular Z$B_{n}\subseteq $Ker(res) and there is a factorization:

$%
\begin{array}{ccc}
B_{n} & \rightarrow & C_{n} \\ 
\searrow &  & \nearrow \alpha \\ 
& B_{n}\text{/Z}B_{n} & 
\end{array}%
$

and since $B_{n}$/Z$B_{n}\cong C_{n}$ it follows by dimension, that $\alpha $
is an isomorphism.
\end{proof}

\begin{lemma}
With the same notation as in the previous lemma, let $M$ be a Koszul $%
C_{n}^{!}$-module and $\psi :G_{B_{n}^{!}}$($B_{n}^{!}\underset{C_{n}^{!}}{%
\otimes }M$)$\rightarrow G_{C_{n}^{!}}$($M$), the isomorphism in the
previous lemma.

Then given y$\in $Ext$_{B_{n}^{!}}^{m}$($B_{n}^{!}\underset{C_{n}^{!}}{%
\otimes }M$,$\Bbbk $) and c$\in $Ext$_{B_{n}^{!}}^{1}$($\Bbbk $,$\Bbbk $),
we have $\psi $(cy)=res(c)$\psi $(y).
\end{lemma}

\begin{proof}
The map f:$B_{n}^{!}\underset{C_{n}^{!}}{\otimes }\Omega ^{m}(M)\rightarrow
\Bbbk $ corresponding to the extension $y$ is the map in the commutative
exact diagram:\newline
$%
\begin{array}{ccccccccc}
\text{0}\rightarrow  & B_{n}^{!}\underset{C_{n}^{!}}{\otimes }\Omega ^{m}%
\text{(M)} & \rightarrow  & B_{n}^{!}\underset{C_{n}^{!}}{\otimes }\text{C}%
_{n}^{!k_{m-1}} & \text{...}\rightarrow  & B_{n}^{!}\underset{C_{n}^{!}}{%
\otimes }\text{C}_{n}^{!k_{0}} & \rightarrow  & B_{n}^{!}\underset{C_{n}^{!}}%
{\otimes }\text{M} & \rightarrow \text{0} \\ 
& \downarrow \text{f} &  & \downarrow  &  & \downarrow  &  & \downarrow 
\text{1} &  \\ 
\text{0}\rightarrow  & \Bbbk  & \rightarrow  & E_{1} & \text{...}\rightarrow 
& E_{m} & \rightarrow  & B_{n}^{!}\underset{C_{n}^{!}}{\otimes }\text{M} & 
\rightarrow \text{0}%
\end{array}%
$

where y is the bottom raw.

If j:$\Omega ^{m}$($M$)$\rightarrow B_{n}^{!}\underset{C_{n}^{!}}{\otimes }%
\Omega ^{m}$($M$) is the map j(m)=1$\otimes $m, then $\psi $(y) is the
extension corresponding to the map fj.

Consider the commutative diagram with exact rows:\newline
$%
\begin{array}{ccccccc}
0\rightarrow  & \Omega ^{m+1}\text{(}M\text{)} & \rightarrow  & 
C_{n}^{!k_{m}} & \rightarrow  & \Omega ^{m}\text{(}M\text{)} & \rightarrow 0
\\ 
& \text{j}\downarrow  &  & \text{j}\downarrow  &  & \text{j}\downarrow  & 
\\ 
0\rightarrow  & B_{n}^{!}\underset{C_{n}^{!}}{\otimes }B\Omega ^{m+1}\text{(}%
M\text{)} & \rightarrow  & B_{n}^{!}\underset{C_{n}^{!}}{\otimes }%
C_{n}^{!k_{m}} & \rightarrow  & B_{n}^{!}\underset{C_{n}^{!}}{\otimes }%
\Omega ^{m}\text{(}M\text{)} & \rightarrow 0 \\ 
& \Omega \text{f}\downarrow  &  & \downarrow  &  & \text{f}\downarrow  &  \\ 
0\rightarrow  & JB_{n}^{!} & \rightarrow  & B_{n}^{!} & \rightarrow  & \Bbbk 
& \rightarrow 0 \\ 
& \text{g}\downarrow  &  & \downarrow  &  & \text{1}\downarrow  &  \\ 
0\rightarrow  & \Bbbk  & \rightarrow  & L & \rightarrow  & \Bbbk  & 
\rightarrow 0%
\end{array}%
$

where c is the bottom sequence.

Since $B_{n}^{!}$=$C_{n}^{!}\oplus C_{n}^{!}$Z as $C_{n}^{!}$-module, the
map g restricted to $C_{n}^{!}$ represents the extension res(c).

Taking the pullback we obtain a commutative diagram with exact rows:

$%
\begin{array}{ccccccc}
0\rightarrow & \Omega ^{m+1}\text{(}M\text{)} & \rightarrow & C_{n}^{!k_{m}}
& \rightarrow & \Omega ^{m}\text{(}M\text{)} & \rightarrow 0 \\ 
& \text{g}\Omega \text{(f)j}\downarrow &  & \downarrow &  & \text{1}%
\downarrow &  \\ 
0\rightarrow & \Bbbk & \rightarrow & W & \rightarrow & \Omega ^{m}\text{(}M%
\text{)} & \rightarrow 0 \\ 
& \text{1}\downarrow &  & \downarrow &  & \text{fj}\downarrow &  \\ 
0\rightarrow & \Bbbk & \rightarrow & L & \rightarrow & \Bbbk & \rightarrow 0%
\end{array}%
$

and $\Omega $(fj)=$\Omega $(f)j.

It follows $\psi $(cy)=res(c)$\psi $(y).
\end{proof}

As a corollary we obtain the following:

\begin{proposition}
Let $M$ be a Koszul $C_{n}^{!}$-module and $G_{B_{n}^{!}}=\underset{m\geq 0}{%
\oplus }$Ext$_{C_{n}^{!}}^{m}$(-,$\Bbbk $) Koszul duality. Then Z($%
G_{B_{n}^{!}}$($B_{n}^{!}\underset{C_{n}^{!}}{\otimes }M$))=0.
\end{proposition}

\begin{proof}
Denote by z the extension corresponding to Z under the isomorphism $%
B_{n}\cong $ $\underset{m\geq 0}{\oplus }$Ext$_{B_{n}^{!}}^{m}$($\Bbbk $,$%
\Bbbk $). By the previous lemma, for any extension y$\in $Ext$%
_{B_{n}^{!}}^{m}$(B$_{n}^{!}\underset{C_{n}^{!}}{\otimes }$M),$\Bbbk $), $%
\psi $(zy)=res(z)$\psi $(y) and by lemma 10, res(z)=0. Since $\psi $ is an
isomorphism, it follows zy=0, hence Z(G$_{B_{n}^{!}}$(B$_{n}^{!}\underset{%
C_{n}^{!}}{\otimes }$M))=0.
\end{proof}

\begin{proposition}
Let $B_{n}^{!}$ and $C_{n}^{!}$ be the algebras given above. Then for any
induced module $B_{n}^{!}\underset{C_{n}^{!}}{\otimes }M$ when we apply the
duality $\overline{\phi }$ to $B_{n}^{!}\underset{C_{n}^{!}}{\otimes }M$ we
obtain an element of $\mathcal{T}$.
\end{proposition}

\begin{proof}
There exists some integer n$\geq $0 such that $\Omega ^{n}M$ and $\Omega
^{n} $($B_{n}^{!}\underset{C_{n}^{!}}{\otimes }M$) are weakly Koszul. Since $%
\overline{\phi }$($\Omega ^{n}$($B_{n}^{!}\underset{C_{n}^{!}}{\otimes }M$))$%
\cong \overline{\phi }$($B_{n}^{!}\underset{C_{n}^{!}}{\otimes }M$)[n].The
object $\overline{\phi }$($\Omega ^{n}$($B_{n}^{!}\underset{C_{n}^{!}}{%
\otimes }M$)) is in $\mathcal{T}$ if and only if $\overline{\phi }$($%
B_{n}^{!}\underset{C_{n}^{!}}{\otimes }M$) is in $\mathcal{T}$. We may
assume $M$ and $B_{n}^{!}\underset{C_{n}^{!}}{\otimes }M$ are weakly Koszul.

The module $M$ has a filtration: $M$=$U_{p}\supset U_{p-1}\supset
...U_{1}\supset U_{0}$ such that $U_{i}$/$U_{i-1}$ is Koszul, hence; $%
B_{n}^{!}\underset{C_{n}^{!}}{\otimes }M$ has a filtration: $B_{n}^{!}%
\underset{C_{n}^{!}}{\otimes }M$=$B_{n}^{!}\underset{C_{n}^{!}}{\otimes }%
U_{p}\supset B_{n}^{!}\underset{C_{n}^{!}}{\otimes }U_{p-1}\supset
...B_{n}^{!}\underset{C_{n}^{!}}{\otimes }U_{1}\supset B_{n}^{!}\underset{%
C_{n}^{!}}{\otimes }U_{0}$ such that $B_{n}^{!}\underset{C_{n}^{!}}{\otimes }%
U_{i}$/$B_{n}^{!}\underset{C_{n}^{!}}{\otimes }U_{i-1}\cong B_{n}^{!}%
\underset{C_{n}^{!}}{\otimes }$ $U_{i}$/$U_{i-1}$ is Koszul.

The exact sequence: 0$\rightarrow B_{n}^{!}\underset{C_{n}^{!}}{\otimes }%
U_{0}\rightarrow B_{n}^{!}\underset{C_{n}^{!}}{\otimes }U_{1}\rightarrow
B_{n}^{!}\underset{C_{n}^{!}}{\otimes }U_{1}$/$U_{0}\rightarrow $0 induces
an exact sequence of complexes:

0$\rightarrow \phi $($B_{n}^{!}\underset{C_{n}^{!}}{\otimes }U_{1}$/$%
U_{0})\rightarrow \phi (B_{n}^{!}\underset{C_{n}^{!}}{\otimes }U_{1}$)$%
\rightarrow \phi $($B_{n}^{!}\underset{C_{n}^{!}}{\otimes }U_{0})\rightarrow 
$0 which in turn induces a long exact sequence:\newline
...$\rightarrow $H$^{1}$($\phi $($B_{n}^{!}\underset{C_{n}^{!}}{\otimes }%
U_{1}$/$U_{0}$))$\rightarrow $H$^{1}$($\phi (B_{n}^{!}\underset{C_{n}^{!}}{%
\otimes }U_{1}$))$\rightarrow $H$^{1}$($\phi (B_{n}^{!}\underset{C_{n}^{!}}{%
\otimes }U_{0}$))$\rightarrow $H$^{0}$($\phi $($B_{n}^{!}\underset{C_{n}^{!}}%
{\otimes }U_{1}$/$U_{0}$))\linebreak $\rightarrow $H$^{0}$($\phi $($B_{n}^{!}%
\underset{C_{n}^{!}}{\otimes }U_{1}$))$\rightarrow $H$^{0}$($\phi $($%
B_{n}^{!}\underset{C_{n}^{!}}{\otimes }U_{0}$))$\rightarrow $0

where $H^{i}$($\phi $($B_{n}^{!}\underset{C_{n}^{!}}{\otimes }U_{0}$))=0 for
i$\neq $0 and H$^{0}$($\phi $($B_{n}^{!}\underset{C_{n}^{!}}{\otimes }U_{0}$%
))=$G_{B_{n}^{!}}$($B_{n}^{!}\underset{C_{n}^{!}}{\otimes }U_{0}$)$\cong
G_{C_{n}^{!}}$($U_{0}$) of Z-torsion, H$^{0}$($\phi $($B_{n}^{!}\underset{%
C_{n}^{!}}{\otimes }U_{1}$))$\cong $H$^{0}$($\phi $($B_{n}^{!}\underset{%
C_{n}^{!}}{\otimes }U_{0}$)) and H$^{i}$($\phi $( $B_{n}^{!}\underset{%
C_{n}^{!}}{\otimes }U_{1}$/$U_{0}))\cong $\linebreak H$^{i}$($\phi $($%
B_{n}^{!}\underset{C_{n}^{!}}{\otimes }U_{1}$)) for i$\neq $0. It follows H$%
^{i}$($\phi $($B_{n}^{!}\underset{C_{n}^{!}}{\otimes }U_{1}$)) is of
Z-torsion for all $i$. By induction H$^{i}$($\phi $($B_{n}^{!}\underset{%
C_{n}^{!}}{\otimes }M$)) is of Z-torsion for all i.

We have proved $\phi $($B_{n}^{!}\underset{C_{n}^{!}}{\otimes }M$)$\in 
\mathcal{T}$.
\end{proof}

\begin{lemma}
Let $M$ be a $B_{n}^{!}$-module and assume there is an integer n$\geq $0
such that $\Omega ^{n}M$ =$N$ has the following properties:

The module $N$ is weakly Koszul, it has a filtration: $N$=$U_{p}\supset
U_{p-1}\supset ...U_{1}\supset U_{0}$ such that $U_{i}$/$U_{i-1}$ is Koszul,
and for all k$\geq $0, $J^{k}U_{i}\cap U_{i-1}$=$J^{k}U_{i-1}.$

The Koszul modules $G_{B_{n}^{!}}$($U_{i}$/$U_{i-1}$)=$V_{i}$ are of
Z-torsion.

Then $\phi $($M$) is in $\mathcal{T}$.
\end{lemma}

\begin{proof}
As above, $\phi $($M$) is in $\mathcal{T}$ if and only if $\phi $($N$) is in 
$\mathcal{T}$.

The exact sequence: $0\rightarrow U_{0}$ $\rightarrow U_{1}\rightarrow $ $%
U_{1}$/ $U_{0}\rightarrow 0$ induces an exact sequence: $0\rightarrow \phi $(%
$U_{1}$/ $U_{0}$)$\rightarrow \phi $($U_{1}$)$\rightarrow \phi $($U_{0}$)$%
\rightarrow 0$ such that H$^{0}$($\phi $($U_{1}$))$\cong $H$^{0}$($\phi $($%
U_{0}$))$\cong G_{B_{n}^{!}}$($U_{0}$) is of Z-torsion and H$^{i}$($\phi $($%
U_{1}$/$U_{0}$))$\cong $H$^{i}$($\phi $($U_{1}$)) is of Z-torsion for all i$%
\neq $0. By induction, H$^{i}$($\phi $(N)) is of Z-torsion for all i, hence $%
\phi $($N$) is in $\mathcal{T}$.
\end{proof}

\begin{theorem}
Let $\mathcal{T}$ $^{\prime }$ be the subcategory of \underline{gr}$%
_{B_{n}^{!}}$ corresponding to \emph{T }under the duality: $\overline{\phi }$%
:\underline{gr}$_{_{B_{n}^{!}}}\rightarrow $D$^{b}$(Qgr$_{B_{n}}$). This is: 
$\overline{\phi }$($\mathcal{T}^{\prime }$)=\emph{T}.\emph{\ }Then $\mathcal{%
T}$ $^{\prime }$ is the smallest triangulated subcategory of \underline{gr}$%
_{B_{n}^{!}}$containing the induced modules and closed under the Nakayama
automorphism.
\end{theorem}

\begin{proof}
Let $\mathcal{B}$ be a triangulated subcategory of \underline{gr}$%
_{B_{n}^{!}}$containing the induced modules and closed under the Nakayama
automorphism. Let $M\in $ $\mathcal{T}$ $^{\prime }$ and $\Omega ^{n}M$=$N$
weakly Koszul with a filtration $N$=$U_{p}\supset U_{p-1}\supset
...U_{1}\supset U_{0}$ such that $U_{i}/U_{i-1}$ is Koszul, and for all k$%
\geq $0, $J^{k}U_{i}\cap U_{i-1}$=$J^{k}U_{i-1}$.

Since $\mathcal{T}$ $^{\prime }$ is closed under the shift the module $N$ is
also in $\mathcal{T}$ $^{\prime }$. We prove by induction on the length of
the filtration that for each $i$ the modules $U_{i}$, $U_{i}/U_{i-1}$are in $%
\mathcal{T}$ $^{\prime }$.

We have an exact sequence of complexes: $0\rightarrow \phi ($ $N/$ $%
U_{0})\rightarrow \phi (N)\rightarrow \phi (U_{0})\rightarrow 0$

By the long homology sequence there is an exact sequence:

...H$_{i+1}$($\phi $($U_{0}$)$\rightarrow $H$_{i}$($\phi $( $N$/ $U_{0}$))$%
\rightarrow $H$_{i}$($\phi $($N$))$\rightarrow $H$_{i}$($\phi $($U_{0}$))$%
\rightarrow $H$_{i-1}$($\phi $($N$/ $U_{0}$))...\linebreak $\rightarrow $H$%
_{0}$($\phi $( $N$/ $U_{0}$))$\rightarrow $H$_{0}$($\phi $($N$))$\rightarrow 
$H$_{0}$($\phi $($U_{0}$))$\rightarrow 0$

By [MZ], H$_{0}$($\phi $($N$))=H$_{0}$($\phi $($U_{0}$)), H$_{0}$($\phi $( $%
N $/$U_{0}$))=$0$ and H$_{i}$($\phi $($U_{0})$)=0 for all i$\neq $0. Then H$%
_{i}$($\phi $($N/$ $U_{0}$))=H$_{i}$($\phi $($N$)) for all i$\neq $0 and H$%
_{0}$($\phi $($U_{0}$)) is of Z-torsion and \linebreak H$_{i}$($\phi $($N$/$%
U_{0}$)) is of Z -torsion for all i.

It follows by induction, $U_{i}$, $U_{i}$/$U_{i-1}$are in $\mathcal{T}$ $%
^{\prime }$ for all $i$.

The Koszul modules $G_{B_{n}^{!}}$($U_{i}$/$U_{i-1}$)=V$_{i}$ are of
Z-torsion and each Z$^{j}V_{i}$ is Koszul.

There exists an exact sequence:

$0\rightarrow F_{B_{n}}$(Z$^{k_{i}}V_{i}$)[-k$_{i}$]$\rightarrow F_{B_{n}}$(Z%
$^{k_{i}-1}V_{i}$/Z$^{k_{i}}V_{i}$)[-k$_{i}$+1]...

$\rightarrow F_{B_{n}}$($V_{i}$/Z$V_{i}$)$\rightarrow U_{i}$/$%
U_{i-1}\rightarrow $0 where each $F_{B_{n}}$(Z$^{j}V_{i}$/Z$^{j+1}V_{i}$)$%
\cong B_{n}^{!}\underset{C_{n}^{!}}{\otimes }X_{ij}$ is an induced module of
a Koszul $C_{n}^{!}$-module $X_{ij}$.

Then each $F_{B_{n}}$(Z$^{j}V_{i}$/Z$^{j+1}V_{i}$) $\cong B_{n}^{!}\underset{%
C_{n}^{!}}{\otimes }X_{ij}$ is in $\mathcal{B}$ .

Moreover, the exact sequences: 0$\rightarrow $B$_{n}^{!}\underset{C_{n}^{!}}{%
\otimes }$X$_{ik_{i}}\rightarrow $B$_{n}^{!}\underset{C_{n}^{!}}{\otimes }$X$%
_{ik_{i}-1}\rightarrow $K$_{k_{i}-2}\rightarrow $0 gives rise to triangles: $%
B_{n}^{!}\underset{C_{n}^{!}}{\otimes }X_{ik_{i}}\rightarrow B_{n}^{!}%
\underset{C_{n}^{!}}{\otimes }X_{ik_{i}-1}\rightarrow K_{k_{i}-2}\rightarrow
\Omega ^{-1}$($B_{n}^{!}\underset{C_{n}^{!}}{\otimes }X_{ik_{i}}$).
Therefore $K_{k_{i}-2}\in \mathcal{B}$. It follows by induction, $U_{i}$/$%
U_{i-1}\in \mathcal{B}$.

The filtration $N$=$U_{p}\supset U_{p-1}\supset ...U_{1}\supset U_{0}$
induces triangles:

$U_{0}\rightarrow U_{1}\rightarrow U_{1}$/$U_{0}\rightarrow \Omega ^{-1}$($%
U_{0}$) with $U_{0}$,$U_{1}$/$U_{0}\in \mathcal{B}$. It follows $U_{1}\in 
\mathcal{B}$.

By induction, $N\in \mathcal{B}$.

We have proved $\mathcal{T}^{\prime }\subset \mathcal{B}$.
\end{proof}

\begin{theorem}
\bigskip Let $\mathcal{T}$ $^{\prime }$ be the subcategory of \underline{gr}$%
_{B_{n}^{!}}$ corresponding to \emph{T }under the duality: $\overline{\phi }$%
:\underline{gr}$_{_{B_{n}^{!}}}\rightarrow $D$^{b}$(Qgr$_{B_{n}}$). This is: 
$\overline{\phi }$($\mathcal{T}^{\prime }$)=\emph{T }. Then $\mathcal{T}$ $%
^{\prime }$has Auslander Reiten triangles and they are of type $%
\mathbb{Z}
A_{\infty }$.
\end{theorem}

\begin{proof}
Let $M$ be an indecomposable non projective module in $\mathcal{T}$ $%
^{\prime }$. Then we have almost split sequences: $0\rightarrow \sigma
\Omega ^{2}M\rightarrow E\rightarrow M\rightarrow 0$ and $0\rightarrow
M\rightarrow F\rightarrow \sigma ^{-1}\Omega ^{-2}M\rightarrow 0$, since the
category \emph{T }is closed under the Nakayama automorphism, $\mathcal{T}$ $%
^{\prime }$ is also closed under the Nakayama automorphism and $\sigma
\Omega ^{2}M$, $\sigma ^{-1}\Omega ^{-2}M$ are objects in $\mathcal{T}$ $%
^{\prime }$. From the exact sequences of complexes: $0\rightarrow \phi $($M$)%
$\rightarrow \phi $($E$)$\rightarrow \phi $($\sigma \Omega ^{2}M$)$%
\rightarrow 0$ and $0\rightarrow \phi $($\sigma ^{-1}\Omega
^{-2}M)\rightarrow \phi $($F$)$\rightarrow \phi $($M$)$\rightarrow 0$ and
the long homology sequence we get that both $\phi $( $E$) and $\phi $( $F$)
are in \emph{T}. Therefore: $E$ and $F$ are in $\mathcal{T}^{\prime }$. We
have proved $\mathcal{T}$ $^{\prime }$has almost split sequences and they
are almost split sequences in $gr_{_{B_{n}^{!}}}$. We proved in [MZ] that
the Auslander Reiten components of \underline{gr}$_{B_{n}^{!}}$ are of type $%
\mathbb{Z}
A_{\infty }$. It follows $\sigma \Omega ^{2}M\rightarrow E\rightarrow
M\rightarrow \sigma \Omega ^{2}M$[-1], $M\rightarrow F\rightarrow \sigma
^{-1}\Omega ^{-2}M\rightarrow M$[-1] are Auslander Reiten triangles and the
Auslander Reiten components are of type $%
\mathbb{Z}
A_{\infty }$.
\end{proof}

We will characterize now the full subcategory $\mathcal{F}$ $^{\prime }$ of 
\underline{\text{gr}}$_{B_{n}^{!}}$such that $\overline{\phi }$($\mathcal{F}$
$^{\prime }$)=$\mathcal{F}$\emph{\ .}

\begin{theorem}
The subcategory $\mathcal{F}$ $^{\prime }$ of \underline{gr}$_{B_{n}^{!}}$%
such that $\overline{\phi }$($\mathcal{F}$ $^{\prime }$)=$\mathcal{F}$
consists of the graded $B_{n}^{!}$-modules $M$ such that the restriction of $%
M$ to $C_{n}^{!}$ is injective.
\end{theorem}

\begin{proof}
Let $M\in \mathcal{F}$ $^{\prime }.$ There is an isomorphism:

Hom$_{D^{b}(Qgr_{B_{n}^{op}})}$($\mathcal{T}$,$\pi \phi $($M$))$\cong $%
\underline{Hom}$_{grB_{n}^{!}}$($M$,$\mathcal{T}$ $^{\prime }$)=0, which
implies \underline{Hom}$_{B_{n}^{!}}$(M,$\mathcal{T}^{\prime }$)=0

In particular for any induced module $B_{n}^{!}\underset{C_{n}^{!}}{\otimes }%
\Omega ^{2}L$ we have:

\underline{Hom}$_{B_{n}^{!}}$($M$,$B_{n}^{!}\underset{C_{n}^{!}}{\otimes }%
\Omega ^{2}L$)=0=\underline{Hom}$_{B_{n}^{!}}$($\Omega ^{-2}M$,$B_{n}^{!}%
\underset{C_{n}^{!}}{\otimes }L).$

By Auslander-Reiten formula:

D(\underline{Hom}$_{B_{n}^{!}}$($\Omega ^{-2}M$,$B_{n}^{!}\underset{C_{n}^{!}%
}{\otimes }L$))=Ext$_{B_{n}^{!}}^{1}$($B_{n}^{!}\underset{C_{n}^{!}}{\otimes 
}L$,$M)$=0 for all $L\in $gr$_{C_{n}^{!}}.$

Consider the exact sequences: $0\rightarrow \Omega _{C_{n}^{!}}$($L$)$%
\rightarrow F\rightarrow L\rightarrow 0$, with $F$ the projective cover of $%
L $. It induces an exact sequence:\newline
$0\rightarrow B_{n}^{!}\underset{C_{n}^{!}}{\otimes }\Omega _{C_{n}^{!}}$($L$%
)$\rightarrow B_{n}^{!}\underset{C_{n}^{!}}{\otimes }F\rightarrow B_{n}^{!}%
\underset{C_{n}^{!}}{\otimes }L\rightarrow 0$

By the long homology sequence, there is an exact sequence:

$0\rightarrow $Hom$_{B_{n}^{!}}$($B_{n}^{!}\underset{C_{n}^{!}}{\otimes }L$,$%
M$)$\rightarrow $Hom$_{B_{n}^{!}}$($B_{n}^{!}\underset{C_{n}^{!}}{\otimes }F$%
,$M$)$\rightarrow $Hom$_{B_{n}^{!}}$($B_{n}^{!}\underset{C_{n}^{!}}{\otimes }%
\Omega _{C_{n}^{!}}$($L$),$M$)$\rightarrow $

Ext$_{B_{n}^{!}}^{1}$($B_{n}^{!}\underset{C_{n}^{!}}{\otimes }L$,$M$)$%
\rightarrow 0$ which by the adjunction isomorphism are isomorphic to the
exact sequences:

$0\rightarrow $Hom$_{C_{n}^{!}}$($L$, $M$)$\rightarrow $Hom$_{C_{n}^{!}}$($F$%
, $M)$ $\rightarrow $Hom$_{C_{n}^{!}}$($\Omega _{C_{n}^{!}}$($L$), $M$)$%
\rightarrow $Ext$_{C_{n}^{!}}^{1}$($L$, $M$)$\rightarrow 0.$

It follows, Ext$_{C_{n}^{!}}^{1}$($L$,$M$)$\cong $Ext$_{B_{n}^{!}}^{1}$($%
B_{n}^{!}\underset{C_{n}^{!}}{\otimes }L$, $M$) and by dimension shift

Ext$_{C_{n}^{!}}^{k}$($L$, $M$)$\cong $Ext$_{B_{n}^{!}}^{k}$($B_{n}^{!}%
\underset{C_{n}^{!}}{\otimes }L$, $M$)=0 for all k$\geq $1.

We have proved the restriction of $M$ to $C_{n}^{!}$ is injective.

Let's assume now the restriction of $M$ to $C_{n}^{!}$ is injective:

Then for any integer $n$ the restriction of $\Omega ^{n}M$ to $C_{n}^{!}$ is
injective.

Let $X\in \mathcal{T}^{\prime }$. There exists an integer n$\geq $0 such
that $\Omega ^{n}X$ =$Y$, is weakly Koszul and it has a filtration: $Y$=$%
U_{p}\supset U_{p-1}\supset ...U_{1}\supset U_{0}$ such that $U_{i}$/$%
U_{i-1} $ is Koszul, and for all k$\geq $0, $J^{k}U_{i}\cap U_{i-1}$=$%
J^{k}U_{i-1}.$The Koszul modules $G_{B_{n}^{!}}$($U_{i}$/$U_{i-1}$)=$V_{i}$
are of $Z$-.torsion and each Z$^{j}V_{i}$ is Koszul.

Set $N$=$\Omega ^{n}M$, the restriction of $N$ to $C_{n}^{!}$ is injective.

There exist exact sequences:

$0\rightarrow F_{B_{n}}$(Z$^{k_{i}}V_{i}$)[-k$_{i}$]$\rightarrow F_{B_{n}}$(Z%
$^{k_{i}-1}V_{i}$/Z$^{k_{i}}V_{i}$)[-k$_{i}$+1]...

$\rightarrow F_{B_{n}}$($V_{i}$/Z$V_{i}$)$\rightarrow U_{i}$/$%
U_{i-1}\rightarrow 0$ where each $F_{B_{n}}$(Z$^{j}V_{i}$/Z$^{j+1}V_{i}$) $%
\cong B_{n}^{!}\underset{C_{n}^{!}}{\otimes }X_{ij}$ is an induced module of
a Koszul $C_{n}^{!}$-module $X_{ij}$.

The exact sequences: $0\rightarrow B_{n}^{!}\underset{C_{n}^{!}}{\otimes }%
X_{ik_{i}}\rightarrow B_{n}^{!}\underset{C_{n}^{!}}{\otimes }%
X_{ik_{i}-1}\rightarrow K_{k_{i}-2}\rightarrow 0$ induce exact sequences:

$0\rightarrow $Hom$_{B_{n}^{!}}$($K_{k_{i}-2}$,$N$)$\rightarrow $Hom$%
_{B_{n}^{!}}$($B_{n}^{!}\underset{C_{n}^{!}}{\otimes }X_{ik_{i}-1}$ ,$N$)$%
\rightarrow $Hom$_{B_{n}^{!}}$($B_{n}^{!}\underset{C_{n}^{!}}{\otimes }%
X_{ik_{i}}$,$N$)

$\rightarrow $Ext$_{B_{n}^{!}}^{1}$($K_{k_{i}-2}$,$N$)$\rightarrow $Ext$%
_{B_{n}^{!}}^{1}$($B_{n}^{!}\underset{C_{n}^{!}}{\otimes }X_{ik_{i}-1}$,$N$)$%
\rightarrow $Ext$_{B_{n}^{!}}^{1}$($B_{n}^{!}\underset{C_{n}^{!}}{\otimes }%
X_{ik_{i}}$,$N$)

$\rightarrow $Ext$_{B_{n}^{!}}^{2}$($K_{k_{i}-2}$,$N$)$\rightarrow $Ext$%
_{B_{n}^{!}}^{2}$($B_{n}^{!}\underset{C_{n}^{!}}{\otimes }X_{ik_{i}-1}$,$N$)$%
\rightarrow $..

where Ext$_{B_{n}^{!}}^{j}$($B_{n}^{!}\underset{C_{n}^{!}}{\otimes }%
X_{ik_{i}-l}$,$N$)$\cong $Ext$_{C_{n}^{!}}^{j}$($X_{ik_{i}-1}$,$N$)=0 for
all j$\geq $1.

It follows: Ext$_{B_{n}^{!}}^{j}$($K_{k_{i}-2}$,$N$)=0 for all j$\geq $2.But
Ext$_{B_{n}^{!}}^{j}$($K_{k_{i}-2},N$)

$\cong $Ext$_{B_{n}^{!}}^{j-1}$($K_{k_{i}-2}$,$\Omega ^{-1}N$) and Ext$%
_{B_{n}^{!}}^{j}$($K_{k_{i}-2}$,$\Omega ^{-1}N$)=0 for all j$\geq $1.

The sequences: $0\rightarrow K_{k_{i}-2}\rightarrow B_{n}^{!}\underset{%
C_{n}^{!}}{\otimes }X_{ik_{i}-2}\rightarrow K_{k_{i}-3}\rightarrow 0$ induce
exact \linebreak sequences:

Ext$_{B_{n}^{!}}^{j}$($K_{k_{i}-3}$,$\Omega ^{-1}N$)$\rightarrow $Ext$%
_{B_{n}^{!}}^{j}$($B_{n}^{!}\underset{C_{n}^{!}}{\otimes }X_{ik_{i}-2}$,$%
\Omega ^{-1}N$)$\rightarrow $Ext$_{B_{n}^{!}}^{j}$($K_{k_{i}-2}$,$\Omega
^{-1}N$)

$\rightarrow $Ext$_{B_{n}^{!}}^{j+1}$($K_{k_{i}-3}$,$\Omega ^{-1}N$)$%
\rightarrow $Ext$_{B_{n}^{!}}^{j+1}$($B_{n}^{!}\underset{C_{n}^{!}}{\otimes }%
X_{ik_{i}-2}$,$\Omega ^{-1}$N)$\rightarrow $..

Therefore Ext$_{B_{n}^{!}}^{j+1}$($K_{k_{i}-3}$,$\Omega ^{-1}N$)=0 for j$%
\geq $1which implies

Ext$_{B_{n}^{!}}^{j}$($K_{k_{i}-3}$,$\Omega ^{-2}N$)=0 for j$\geq $1.

Continuing by induction there exist some m$\geq $0 such that

Ext$_{B_{n}^{!}}^{j}$($U_{i}$/$U_{i-1}$,$\Omega ^{-m}N$)=0 for j$\geq $1$.$

By induction on $p$ we obtain Ext$_{B_{n}^{!}}^{j}$($Y$,$\Omega ^{-m}N$)=0
for j$\geq $1, in particular \linebreak Ext$_{B_{n}^{!}}^{1}$(Y,$\Omega
^{-m} $N)=0.

By Auslander-Reiten formula, Ext$_{B_{n}^{!}}^{1}$($Y$,$\Omega ^{-m}N$)$%
\cong $D(\underline{Hom}$_{B_{n}^{!}}$($\Omega ^{-m}N$,$\Omega ^{2}Y$))$%
\cong $\linebreak D(\underline{Hom}$_{B_{n}^{!}}$($N$,$\Omega ^{2+m}Y$))$%
\cong $D(\underline{Hom}$_{B_{n}^{!}}$($\Omega ^{n}M$,$\Omega ^{2+m}Y$)).

It follows\underline{ Hom}$_{B_{n}^{!}}$($\Omega ^{n}M$,$\Omega ^{2+m+n}X$%
)=0 which implies \underline{Hom}$_{B_{n}^{!}}$($M$,$\Omega ^{2+m}X$)=0.

Observe m depends only on $X$. Taking $\Omega ^{2+m}M$ instead of $M$ we
obtain\linebreak\ \underline{Hom}$_{B_{n}^{!}}$($\Omega ^{2+m}M$,$\Omega
^{2+m}X$)=\underline{Hom}$_{B_{n}^{!}}$($M$,$X$)=0.

Therefore \underline{Hom}$_{grB_{n}^{!}}$($M$,$X$)=0. It follows $M\in 
\mathcal{F}^{\prime }$.
\end{proof}

\begin{theorem}
\bigskip The category $\mathcal{F}^{\prime }$ is closed under the Nakayama
automorphism, $\mathcal{F}^{\prime }$ has Auslander Reiten sequences and
they are of the form $%
\mathbb{Z}
A_{\infty }$. Moreover, $\mathcal{F}^{\prime }$ is a triangulated category
with Auslander-Reiten triangles and they are of type $%
\mathbb{Z}
A_{\infty }$.
\end{theorem}

\begin{proof}
Let $M$ be an indecomposable non projective object in $\mathcal{F}^{\prime }$
and $0\rightarrow \Omega M\rightarrow P\rightarrow M\rightarrow 0$ exact
with $P$ the projective cover of $M$. Since the restriction of $P$ to $%
C_{n}^{!}$ is projective and restriction is an exact functor, it follows $%
\Omega M$ is in $\mathcal{F}^{\prime }$. Similarly, $\Omega ^{-1}M$ is in $%
\mathcal{F}^{\prime }$.

If $P$ is a projective $B_{n}^{!}$ -module and $\sigma $ the Nakayama
automorphism, then $\sigma P$ is also projective. Therefore: $\mathcal{F}%
^{\prime }$ is closed under the Nakayama automorphism.

It is clear now that $\mathcal{F}^{\prime }$ has Auslander Reiten sequences
and they are of the form $%
\mathbb{Z}
A_{\infty }$, by [MZ].

Let f:$M\rightarrow N$ be a homomorphism with $M$, $N$ in $\mathcal{F}%
^{\prime }$ and let j:$M\rightarrow P$ be the injective envelope of $M$.
There is an exact sequence: $0\rightarrow M\rightarrow P\oplus N\rightarrow
L\rightarrow 0$ with $M$ and $P\oplus N$ in $\mathcal{F}^{\prime }$. Then $L$
is also in $\mathcal{F}^{\prime }$ and the triangle $M\rightarrow
N\rightarrow L\rightarrow \Omega ^{-1}M$ is a triangle in $\mathcal{F}%
^{\prime }$.
\end{proof}

We have characterized the pair ($\mathcal{F}^{\prime }$,$\mathcal{T}^{\prime
}$) corresponding to ($\mathcal{T}$,$\mathcal{F})$ under the duality $%
\overline{\phi }$:\underline{gr}$_{_{B_{n}^{!}}}\rightarrow $D$^{b}$(Qgr$%
_{B_{n}}$)$.$ Applying the usual duality D:\underline{gr}$%
_{_{B_{n}^{!}}}\rightarrow $\underline{gr}$_{_{B_{n}^{!op}}}$we obtain a
pair (D($\mathcal{T}^{\prime }$),D($\mathcal{F}^{\prime }$)) which
corresponds to ($\mathcal{T}$,$\mathcal{F}$) under the equivalence: 

$\overline{\phi }$D:\underline{gr}$_{_{B_{n}^{!op}}}\rightarrow $D$^{b}$(Qgr$%
_{B_{n}}$).

Observe the following:

From the bimodule isomorphism $B_{n}^{!}\sigma ^{-1}\cong $D($B_{n}^{!}$),
for any induced $B_{n}^{!}$-module $B_{n}^{!}\underset{C_{n}^{!}}{\otimes }X$%
, there are natural isomorphisms:

Hom$_{B_{n}^{!}}$($B_{n}^{!}\underset{C_{n}^{!}}{\otimes }X$,D($B_{n}^{!}$))$%
\cong $D($B_{n}^{!}\underset{C_{n}^{!}}{\otimes }X$)$\cong $Hom$_{B_{n}^{!}}$%
($B_{n}^{!}\underset{C_{n}^{!}}{\otimes }X$, $B_{n}^{!}\sigma ^{-1}$)$\cong $

Hom$_{C_{n}^{!}}$($X$, $B_{n}^{!}\sigma ^{-1}$)$\cong $Hom$_{C_{n}^{!}}$($X$,%
$C_{n}^{!}$)$\underset{C_{n}^{!}}{\otimes }B_{n}^{!}\sigma ^{-1}$.

For any finitely generated right $C_{n}^{!}$-module $Y$ there exists a left $%
C_{n}^{!}$-module $X$ such that Hom$_{C_{n}^{!}}$($X$,$C_{n}^{!}$)$\cong Y$,
hence D($B_{n}^{!}\underset{C_{n}^{!}}{\otimes }X$)$\sigma \cong Y\underset{%
C_{n}^{!}}{\otimes }B_{n}^{!}$. Since $\mathcal{T}^{\prime }$ is invariant
under $\sigma $, D($\mathcal{T}^{\prime }$) is also invariant under $\sigma $
and D($\mathcal{T}^{\prime }$) contains the induced modules.

Let $\emph{B}$ be a triangulated subcategory of \underline{gr}$%
_{B_{n}^{!op}} $ containing the induced modules. A triangle $A$ $\overset{f}{%
\rightarrow }B\overset{g}{\rightarrow }C\overset{h}{\rightarrow }A$[1] in 
\underline{gr}$_{B_{n}^{!}}$ comes from an exact sequence\linebreak\ $%
0\rightarrow A$ $\overset{\left( 
\begin{array}{c}
\text{f} \\ 
\text{u}%
\end{array}%
\right) }{\rightarrow }B\oplus P\overset{(g,v)}{\rightarrow }C\rightarrow 0$
with \ $P$ a projective module, hence \linebreak D($C$)$\overset{D(g)}{%
\rightarrow }$D($B$)$\overset{D(f)}{\rightarrow }$D($A$)$\rightarrow $D($C$%
)[1] is a triangle in \underline{gr}$_{_{B_{n}^{!op}}}$. Therefore: D(\emph{B%
}) is a triangulated category containing the duals of the induced modules D($%
Y\underset{C_{n}^{!}}{\otimes }B_{n}^{!}$)$\cong $\linebreak D(D($B_{n}^{!}%
\underset{C_{n}^{!}}{\otimes }X$)$\sigma $)$\cong $D(Hom$_{C_{n}^{!}}$($X,$ $%
C_{n}^{!}$)$\underset{C_{n}^{!}}{\otimes }B_{n}^{!}$)$\cong $Hom$%
_{B_{n}^{!}} $(Hom$_{C_{n}^{!}}$($X$, $C_{n}^{!}$)$\underset{C_{n}^{!}}{%
\otimes }B_{n}^{!} $,$\sigma B_{n}^{!}$)$\cong \sigma B_{n}^{!}\underset{%
C_{n}^{!}}{\otimes }X$. Clearly $\sigma $D(\emph{B}) is a triangulated
category containing the induced modules. Therefore: $\mathcal{T}^{\prime
}\subset $ $\sigma $D(\emph{B}). Since $\mathcal{T}^{\prime }$ is closed
under Nakayama%
\'{}%
s automorphism $\sigma $, $\mathcal{T}^{\prime }\subset $ D(\emph{B}). It
follows D($\mathcal{T}^{\prime }$)$\subset $ $\emph{B}$ and D($\mathcal{T}%
^{\prime }$) can be described as the smallest triangulated subcategory of 
\underline{gr}$_{_{B_{n}^{!op}}}$ that contains the induced modules.

The usual duality D induces an isomorphism:

D:Ext$_{C_{n}^{!}}^{i}$($M$,$N$)$\rightarrow $Ext$_{C_{n}^{op!}}^{i}$(D($N$%
),D($M$)).

It follows that the restriction of $M$ to $C_{n}^{!}$ is injective if and
only if the restriction of D($M$) to $C_{n}^{op!}$ is projective
(injective). It follows D($\mathcal{F}^{\prime }$) is the category of $%
B_{n}^{op!}$-modules whose restriction to $C_{n}^{op!}$ is injective.

\emph{T=}D($\mathcal{T}^{\prime }$) is a "epasse" subcategory of \underline{%
gr}$_{_{B_{n}^{!op}}}.$The functor $\overline{\phi }$D induces an
equivalence of categories: \underline{gr}$_{_{B_{n}^{!op}}}$/\emph{T}$\cong $%
D$^{b}$(Qgr$_{B_{n}^{op}}$)/$\mathcal{T}$ and we proved D$^{b}$(Qgr$%
_{B_{n}^{op}}$)/$\mathcal{T}\cong $\linebreak D$^{b}$(gr$_{(B_{n})_{Z}}$).

The equivalence gr$_{(B_{n})_{Z}}\cong $mod$_{A_{n}}$ induces an
equivalence: D$^{b}$(gr$_{(B_{n})_{Z}}$)$\cong $\linebreak D$^{b}$(mod$%
_{A_{n}}$).

We have proved:

\begin{theorem}
There is an equivalence of triangulated categories:

\underline{gr}$_{_{B_{n}^{!}}}$/\emph{T}$\cong $D$^{b}$(mod$_{A_{n}}$).
\end{theorem}

The category \emph{F}=D($\mathcal{F}^{\prime }$) is the category of all 
\emph{T} -local objects it is triangulated. By [Mi], there is a full
embedding: $\emph{F}\rightarrow $\underline{gr}$_{_{B_{n}^{!op}}}$/\emph{T}$%
\cong $D$^{b}$(mod$_{A_{n}}$).

\begin{proposition}
The category ind$_{C_{n}^{!}}$ of all induced $B_{n}^{!}$-modules is
contravariantly finite in gr$_{B_{n}^{!}}$.
\end{proposition}

\begin{proof}
Let $M$ be a $B_{n}^{!}$-module and $\mu $: $B_{n}^{!}\underset{C_{n}^{!}}{%
\otimes }M\rightarrow M$ the map given by multiplication. Let $\alpha $:Hom$%
_{B_{n}^{!}}$($B_{n}^{!}\underset{C_{n}^{!}}{\otimes }M$,$M$)$\rightarrow $%
Hom$_{C_{n}^{!}}$($M$,$M$) the morphism giving the adjunction. It is easy to
see that $\alpha $($\mu $)=1$_{M}$.

Let $\varphi $:$B_{n}^{!}\underset{C_{n}^{!}}{\otimes }N\rightarrow M$ be
any map and $\alpha $($\varphi $)=f: $N\rightarrow M$ the map given by
adjunction. There is a commutative square

$%
\begin{array}{ccc}
\text{Hom}_{B_{n}^{!}}\text{(}B_{n}^{!}\underset{C_{n}^{!}}{\otimes }M\text{,%
}M\text{)} & \overset{\alpha _{M}}{\rightarrow } & \text{Hom}_{C_{n}^{!}}%
\text{(}M\text{,}M\text{)} \\ 
\downarrow \text{(1}\otimes \text{f,}M\text{)} &  & \downarrow \text{(f,}M%
\text{)} \\ 
\text{Hom}_{B_{n}^{!}}\text{(}B_{n}^{!}\underset{C_{n}^{!}}{\otimes }N\text{,%
}M\text{)} & \overset{\alpha _{N}}{\rightarrow } & \text{Hom}_{C_{n}^{!}}%
\text{(}N\text{,}M\text{)}%
\end{array}%
$

from the commutativity of the diagram f=$\alpha _{N}$($\varphi $)=$\alpha
_{N}$($\mu $1$\otimes $f) implies $\varphi $=$\mu $1$\otimes $f.

We have proved the triangle:

$%
\begin{array}{ccc}
&  & B_{n}^{!}\underset{C_{n}^{!}}{\otimes }N \\ 
& \text{1}\otimes \text{f}\swarrow & \downarrow \phi \\ 
B_{n}^{!}\underset{C_{n}^{!}}{\otimes }M & \overset{\mu }{\rightarrow } & M%
\end{array}%
$

commutes.
\end{proof}

\begin{corollary}
add(ind$_{C_{n}^{!}}$) is contravariantly finite.
\end{corollary}

\begin{corollary}
ind$_{C_{n}^{!}}$is functorialy finite.

\begin{proof}
It is clear from the duality D(ind$_{C_{n}^{!}}$)$\cong $ind$_{C_{n}^{op!}}$
\end{proof}
\end{corollary}

Observe (ind$_{C_{n}^{!}}$)$^{\perp }\cong \mathcal{F}^{\prime }$ however ind%
$_{C_{n}^{!}}$ is not necessary closed under extensions and we can not
conclude $\mathcal{F}^{\prime }$ contravariantly finite.

For the notions of contravariantly finite, covariantly finite and
functorialy finite, we refer to [AuS].

\newpage

\begin{center}
{\LARGE References}
\end{center}

[AS] Artin M. Schelter W. Graded algebras of dimension 4, Advances in Math.
66 (1987), 172-216.

[AuS] Auslander, M., Smalo, S.Almost split sequences in subcategories. J.
Algebra 69 (1981), no. 2, 426--454.

\ [AZ] Artin, M.; Zhang, J. J. Noncommutative projective schemes. Adv. Math.
109 (1994), no. 2, 228-287.

[Co] Coutinho S.C. A Primer of Algebraic D-modules, London Mathematical
Society, Students Texts 33, 1995

[Da] Dade, Everett C. Group-graded rings and modules. Math. Z. 174 (1980),
no. 3, 241--262.

[Ga] Gabriel, Pierre Des cat\'{e}gories ab\'{e}liennes. Bull. Soc. Math.
France 90 1962 323--448.

[GM1] Green, E. L.; Mart\'{\i}nez Villa, R.; Koszul and Yoneda algebras.
Representation theory of algebras (Cocoyoc, 1994), 247--297, CMS Conf.
Proc., 18, Amer. Math. Soc., Providence, RI, 1996.

[GM2] Green, E. L.; Mart\'{\i}nez-Villa, R.; Koszul and Yoneda algebras. II.
Algebras and modules, II (Geiranger, 1996), 227--244, CMS Conf. Proc., 24,
Amer. Math. Soc., Providence, RI, 1998.

[Le] Levandovskyy V. : On Gr\"{o}bner bases for noncommutative G-algebras .
In: Proceedings of the 8th Rhine Workshop on Computer Algebra, Mannheim,
Germany. (2002)

[M] Martinez-Villa, R. Graded, Selfinjective, and Koszul Algebras, J.
Algebra 215, 34-72 1999

[M2] Mart\'{\i}nez-Villa, R. Serre duality for generalized Auslander regular
algebras. Trends in the representation theory of finite-dimensional algebras
(Seattle, WA, 1997), 237--263, Contemp. Math., 229, Amer. Math. Soc.,
Providence, RI, 1998.

[Mac] Maclane S. Homology, Band 114, Springer, 1975.

[MM] Mart\'{\i}nez-Villa, R., Martsinkovsky; A. Stable Projective Homotopy
Theory of Modules, Tails, and Koszul Duality, Comm. Algebra 38 (2010), no.
10, 3941--3973.

[MMo] Homogeneous G-algebras I, preprint, UNAM\ 2014.

[MS] Mart\'{\i}nez Villa, Roberto; Saor\'{\i}n, M.; Koszul equivalences and
dualities. Pacific J. Math. 214 (2004), no. 2, 359--378.

[MZ] Mart\'{\i}nez-Villa, Roberto; Zacharia, Dan; Approximations with
modules having linear resolutions. J. Algebra 266 (2003), no. 2, 671--697.

[Mil] Mili\v{c}i\'{c}, D. Lectures on Algebraic Theory of D-modules.
University of Utah 1986.

[Mi] Miyachi, Jun-Ichi; Derived Categories with Applications to
Representations of Algebras, Chiba Lectures, 2002.

[Mo] Mondragon J. Sobre el algebra de Weyl homgenizada, tesis doctoral (in
process)

[P] Popescu N.; Abelian Categories with Applications to Rings and Modules;,
L.M.S. Monographs 3, Academic Press 1973.

[Sm] Smith, P.S.; Some finite dimensional algebras related to elliptic
curves, Rep. Theory of Algebras and Related Topics, CMS Conference
Proceedings, Vol. 19, 315-348, Amer. Math. Soc 1996.

[Y] Yamagata, K.; Frobenius algebras. Handbook of algebra, Vol. 1, 841--887,
North-Holland, Amsterdam, 1996.

\end{document}